\documentclass[11pt, a4paper]{scrartcl}

\usepackage[utf8]{inputenc}
\usepackage[T1]{fontenc}
\usepackage[english]{babel}
\usepackage{verbatim}
\usepackage{abstract}
\usepackage{url}
\usepackage{caption}
\usepackage{appendix}

\usepackage[nottoc]{tocbibind}
\usepackage[hidelinks]{hyperref}
\usepackage{aliascnt}

\usepackage[pdftex]{graphicx}
\usepackage{latexsym}
\usepackage{amsmath,amssymb,amsthm, mathtools,tikz}
\usepackage{stmaryrd}
\usetikzlibrary{cd}
\usepackage{pifont}
\usepackage[a4paper, left=2.7cm, right=2.7cm, bottom=3.2cm, top=3cm]{geometry}
\usepackage{makecell}
\usepackage{slashed}

\usepackage{float}

\DeclareMathOperator{\vol}{vol}

\newcommand{\Hom}{\textup{Hom}}
\DeclareMathOperator{\End}{End}
\newcommand{\Ad}{\textup{Ad}}

\newcommand{\pr}{\textup{pr}}

\newcommand{\Fr}{\textup{Fr}}
\newcommand{\ob}{\textup{ob}}

\newcommand{\SU}{\textup{SU}}
\newcommand{\U}{\textup{U}}
\newcommand{\GL}{\textup{GL}}
\newcommand{\Spin}{\textup{Spin}}

\newcommand{\Stab}{\textup{Stab}}
\newcommand{\coker}{\textup{coker}}
\newcommand{\coimage}{\textup{coimage}}
\newcommand{\image}{\textup{image}}
\newcommand{\e}{\textup{e}}
\newcommand{\id}{\textup{Id}}

\renewcommand{\Im}{\textup{Im}\thinspace}
\renewcommand{\Re}{\textup{Re}\thinspace}
\newcommand{\col}{\colon \thinspace}
\newcommand\diff{\mathop{}\!\textup{d}}
\newcommand\Diff{\textup{D}}

\newcommand{\VertC}[2]{\Vert_{C^{{#1}}{#2}}}
\newcommand{\VertWC}[3]{\Vert_{C^{{#1}}_{{#2}}{#3}}}
\newcommand{\vertWC}[3]{\vert_{C^{{#1}}_{{#2}}{#3}}}

\newcommand{\VertWH}[3]{\Vert_{C^{{#1},\alpha}_{{#2}}{#3}}}
\newcommand{\WsH}[3]{]_{C^{{#1},\alpha}_{{#2}}{#3}}}

\title{The moduli space of conically singular instantons over an SU(3)-manifold}

\author{Dominik Gutwein \and Yuanqi Wang}

\date{\today}

\numberwithin{equation}{section}

\newcommand{\mynewtheorem}[2]{
  \newaliascnt{#1}{dummycounter}
  \newtheorem{#1}[#1]{#2}
  \aliascntresetthe{#1}
  \expandafter\def\csname #1autorefname\endcsname{#2}
}

\theoremstyle{plain}
\mynewtheorem{theorem}{Theorem}
\mynewtheorem{proposition}{Proposition}
\mynewtheorem{lemma}{Lemma}
\mynewtheorem{corollary}{Corollary}
\mynewtheorem{observation}{Observation}
\mynewtheorem{introtheorem}{Theorem}

\mynewtheorem{introproposition}{Proposition}

\theoremstyle{definition}
\mynewtheorem{definition}{Definition}
\mynewtheorem{assumption}{Assumption}
\mynewtheorem{notation}{Notation}
\mynewtheorem{example}{Example}

\theoremstyle{remark}
\mynewtheorem{remark}{Remark}
\newtheorem*{introremark}{Remark}

\addto\extrasenglish{

}

\begin{document}
\maketitle

\begin{abstract}
In this article we study the moduli space of conically singular instantons (or Hermitian Yang--Mills connections) with prescribed tangent connections over a 6-manifold equipped with an $\SU(3)$-structure. That is, we develop a Fredholm deformation theory for such $\SU(3)$-instantons in which we fix the tangent connection but allow the underlying principal bundle (and, in particular, the singular set) to vary. This leads to the existence of a Kuranishi structure for this moduli space. Moreover, we investigate the cokernel of the instanton deformation operator and give under certain assumptions a formula for its dimension. Ultimately, we apply our results to conically singular instantons with structure group $\mathbb{P}\U(n)$ and give a formula for the virtual dimension of their moduli space in terms of sheaf cohomology of certain vector bundles over $\mathbb{P}^2$.
\end{abstract}

\section{Introduction}

An $\SU(3)$-manifold is a (real) 6-manifold $Z$ equipped with a symplectic and a holomorphic volume form $(\omega,\Omega)$, which are point-wise modelled upon the standard Kähler and holomorphic volume forms over $\mathbb{C}^3$ (cf. \autoref{def: SU(3) structure}). An instanton (called $\SU(3)$-instanton in this article) over such an $\SU(3)$-manifold $(Z,\omega,\Omega)$ is a connection $A$ on a principal $G$-bundle whose curvature satisfies \[ \Lambda_\omega F_A = 0 \quad \textup{and} \quad F_A \wedge \Im \Omega = 0, \] where $\Lambda_\omega$ denotes the dual Lefschetz operator. Note that the holomorphic volume form $\Omega$ induces an almost complex structure on $Z$ and that the $\SU(3)$-instanton equation is simply the Hermitian Yang--Mills condition $\Lambda_\omega F_A = 0$ and $F_A^{0,2} =0$ (where we have tacitly assumed that the structure group $G$ of the principal bundle is real). 

In their highly influential article \cite{DonaldsonThomas-higherdimensionalGaugeTheory} Donaldson and Thomas initiated a program aimed to develop gauge theoretic invariants for special holonomy spaces in dimension 6, 7, and 8 that (formally) mirror the familiar Casson--Floer picture in dimension 2, 3, and 4 (see also \cite{DonaldsonSegal-higherGaugeTheory}). These proposed invariants are based on the moduli space $\mathcal{M}\equiv \mathcal{M}_{\SU(3)/\textup{G}_2/\Spin(7)}$ of instantons over the respective spaces (including $\SU(3)$-instantons in dimension 6).  However, a rigorous definition of such invariants (even a complete understanding of the list of its ingredients) is met with great analytic difficulties due to various non-compactness phenomena related to the instanton equation (cf. \cite{Tian-gaugetheory_calibrated}, \cite{DonaldsonSegal-higherGaugeTheory}, \cite{Haydys-instantons_and_SW}, \cite{DoanWalpuski--CountingAssociatives}, \cite{DoanWalpuski-ExistenceZ2Spinors}). 

In dimension 6, when $(Z,\omega,\Omega)$ is Calabi--Yau (i.e. $\omega$ and $\Omega$ are both closed and therefore parallel), the Donaldson--Uhlenbeck--Yau Theorem identifies the moduli space of $\SU(3)$-instantons with the moduli space of (slope (poly-) stable) locally free sheaves over $Z$ and algebraic geometry may be used to compactify this space (cf. \cite[Chapter~3]{Thomas-CalabiYauGaugeTheoryThesis}). In this situation, a holomorphic Casson invariant can be defined and is now known as the Donaldson--Thomas invariant in algebraic geometry (cf. \cite{Thomas-CalabiYauGaugeTheoryThesis} and \cite{Thomas-holomorphicCassonInvariant}). However, when $(\omega,\Omega)$ is non-integrable (or if one considers instantons over 7-dimensional $\textup{G}_2$- or 8-dimensional $\Spin(7)$-manifolds) these methods do not apply and one needs to construct a compactification $\overline{\mathcal{M}}$ of the instanton moduli space geometric-analytically. The works of Uhlenbeck~\cite{Uhlenbeck-ConnectionswithLpbounds}, Price~\cite{Price-monotonicityformula}, Nakajima~\cite{Nakajima-compactness_higherYangMills}, and Tian~\cite{Tian-gaugetheory_calibrated} (see also \cite{Riviere-variationsofYMLagrangian} for a summary and \cite{ChenWentworth-OmegaYangMills} for an extension of Tian's results) show that a sequence $([A_n])_{n\in \mathbb{N}} \subset \mathcal{M}$ converges to another instanton $[A]$ (in a $C^\infty_{\textup{loc}}$-sense outside the so-called blow-up locus) with the following two phenomena possibly occurring:

\begin{enumerate}
\item Some of the energy of the instantons $([A_n])_{n \in \mathbb{N}}\subset \mathcal{M}$ may get lost (in the limit) due to ASD-instantons bubbling of transverse to a (possibly singular) calibrated submanifold.
\item The limiting instanton $[A]$ may only be defined outside of a subset $S\subset Z$ and may not be extendible over $S$. That is, the limiting instanton may have non-removable singularities.
\end{enumerate}

The previous result shows that a compactification of $\mathcal{M}$ needs to contain (amongst other data) the moduli space of singular instantons $\mathcal{M}_{\textup{sing}} \subset \partial \overline{\mathcal{M}}$ as part of the boundary. However, little is currently known about such a space in general. In this paper, we construct and study the moduli space of such singular instantons, in the case where the singularities are isolated and of conical nature (and where the tangent connection has been fixed). 

\textbf{Conically singular instantons:} The simplest(non-trivial) singular set $S$ for an instanton $A$ consists of a finite number of isolated points $S = \{s_1,\dots,s_N\}$. Around any point $s_i \in S$ we may use a coordinate system $\Upsilon_i \col B_R(0) \to Z$ centred at $s_i$ to regard $\Upsilon_i^*A$ as a connection over $B_R(0)\setminus \{0\}$. Rescaling this connection via a sequence $r_n \to 0$, one obtains a sequence of connections $\delta_{r_n}^*(\Upsilon_i^*A)$ that converges (outside of the blow-up locus and up to taking subsequences and gauge transformations) to a radially invariant $\SU(3)$-instanton $A_{s_i}$ over $\mathbb{C}^3\setminus S_0$ (cf. \cite[Discussion prior to Lemma~5.3.1]{Tian-gaugetheory_calibrated}). The limit $A_{s_i}$ is called a tangent connection or tangent cone for $A$ at $s_i\in S$. Note, that $A_{s_i}$ may not be unique even up to gauge transformations. We call $A$ conically singular if for each $s_i\in S$, we may find such a tangent connection $A_{s_i}$ that is defined over $\mathbb{C}^3\setminus \{0\}$ (i.e. $S_0 = \{0\}$ in the notation above) and if the convergence of $\Upsilon_i^*A$ to $A_{s_i}$ occurs at a polynomial rate, that is \[ \big\vert \nabla^k \big( \Upsilon_i^*A - A_{s_i}\big)\big\vert = \mathcal{O}(r^{\mu_i-k}) \quad \textup{for every $k\in \mathbb{N}_0$ as $r\to 0$} \] for some $\mu_i>-1$. Note that due to the polynomial rate of convergence, the tangent connections $A_{s_i}$ for a conically singular connection $A$ are, in fact, unique up to gauge. (See also \cite{Yang-TangentConesUniqueness}, \cite{AdamSaEarpWalpuski-tangent-cones-of-HYM-connections}, and \cite{CaniatoParise-TangentConesUniqueness} for conditions on $A$ such that it is conically singular around $s_i\in S$. Moreover, \cite{ChenSun-tangents_of_HYM_and_refl_sheaves} gives a complete algebraic geometric characterisation of the analytic tangent of a general singular Hermitian Yang--Mills connection over a Kähler manifold.)

\textbf{The moduli space of conically singular instantons with prescribed tangent connections:} In this article we study the moduli space of conically singular instantons over an $\SU(3)$-manifold $(Z,\omega,\Omega)$. We hereby allow for varying underlying principal $G$-bundles and, in particular, varying singular sets of the instantons. However, for simplicity we fix a set of radially invariant connections over $\mathbb{C}^3\setminus \{0\}$ that we prescribe as the tangent connections of the singular instantons. That is, we pre-fix the model cone-connections that we wish to exhibit at each singular point. A 'full' moduli theory should of course allow for varying tangent cones and we have added remarks on our expectations regarding the generalisations of our results to such a comprehensive moduli theory (cf. \autoref{rem: alternative interpretation of rotations}, \autoref{rem: overcoming obstructions via deformations of tangent cone}, and \autoref{rem: only Fubini--Study connections appear generically}).

The main contribution of the article at hand is the development of a Fredholm deformation theory for conically singular $\SU(3)$-instantons (with prescribed tangent cones) with which we prove the existence of a Kuranishi structure for the aforementioned moduli space (see also \autoref{intro-thm: Kuranishi structure for moduli space} for a precise statement and the next paragraph for a detailed summary of our results). The difficulty when allowing varying underlying principal $G$-bundles is to show that the space of bundles up to suitable isomorphisms is finite dimensional. As an intermediate step we therefore introduce in \autoref{sec: moduli space of framed conically singular connections} the moduli space of \textit{framed} conically singular connections in which an (ungeometric) choice of framing has been added to the collected data and carefully investigate the deformations of the underlying framed bundle modulo isomorphisms. This study together with the well-known Fredholm theory of conically singular elliptic differential operators then leads to the deformation theory of conically singular instantons (with prescribed tangent connections).

Note that the methods used to develop said deformation theory are not specific to 6-manifolds with $\SU(3)$-structures but should apply to other instantons as well. In fact, \autoref{thm: local structure of unframed B Phi is homeo} together with the results proven in \cite[Chapter~1]{SoleFarre-thesis} should immediately give rise to a Kuranishi structure for the moduli spaces of conically singular $\textup{G}_2$- and $\Spin(7)$-instantons with prescribed tangent connections.

\textbf{Summary and statement of results:} We begin in \autoref{sec: SU(3) structures and conically instantons} by reviewing the necessary background on $\SU(3)$-structures and dilation-invariant $\SU(3)$-instantons over $\mathbb{C}^3\setminus \{0\}$ which subsequently serve as singularity models for more general conically singular instantons. In \autoref{sec: conically singular connections} we give the definition of (framed) conically singular connections over a fixed bundle and prove in \autoref{prop: change of framing} that the set of compatible framings for such a conically singular connection is (essentially) a torsor over the compact Lie group given in~\eqref{equ: definition Stab_SU(3)(A_s)}. Note that the $\SU(3)$-instanton equation is a priori overdetermined. We therefore prove in \autoref{prop: augmented instanton equation} that whenever the $\SU(3)$-structure $(\omega,\Omega)$ on $Z$ satisfies $\diff^*\omega=0$ and $\diff \Omega = w_1 \omega^2$ for some $w_1\in \mathbb{R}$, then the $\SU(3)$-instanton equation can be augmented to an elliptic system (modulo gauge equivalence).

We subsequently begin our study of the moduli space of conically singular $\SU(3)$-instantons. For this, we define in \autoref{sec: moduli space of framed conically singular connections} the moduli space of framed conically singular connections and describe its local structure. In \autoref{sec: space of (unframed) cs connections} these results are then extended to the space of (unframed) conically singular connections. In \autoref{sec: moduli space of cs SU(3)-instantons} we finally define the moduli space of conically singular $\SU(3)$-instantons and use our local description of the moduli space of conically singular connections together with the well-known Fredholm theory of conically singular elliptic differential operators to show that this instanton moduli space admits a Kuranishi structure. More precisely, we prove:

\begin{introtheorem}[cf. \autoref{thm: Kuranishi structure} and \autoref{thm: relating different rates}]\label{intro-thm: Kuranishi structure for moduli space}
Assume that the $\SU(3)$-structure $(\omega,\Omega)$ on $Z$ satisfies $\diff^*\omega=0$ and $\diff \Omega = w_1 \omega^2$ for some $w_1\in \mathbb{R}$ (so that the $\SU(3)$-instanton equation can be augmented to an elliptic system) and let $G$ be a compact Lie group with finite center. Moreover, fix $N\in \mathbb{N}$ and for every $i=1,\dots,N$ a principal $G$-bundle $\pi_i \col P_i \to S^5$ together with an irreducible connection\footnote{In fact, it suffices that all $A_i$ are infinitesimally irreducible, i.e. the only elements $\xi_i\in \Omega^0(S^5,\mathfrak{g}_{P_i})$ satisfying $\diff_{A_i} \xi_i = 0$ are $\xi_i=0$. However, then one needs to include the (discrete) stabiliser groups of the conically singular instanton into the statement (cf. \autoref{thm: Kuranishi structure}).} $A_i \in \mathcal{A}(P_i)$ satisfying~\eqref{equ: cone reduction of instanton equation}. For $\mu \in (-1,\Bar{\mu}_1)\times \dots \times (-1,\Bar{\mu}_N)$, where $\bar{\mu}_i \coloneqq \min \{((-1,0)\cap \mathcal{D}(L_{A_i})) \cup \{0\}\}$ (with $\mathcal{D}(L_{A_i})$ as in \autoref{def: critical rates}), let $\mathcal{M}_\mu(\{P_i,A_i\})$ be the moduli space of conically singular $\SU(3)$-instantons with $N$ singularities and prescribed tangent connections $\{(P_i,A_i)_{i=1,\dots,N}\}$ as defined and topologised in \autoref{sec: moduli space and its topology}. 
\begin{enumerate}
\item For every $[\mathbb{A}] \in \mathcal{M}_\mu(\{P_i,A_i\})$ there exist two finite dimensional vector spaces $W_1,W_2$ and a smooth map $\ob_{\mathbb{A}} \col W_1 \to W_2$ with $\ob_{\mathbb{A}}(0)=0$, such that a neighbourhood of $[\mathbb{A}]$ in $\mathcal{M}_\mu(\{P_i,A_i\})$ is homeomorphic to a neighbourhood of $0$ in $\ob_{\mathbb{A}}^{-1}(0)$.
\item For $W_1$ and $W_2$ as in the previous point we have 
\begin{align*}
\textup{virt-dim}\big(\mathcal{M}_\mu(\{P_i,A_i\})\big) &\coloneqq \dim W_1 - \dim W_2 \\
&= \sum_{i=1}^N 6 + (8 - \dim(\Stab_{\SU(3)}(A_i))) - \hspace*{-20pt} \sum_{\scriptscriptstyle \nu_i \in \mathcal{D}(L_{A_i}) \cap (-5/2,\mu_i)} \hspace*{-20pt} \dim \mathcal{K}(L_{A_i})_{\nu_i}
\end{align*}
where $\Stab_{\SU(3)}(A_i)$ is defined in \eqref{equ: definition Stab_SU(3)(A_s)} and $\mathcal{K}(L_{A_i})_{\nu_i}$ in \eqref{equ: definition homogeneous kernel}.
\item The moduli space $\mathcal{M}_\mu(\{P_i,A_i\})$ is homeomorphic to $\mathcal{M}_{\mu^\prime}(\{P_i,A_i\})$ for any other rate $\mu^\prime \in (-1,\Bar{\mu}_1)\times \dots \times (-1,\Bar{\mu}_N)$.
\end{enumerate}
\end{introtheorem}

\begin{introremark}
The assumption that $G$ has a finite center and that the tangent connections are (infinitesimally) irreducible makes the presentation in \autoref{sub: Coulomb gauge} simpler. However, \cite[Chapter~I.5]{SoleFarre-thesis} shows that these conditions can be removed (see also \autoref{rem: Coulomb gauge and slice for groups with positive center} and \autoref{rem: Kuranishi chart for groups with positive center}).
\end{introremark}

\begin{introremark}
The assumption $\diff^* \omega=0$ and $\diff \Omega = w_1 \omega^2$ for some $w_1\in \mathbb{R}$ on the $\SU(3)$-structure is used in the previous theorem to argue that the instanton equation may be augmented by introducing two further unknowns by $\xi_1,\xi_2 \in \Omega^0(Z,\mathfrak{g}_P)$ to the elliptic (modulo gauge) system 
\begin{align}\label{intro-equ: augmented instanton equation}
\Lambda_\omega F_A = 0 \quad \textup{and} \quad *(F_A \wedge \Im \Omega) + \diff_A \xi_1 + J^*(\diff_A \xi_2)=0. 
\end{align} 
That is, if $(A,\xi_1,\xi_2)$ solves~\eqref{intro-equ: augmented instanton equation}, then $\diff_A\xi_i=0$ for $i=1,2$ and $A$ is an $\SU(3)$-instanton, i.e. $A$ solves the non-augmented equation (cf. \autoref{prop: augmented instanton equation}). If the assumption on $(\omega,\Omega)$ is dropped, then $\diff_A \xi_i = 0$ does not need to hold for solutions of \eqref{intro-equ: augmented instanton equation} anymore, and the augmented equation becomes an honest equation for $(A,\xi_1,\xi_2)$. If one is willing to accept~\eqref{intro-equ: augmented instanton equation} as an equation for $(A,\xi_1,\xi_2)$, then the Fredholm deformation theory (in particular, the virtual-dimension formula) discussed in the previous theorem still applies. Note that Equation~\eqref{intro-equ: augmented instanton equation} could give rise to a symplectic approach to Donaldson--Thomas invariants (cf. \cite[Discussion at the beginning of Section~3]{Thomas-holomorphicCassonInvariant}, \cite{Tanaka-symplecticDT}, and \cite{BallOliveira-almostHermitian-DT}).
\end{introremark}

In \autoref{sec: obstruction space} we then study the obstruction space of a conically singular instanton, that is, the cokernel of its (full) deformation operator. More precisely, we define in \autoref{prop: obstruction pairing} a pairing between the cokernel of the instanton deformation operator over a fixed bundle and the deformation space of the underlying bundle. We then explain in \autoref{thm: non-degenerateness of pairing} how (under certain assumptions) some of the obstructions arising from the deformation problem over a fixed bundle may be overcome by deforming the underlying bundle. This is used in \autoref{cor: estimate dimension cokernel} to show that under the same assumptions the (non-positive) virtual dimension of $\mathcal{M}_\mu(\{P_i,A_i\})$ given in \autoref{intro-thm: Kuranishi structure for moduli space} is precisely the negative of the dimension of the obstruction space.

Finally, in \autoref{sec: structure group PU(n)} we consider instantons with structure group $G=\mathbb{P}\U(n)$ and use the results of the second named authors in \cite{Wang-AtiyahClasses} and \cite{Wang-spectrum_of_operator_for_instantons} to prove the following:

\begin{introtheorem}[cf. \autoref{cor: virtual dimension for structure group PU(n)} and \autoref{thm: non-positive dimension for PU(n) connections}]\label{introthm: moduli space for PU(n)-connections}
Assume we are in the set-up of the previous theorem where now $G= \mathbb{P}\U(n)$ (with $n>1$). The discussion prior to \autoref{cor: virtual dimension for structure group PU(n)} associates to each $(\pi_i \col P_i \to S^5,A_i)$ a holomorphic vector bundle $E_i \to \mathbb{P}^2$ over $\mathbb{P}^2 = S^5/\U(1)$. The virtual dimension of the moduli space $\mathcal{M}_\mu(\{P_i,A_i\})$ is then given by 
\begin{align*}
\textup{virt-dim}\big(\mathcal{M}_\mu(\{P_i,A_i\})\big) 
= \sum_{i=1}^N 6 &+ (8 - \dim(\Stab_{\SU(3)}(A_i))) - 2\mathrm{h}^{1}(\mathbb{P}^{2}, \End E_i) \\
&-2\mathrm{h}^{1}(\mathbb{P}^{2}, (\End E_i)(-1)) 
\end{align*}
where $\mathrm{h}^1(\mathbb{P}^2,F) \coloneqq \dim_{\mathbb{C}}(\mathrm{H}^{0,1}(\mathbb{P}^2,F))$ for any holomorphic vector bundle $F\to \mathbb{P}^2$. Because of \cite[Proposition~4.1]{Wang-AtiyahClasses} this implies \[\textup{virt-dim}\big(\mathcal{M}_\mu(\{P_i,A_i\})\big) \leq 0 \] with equality if and only if $G= \mathbb{P}\U(2)$ and all tangent connections $(\pi_i \col P_i \to S^5,A_i)$ are isomorphic to the pull back of the Fubini--Study connection $(\mathbb{P}\U(T\mathbb{P}^2,h_{\textup{FS}}) \to \mathbb{P}^2 ,A_{\textup{FS}})$.
\end{introtheorem}

\begin{introremark}
If $(Z,\omega,\Omega)$ is Calabi--Yau, then (singular) instantons with structure group $\U(n)$ correspond to slope (poly-) stable reflexive sheaves over $Z$ (cf. \cite[Theorem~3]{BandoSiu-PHYM-over-reflexive-sheaves}). Moreover, \cite[Corollary~10]{Vermeire-Moduli_of_reflexive_sheaves} proves that the expected dimension of the moduli space of reflexive sheaves over $Z$ with fixed Chern-class is zero. Thus, our virtual dimension for the moduli space of conically singular instantons\footnote{note that while the previous theorem is stated for instantons with structure group $\mathbb{P}\U(n)$, one can show that the same virtual dimension formula also hold for instantons with structure group $\U(n)$ (see also the remark after \autoref{intro-thm: Kuranishi structure for moduli space})} whose tangent connections are all modelled on the Fubini--Study connection over $T\mathbb{P}^2$ agrees with the one given in \cite{Vermeire-Moduli_of_reflexive_sheaves}. For singular instantons whose tangents are not all modelled on the Fubini--Study connection, our virtual dimension is strictly negative (even after taking the deformations of the tangent connection into account; cf. \autoref{rem: only Fubini--Study connections appear generically}). Of course, we only restrict to deformations that preserve the singularity, whereas \cite{Vermeire-Moduli_of_reflexive_sheaves} considers deformations as general reflexive sheaves. It appears to us that some of the deformations in \cite{Vermeire-Moduli_of_reflexive_sheaves} could possibly 'smooth out' the corresponding singular instanton to a (degenerating family) of non-singular instantons.
\end{introremark}

\begin{introremark}
With regard to the previous remark (or the observation that the expected dimension of the moduli space of \textit{smooth} $\SU(3)$-instantons is always zero) it seems interesting to investigate if and how conically singular instantons whose tangents are modelled on the Fubini--Study connection on $\mathbb{P}\U(T\mathbb{P}^2,h_{\textup{FS}})$ can appear as the limit of smooth instantons (possibly with higher instanton number).
\end{introremark}

We end this article by collecting in \autoref{app-sec: analytic preliminaries} numerous well-known facts about conically singular elliptic differential operators used throughout this article.

\textbf{Comparison to previous results:} Prior to our article, the deformation problem of conically singular instantons has been considered in \cite{Wang-CSG2Instantons} and \cite{SoleFarre-thesis} for instantons over $\textup{G}_2$- and $\Spin(7)$-manifolds. Note that in contrast to the article at hand both references fixed the underlying bundle and the tangent cones of the singular instantons. The results obtained in the present paper are therefore an extension of the work in \cite{Wang-CSG2Instantons} and \cite{SoleFarre-thesis}. (In fact -- as already mentioned above -- \autoref{thm: local structure of unframed B Phi is homeo} together with the results in \cite[Chapter~1]{SoleFarre-thesis} should give rise to the analogue of \autoref{intro-thm: Kuranishi structure for moduli space} for the moduli spaces of conically singular $\textup{G}_2$- and $\Spin(7)$-instantons.) Moreover, the results in \autoref{introthm: moduli space for PU(n)-connections} build on the work \cite{Wang-spectrum_of_operator_for_instantons} and \cite{Wang-AtiyahClasses} of the second named author on the model operator for dilation invariant instantons over $\mathbb{C}^3\setminus \{0\}$ arising as pullbacks from $\mathbb{P}^2$.

The moduli theory of conically singular calibrated submanifolds, on the other hand, goes back to the work of Joyce \cite{Joyce-Moduli_of_cs-slag} and has by now been developed for all classes of calibrated submanifolds appearing naturally inside exceptional holonomy manifolds (cf. \cite{Lotay-cs_coassociatives}, \cite{Englebert-cs_cayleys}, and \cite{Bera-cs_associatives}). In fact, many of our results and definitions are inspired by their respective analogues for conically singular submanifolds. Note, however, that when working with conically singular connections the equivalence relation posed by bundle isomorphisms (compatible with the singular structure) introduces an additional difficulty not present in the deformation theory of submanifolds.

\subsection{Acknowledgments}

The work for this article was initiated while the authors were in residence at the Simons Laufer Mathematical Sciences Institute (formerly MSRI) in Berkeley, California, during the Fall 2024 semester and is supported by the National Science Foundation under Grant No. DMS-1928930. The authors would like to thank the SLMath for its hospitality and for creating such a vibrant research environment. D.G. would also like to express his gratitude towards Gorapada Bera, Lorenzo Foscolo, Thibault Langlais, Jason Lotay, Viktor Majewski, Jacek Rzemieniecki, Enric Solé-Farré, and Thomas Walpuski for various helpful discussions related to this article. Moreover, D.G. is supported by the Deutsche Forschungsgemeinschaft (DFG, German Research Foundation) under SFB-Geschäftszeichen 1624 – Projektnummer 506632645.

\section{SU(3)-structures and conically singular SU(3)-instantons}
\label{sec: SU(3) structures and conically instantons}

In this section we first review the necessary background on $\SU(3)$-structures and set our conventions. We then discuss dilation-invariant $\SU(3)$-instantons over $\mathbb{C}^3\setminus \{0\}$ which serve as singularity models. In the final subsection we treat singular connections whose singularities are of conical nature.

\subsection{SU(3)-structures}

Throughout this article we consider $\mathbb{C}^{3}$ with coordinates $z_{\alpha}=x_{\alpha}+ix_{\alpha+1}$ for  $\alpha=1,3,5$ together with its standard Kähler and holomorphic volume forms 
\begin{equation}\label{equ standard C3} 
\omega_0=\tfrac{i}{2}\big(\diff z_{1}\wedge \diff \overline{z}_{1}+\diff z_{2}\wedge \diff \overline{z}_{2}+\diff z_{3}\wedge \diff \overline{z}_{3}\big),\quad \Omega_0=\diff z_{1}\wedge \diff z_{2}\wedge \diff z_{3}.
\end{equation}
Identifying $\mathbb{C}^3 \cong \mathbb{R}^6$ these become
\begin{align*}
\omega_0&=\varepsilon^{12}+\varepsilon^{34}+\varepsilon^{56} \\
\Omega_0&= \big(\varepsilon^{135}-\varepsilon^{245}-\varepsilon^{236}-\varepsilon^{146} \big) + i \big( -\varepsilon^{246}+\varepsilon^{235}+\varepsilon^{145}+\varepsilon^{136} \big),
\end{align*}
where $\varepsilon^1,\dots, \varepsilon^6$ is the standard (dual) basis of $(\mathbb{R}^6)^*$ and $\varepsilon^{\alpha \beta \gamma} \coloneqq \varepsilon^{\alpha} \wedge \varepsilon^\beta \wedge \varepsilon^\gamma$.

\begin{definition}\label{def: SU(3) structure}
An $\SU(3)$-structure on a (real) $6$-manifold $Z$ is a tuple of differential forms 
\begin{equation}\label{equ tuple of forms}
(\omega,\Omega)\in \Omega^2(Z) \times \Omega^3(Z,\mathbb{C})
\end{equation} 
such that over every point $z  \in Z$, there exists a (real vector space-) isomorphism $f \col \mathbb{R}^6 \to T_zZ$ such that the pullback satisfies $f^{*}(\omega_z,\Omega_z)= (\omega_0, \Omega_0)$.  
\end{definition}
\begin{remark}
Since the stabiliser $\Stab_{\GL(\mathbb{R}^6)}(\omega_0,\Omega_0)$ of the pair $(\omega_0,\Omega_0)$ in $\GL(\mathbb{R}^6)$ is equal to $\SU(3)$, the previous definition is equivalent to an $\SU(3)$-reduction of the principal frame bundle $\mathcal{F}(TZ)$ (cf. \cite[Chapter~1]{Salamon-Holonomy-book}). Given a pair $(\omega,\Omega)$ as in the previous definition, then the corresponding $\SU(3)$-subbundle of $\mathcal{F}(TZ)$ is defined by \[\{(f \col \mathbb{R}^6 \to T_zZ) \in \mathcal{F}(TZ) \mid f^*(\omega_z,\Omega_z) = (\omega_0,\Omega_0) \}.\] In fact, the common stabiliser of $\omega_0$ and $\Re \Omega_0$ in $\GL(\mathbb{R}^6)$ is already $\SU(3)$ (cf. \cite[Proposition~3.1]{Banos-classification-of-3-forms} or \cite[Remark~31]{Bryant-geometry-of-almost-cpx-6-mfds}). We could have therefore equivalently defined an $\SU(3)$-structure to be a pair of differential forms $(\omega, \Psi) \in \Omega^2(Z)\times \Omega^3(Z)$ which are over each point linearly equivalent to $(\omega_0, \Re \Omega_0)$. Moreover, Banos \cite[Section~2 and~3]{Banos-classification-of-3-forms} (see also \cite[Appendix~A]{Bryant-geometry-of-almost-cpx-6-mfds}) proved that up to isomorphism there are only two pairs consisting of a symplectic form $\tilde{\omega} \in \Lambda^2(\mathbb{R}^6)^*$ and a complex volume form $\tilde{\Omega} \in \Lambda^3 (\mathbb{R}^6)^* \otimes \mathbb{C}$ (i.e. $\tilde{\Omega}$ is decomposable and $\tilde{\Omega}\wedge \overline{\tilde{\Omega}} \neq 0$) over $\mathbb{R}^6$ which satisfy 
\begin{equation} \label{equ: algebraic conditions for SU(3)-structure}
\tilde{\omega} \wedge \tilde{\Omega} = 0 \qquad \textup{and} \qquad \tfrac{1}{6}\tilde{\omega}^3 = \tfrac{i}{8} \tilde{\Omega}\wedge \overline{\tilde{\Omega}}. 
\end{equation}
One of these pairs results in an $\SU(3)$-structure and the other one in an $\SU(1,2)$-structure. Thus, an $\SU(3)$-structure on $Z$ is given by a symplectic form $\omega$ and a complex volume form $\Omega$ satisfying \eqref{equ: algebraic conditions for SU(3)-structure} such that the resulting inner product is positive definite.
\end{remark}

\begin{remark}\label{remark: almost complex structure associated to (omega,Omega)}
As a consequence of the previous remark, an $\SU(3)$-structure induces an almost complex structure $J$ and a compatible Riemannian metric $g$ on $Z$ such that $\omega(\cdot,\cdot) = g(J \cdot,\cdot).$ With respect to this $J$, the no-where vanishing 3-form $\Omega$ is of type $(3,0)$. 
\end{remark}

The following proposition decomposes the differentials $\diff \omega$ and $\diff \Omega$ into irreducible $\SU(3)$-representations and can be found in \cite{ChiossiSalamon-intrinsic-torsion-of-SU3} or \cite{LarforsLukasRuehle-CY-manifolds-and-SU3-structures}. Its proof follows immediately from the decomposition of differential forms into primitives (cf. \cite[Proposition~1.2.30]{Huybrechts-complex-Geometry}) and the identity $\Omega \wedge \omega = 0$.

\begin{proposition}[{cf. \cite[Section~1]{ChiossiSalamon-intrinsic-torsion-of-SU3} or \cite[Section~2.1]{LarforsLukasRuehle-CY-manifolds-and-SU3-structures}}] \label{prop: decomposition of intrinsic torsion}
Let $(\omega,\Omega)$ be an $\SU(3)$-structure on $Z$. Define $w_1 \in C^\infty(Z,\mathbb{C})$ and $w_4,w_5 \in \Omega^1(Z)$ as
\begin{align*}
w_1 \coloneqq -\tfrac{i}{6} \big\langle\! \diff \omega , \overline{\Omega} \big\rangle_{\mathbb{C}}, \quad  w_4 \coloneqq \tfrac{1}{2} \Lambda_\omega \diff \omega, \quad \textup{and} \quad w_5 \coloneqq -\tfrac{1}{2} i_{\Re \Omega} \diff (\Re \Omega)
\end{align*}
where $\Lambda_\omega$ denotes the dual Lefschetz operator. Then there exists a primitive form $w_2 \in \Omega^2(Z,\mathbb{C})$ (i.e. $\Lambda_\omega w_2 =0$) and a primitive form $w_3 \in \Omega^3(Z)$ that additionally satisfies $\langle \Omega ,w_3 \rangle_\mathbb{C}=0$ such that 
\begin{align*}
\diff \omega &=\tfrac{3i}{4}(w_{1}\overline{\Omega}-\overline{w}_{1}\Omega)+w_{4}\wedge \omega+w_{3},\\
\diff \Omega&=w_{1}\omega^{2}+w_{2}\wedge \omega+w_{5}\wedge \Omega. 
\end{align*}
\end{proposition}
\begin{remark}
\cite[Theorem~1.1]{ChiossiSalamon-intrinsic-torsion-of-SU3} identifies the tensors $w_1,\dots,w_5$ in the previous proposition with the five classes of the intrinsic torsion of the $\SU(3)$-structure.
\end{remark}
\begin{remark}\label{rmk: convenient class of SU(3)-structures}
In the following we will restrict to $\SU(3)$-structures for which $w_2=w_4=w_5=0$. This has the advantage that the overdetermined $\SU(3)$-instanton equation can be augmented to an elliptic system modulo gauge (cf. \autoref{prop: augmented instanton equation}). Note that $w_4=0$ is equivalent to $\diff (\omega^2) = 0$ and therefore also to $\diff^* \omega=0$. Thus, $w_2=w_4=w_5=0$ is equivalent to
\begin{align*}
\diff^* \omega &= 0 \\
\diff \Omega &= w_1 \omega^2.
\end{align*}  
\end{remark}
\begin{proposition}\label{prop: modifying the SU(3)-structure such that dIm Omega = 0}
Assume that $(\omega,\Omega)$ is an $\SU(3)$-structure on $Z$ with $w_2=w_4=w_5=0$. Then there exists a $\vartheta \in \U(1)$ such that $(\omega,\vartheta\cdot\Omega)$ is an $\SU(3)$-structure which satisfies $w_2=w_4=w_5=0$ and additionally $\diff \Im 
\Omega = 0$.
\end{proposition}
\begin{proof}
In the previous remark we have seen that $w_2=w_4=w_5=0$ is equivalent to
\begin{align*}
\diff \omega^2 &= 0 \\
\diff \Omega &= w_1 \omega^2.
\end{align*}
Applying the de Rham differential to the second equation and using the first, we find that $(\diff w_1) \wedge \omega^2 =0$ and therefore $\diff w_1 =0$ (cf. \cite[Proposition~1.2.30]{Huybrechts-complex-Geometry}). The complex valued function $w_1$ is constant and there exists a $\vartheta \in \U(1)$ such that $\vartheta \cdot w_1$ is real. The modified $\SU(3)$-structure $(\omega,\vartheta\cdot\Omega)$ now satisfies 
\begin{align*}
\diff \omega^2 &= 0 \\
\diff (\vartheta \cdot \Omega) &= \vartheta \cdot w_1 \omega^2 \in \Omega^4(Z,\mathbb{R})
\end{align*}
which finishes the proof.
\end{proof}

We end this section with a definition for a distinguished coordinate chart that will be useful throughout this article.

\begin{definition}\label{def: SU(3)-coordinate system}
An $\SU(3)$-coordinate chart centered at any $z \in Z$ is a diffeomorphism $\Upsilon \col B_R(0)\subset \mathbb{C}^3 \to U\subset Z$ for some $R>0$ such that $\Upsilon(0)=z$ and $\Upsilon^*(\omega_z,\Omega_z)=(\omega_0,\Omega_0)$ at $z$.     
\end{definition}

\subsection{Dilation-invariant SU(3)-instantons over $\mathbb{C}^3\setminus\{0\}$}

This section discusses dilation-invariant $\SU(3)$-instantons over $\mathbb{C}^3\setminus \{0\}$. It turns out that for dilation-invariant connections, the $\SU(3)$-instanton equation reduces to an equation over $S^5$ (and in certain cases further to an equation over $\mathbb{P}^2 \coloneqq S^5/\U(1)$). First, we therefore recall some facts about the induced Sasaki--Einstein structure on $S^5$.

\subsubsection{The canonical Sasaki--Einstein structure on $S^5$}
\label{subsec: Sasaki--Einstein structure on S5}

In this subsection we collect some well-known facts about the canonical Sasaki--Einstein structure on the unit sphere $S^5\subset \mathbb{C}^3$ needed for the second part of this section.
\begin{enumerate}
\item Let $X_\theta \in \Gamma(TS^5)$ be the infinitesimal generator of the canonical $\U(1)$-action on $S^5 \subset \mathbb{C}^3$. Furthermore, denote by $\theta \coloneqq g_{S^5}(X_\theta,\cdot)$ its dual 1-form, where $g_{S^5}$ is the standard metric on the unit sphere $S^5\subset \mathbb{C}^3$. In fact, $\theta$ is a contact 1-form, which defines the standard contact distribution $H \coloneqq \ker \theta$ on $S^5$, with $X_\theta$ being its associated Reeb vector field.
\item The contact distribution $H$ is invariant under the canonical complex structure $J$ on $\mathbb{C}^3$ and we therefore obtain an induced complex structure $J_1$ on $H$. Its associated Hermitian form $\omega_1 \coloneqq g_{S^5}(J_1\cdot,\cdot) \in \Gamma(\Lambda^2H^*)$ is equal to the restriction of the Kähler form $\omega_0 \in \Omega^2(\mathbb{C}^3)$ to $S^5$, that is, $\omega_1 = {\omega_0}_{\vert S^5}$. In fact, over $\mathbb{C}^3\setminus \{0\}$ we have 
\begin{equation}\label{equ: relation Kähler form on C^3 and omega1 on S5}
\omega_0 = r \diff r \wedge \pr_{S^5}^*\theta + r^2 \pr_{S^5}^* \omega_1 
\end{equation} 
where $\pr_{S^5} \col \mathbb{C}^3\setminus \{0\} \to S^5$ denotes the radial projection.
\item Let $\Lambda^{2,0}H^*_\mathbb{C} \to S^5$ be the bundle of (complexified) 2-covectors of $H$, which are of type $(2,0)$ with respect to $J_1$. The restriction of the canonical bundle $\Lambda^{3,0}T^*_\mathbb{C}\mathbb{C}^3$ to $S^5$ splits as \[ \Lambda^{3,0}T^*_{\mathbb{C}}\mathbb{C}^3_{\vert S^5} \cong \underline{\mathbb{C}} \otimes (\Lambda^{2,0}H^*_{\mathbb{C}})\] where the trivial bundle is generated by $\diff r + i \theta$. Therefore, there exist $\omega_2,\omega_3 \in \Gamma_{S^5}(\Lambda^2H^*))$ such that 
\begin{align*}
\langle \omega_2+i\omega_3\rangle_\mathbb{C} &= \Lambda^{2,0}H^*_\mathbb{C}  \\
\textup{and} \quad \Omega_0 &= r^2 (\diff r + i r \pr_{S^5}^* \theta ) \wedge (\pr_{S^5}^*\omega_2 + i \pr_{S^5}^*\omega_3). 
\end{align*}
\item The quadruple $(\theta,\omega_1,\omega_2,\omega_3)$ defines a special $\SU(2)$-structure on $S^5$ that can be regarded as a Sasaki--Einstein structure (cf. \cite[Section~3.1]{FoscoloHaskinsNordström-G2-from-acon-CY3} and the references therein).
\item \label{bul: relationship SE structure S5 and K structure on P2} The quotient $S^5/\U(1)$ is given by the complex projective space $\mathbb{P}^2$ and the contact 1-form $\theta$ defines a connection on the $\U(1)$-bundle $\pr_{\mathbb{P}^2} \col S^5 \to \mathbb{P}^2$ (a complex Hopf-bundle) with horizontal distribution $H$. The differential of the projection $\pr_{\mathbb{P}^2}$ restricts to a complex linear bundle isomorphism \[(\Diff\pr_{\mathbb{P}^2})_{\vert H} \col H \to \pr_{\mathbb{P}^2}^*T\mathbb{P}^2\] (where $H$ is equipped with $J_1$ and $T\mathbb{P}^2$ with its standard complex structure). This implies that the pullback of the canonical bundle $\pr_{\mathbb{P}^2}^*(\Lambda^{2,0}T^*_\mathbb{C}\mathbb{P}^2)$ is isomorphic to $\Lambda^{2,0}H^*_\mathbb{C}$. Moreover, one can verify that up to a positive factor (which we fix to be one) $\pr_{\mathbb{P}^2}^*\omega_{\textup{FS}} = \omega_1$, where $\omega_{\textup{FS}}$ denotes the Fubini--Study form on $\mathbb{P}^2$. (Alternatively, one may simply define $\omega_{\textup{FS}}$ by this equation.)
\end{enumerate}

\subsubsection{Dilation-invariant instantons over $\mathbb{C}^3\setminus \{0\}$}

In this section we consider $\mathbb{C}^3\setminus\{0\}$ equipped with its canonical flat (Calabi--Yau) $\SU(3)$-structure~\eqref{equ standard C3}. Let $\pi \col P \to \mathbb{C}^3\setminus\{0\}$ be a principal $G$-bundle where $G$ is a compact Lie group whose Lie algebra $\mathfrak{g}$ has been equipped with an $\Ad$-invariant inner product. We denote by $\mathcal{A}(P)$ the set of connections on $P$.

\begin{definition}
A connection $A \in \mathcal{A}(P)$ is called $\SU(3)$-instanton if it satisfies 
\begin{equation} \label{equ: SU(3)-instanton}
    \Lambda_{\omega_0} F_A =0 \quad \textup{and} \quad F_A \wedge \Im \Omega_0 = 0 
\end{equation} 
where $\Lambda_{\omega_0}$ is the dual Lefschetz operator associated to $\omega_0$.
\end{definition}

\begin{remark}
In the following we tacitly assume that $G$ is a real Lie group. The second condition in the previous definition is then equivalent to $F_A^{0,2}=0$, where we have complexified the adjoint bundle $\mathfrak{g}_P$ in order to project to the $(0,2)$-component.
\end{remark}

We now restrict to bundles and connections which are pulled back from the unit sphere $S^5\subset \mathbb{C}^3\setminus\{0\}$. For this recall the canonical Sasaki--Einstein structure $(\theta,\omega_1,\omega_2,\omega_3)$ on $S^5$ discussed in \autoref{subsec: Sasaki--Einstein structure on S5}.
\begin{proposition}\label{prop: dilation invariant instantons}
Assume that $\pr_{S^5}^*\pi \col \pr_{S^5}^*P \to \mathbb{C}^3\setminus \{0\}$ is the pullback of a bundle $\pi \col P \to S^5$. The pullback $\pr_{S^5}^*A$ of a connection $A \in \mathcal{A}(P)$ is an $\SU(3)$-instanton over $\mathbb{C}^3\setminus\{0\}$ if and only if $A$ satisfies 
\begin{align} \label{equ: cone reduction of instanton equation}
F_A \wedge \omega_i = 0 \textup{ for $i=1,2,3$.}
\end{align}
Note that if the curvature satisfies $F_A \wedge \omega_i =0$ for some $i=1,2,3$, then $i_{X_\theta} F_A =0$, where $X_\theta \in \Gamma(TS^5)$ denotes the Reeb vector field associated to $\theta$.
\end{proposition}

\begin{proof}
The curvature of $\pr_{S^5}^*A$ over $\mathbb{C}^3\setminus \{0\}$ satisfies $F_{\pr_{S^5}^*A}= \pr_{S^5}^*F_A$. By \cite[Proposition~1.2.30]{Huybrechts-complex-Geometry} and \eqref{equ: relation Kähler form on C^3 and omega1 on S5} we therefore obtain that $\Lambda_{\omega_0} F_{\pr_{S^5}^*A} = 0$ is equivalent to \[F_A \wedge \omega_1 \wedge \theta =0.\] Similarly, $F_{\pr_{S^5}^*A} \wedge \Im \Omega_0 = 0$ is equivalent to \[ F_A \wedge \omega_3 = 0 \quad \textup{and} \quad F_A \wedge \omega_2 \wedge \theta = 0.\] Since $a \mapsto a \wedge \omega_3$ is an injective map for $a \in \Omega^1(S^5,\mathfrak{g}_P)$ (cf. \cite[Proposition~1.2.30]{Huybrechts-complex-Geometry}), the equation $F_A \wedge \omega_3=0$ implies $i_{X_\theta} F_A =0$. The other two equations then reduce to $F_A \wedge \omega_i =0$ for $i=1,2$.
\end{proof}

Recall from \autoref{bul: relationship SE structure S5 and K structure on P2} of \autoref{subsec: Sasaki--Einstein structure on S5} that the quotient map $\pr_{\mathbb{P}^2} \col S^5 \to S^5/\U(1) = \mathbb{P}^2$ satisfies $\pr_{\mathbb{P}^2}^*(\Lambda^{2,0}T^*_\mathbb{C}\mathbb{P}^2) = \langle \omega_2 + i \omega_3 \rangle_{\mathbb{C}}$ and $\pr_{\mathbb{P}^2}^*\omega_{\textup{FS}} = \omega_1$, where $\omega_{\textup{FS}} \in \Omega^2(\mathbb{P}^2)$ denotes the Fubini--Study form. Since the complexification of the bundle of self-dual forms is given by \[(\Lambda^2_+ T^* \mathbb{P}^2)_{\mathbb{C}} \cong \Lambda^{2,0}T^*_\mathbb{C}\mathbb{P}^2 \oplus \mathbb{C} \cdot \omega_{\textup{FS}} \oplus \Lambda^{0,2} T^*_\mathbb{C}\mathbb{P}^2 \] (cf. \cite[Lemma~2.1.57]{DonaldsonKronheimer-4-manifolds}), we immediately have:

\begin{corollary}\label{cor: ASD-instantons = pulled back SU(3)-instantons}
Let $A$ be a connection on a bundle (with real structure group) over $\mathbb{P}^2$. The pullback of $A$ to $S^5$ satisfies~\eqref{equ: cone reduction of instanton equation} (that is, the pullback of $A$ to $\mathbb{C}^3\setminus \{0\}$ is a dilation-invariant $\SU(3)$-instanton) if and only if $A$ is an ASD instanton over $\mathbb{P}^2$ with respect to the Fubini--Study metric and the orientation induced by the complex structure.
\end{corollary}

The following partial converse to this corollary is due to Baraglia and Hekmati:
\begin{proposition}[{\cite[Proposition~2.8]{BaragliaHekmati-contact-instantons}}]\label{prop: tangent cones pulled back from P2}
Let $\pi \col P \to S^5$ be a principal $G$-bundle and assume that $G$ has a trivial center. Furthermore, let $A\in \mathcal{A}(P)$ be an irreducible connection (i.e. the only gauge transformation preserving $A$ is the identity) that satisfies \eqref{equ: cone reduction of instanton equation}. Then there exists a principal $G$-bundle $\pi^\prime \col P^\prime \to \mathbb{P}^2$ and an ASD instanton $A^\prime \in \mathcal{A}(P^\prime)$ with $\pr_{\mathbb{P}^2}^*P^\prime = P$ and $\pr_{\mathbb{P}^2}^*A^\prime = A$.
\end{proposition}

We end this section with the following important class of examples:

\begin{example}
Let $\pi \col E \to \mathbb{P}^2$ be a holomorphic vector bundle which is slope-stable (or, more generally, slope-polystable) with respect to the Fubini--Study form $\omega_{\textup{FS}}$. The Donaldson--Uhlenbeck--Yau Theorem \cite[Theorem~1]{Donaldson-ASD-over-algebraic-surfaces} (and \cite[Main Theorem]{UhlenbeckYau-Donaldson-Uhlenbeck-Yau} for general compact Kähler manifolds) gives rise to an Hermitian metric $h$ on $E$ and a projective unitary connection $A\in \mathcal{A}(\mathbb{P}\textup{U}(E,h))$ that satisfies \[ \Lambda_{\omega_{\textup{FS}}} F_A = 0 \quad \textup{and} \quad F_A^{0,2}=0. \] Since \[(\Lambda^2_+ T^* \mathbb{P}^2)_{\mathbb{C}} \cong \Lambda^{2,0}T^*_\mathbb{C}\mathbb{P}^2 \oplus \mathbb{C} \cdot \omega_{\textup{FS}} \oplus \Lambda^{0,2} T^*_\mathbb{C}\mathbb{P}^2 \] (cf. \cite[Lemma~2.1.57]{DonaldsonKronheimer-4-manifolds}), \autoref{cor: ASD-instantons = pulled back SU(3)-instantons} implies that the pullback of $A$ to $\mathbb{C}^3\setminus\{0\}$ gives rise to a dilation-invariant $\SU(3)$-instanton.
\end{example}

\subsection{Conically singular connections}\label{sec: conically singular connections}

Let $Z^6$ be a compact 6-manifold with an $\SU(3)$-structure $(\omega,\Omega)$. Assume that $S \subset Z$ is a finite subset and $\pi \col P \to Z\setminus S$ is a principal $G$-bundle where $G$ is a compact Lie group whose Lie algebra $\mathfrak{g}$ has been equipped with an $\Ad$-invariant inner product. We denote the space of connections on $P$ by $\mathcal{A}(P)$. For the following discussion we fix for every $s\in S$
\begin{itemize}
\item a rate $\mu_s\in(-1,0)$,
\item a principal $G$-bundle $\pi_s \col P_s \to S^5$,
\item a connection $A_s \in \mathcal{A}(P_s)$ that satisfies~\eqref{equ: cone reduction of instanton equation}.
\end{itemize}

\begin{definition}\label{def: framed CS connections}
For each $s \in S$ let $\Upsilon_s\col B_R(0)\to Z$ be an $\SU(3)$-coordinate system (as in \autoref{def: SU(3)-coordinate system}) centered at $s$ and $\tilde{\Upsilon}_s \col \pr_{S^5}^*P_s \to P$ be a bundle isomorphism covering $\Upsilon_s$. We call a connection $A\in \mathcal{A}(P)$ conically singular with respect to $\{(\Upsilon_s,\tilde{\Upsilon}_s)\}_{s\in S}$ with rates $\mu\coloneqq \{\mu_s\}_{s\in S}$ and tangent cones $\{(P_s,A_s)\}_{s\in S}$
if 
\begin{equation}\label{equ: decay condition on framing}
\big\vert \nabla^k_{\pr^*_{S^5}A_s} (\tilde{\Upsilon}_s^*A-\pr_{S^5}^*A_s) \big\vert = \mathcal{O}(r^{\mu_s -k}) \quad \textup{for every $k\in \mathbb{N}_0$ as $r\to 0$.}
\end{equation} 
The set of all such conically singular connections will be denoted by $\mathcal{A}_\mu^{\textup{Fr}}(P,\{\Upsilon_s,\tilde{\Upsilon}_s,A_s\})$. Finally, we will call a bundle isomorphism $\tilde{\Upsilon}_s \col \pr_{S^5}^*P_s \to P$ covering an $\SU(3)$-coordinate system $\Upsilon_s$ that satisfies~\eqref{equ: decay condition on framing} a framing of $(\pi \col P \to Z\setminus S,A)$ at $s$.
\end{definition}

\begin{definition}\label{def: conically singular connections on fixed bundle}
A connection $A\in \mathcal{A}(P)$ is called conically singular with rates $\mu\coloneqq \{\mu_s\}_{s\in S}$ and tangent cones $\{(P_s,A_s)\}_{s\in S}$ if there exists a set framings $\{(\Upsilon_s,\tilde{\Upsilon}_s)\}_{s\in S}$ as in the previous definition, such that $A \in \mathcal{A}_\mu^{\textup{Fr}}(P,\{\Upsilon_s,\tilde{\Upsilon}_s,A_s\})$. The set of all conically singular connections on $P$ with given rates $\mu$ and tangent cones $\{(P_s,A_s)\}_{s\in S}$ will be denoted by $\mathcal{A}_\mu(P,\{P_s,A_s\})$. Moreover, $A$ is called a conically singular $\SU(3)$-instanton if $A$ is conically singular and satisfies~\eqref{equ: SU(3)-instanton}.
\end{definition}

\begin{remark}
The condition $\mu_s>-1$ ensures that $\pr_{S^5}^*A_s$ is the (up to gauge unique) tangent cone connection of $A\in \mathcal{A}_\mu(P,\{P_s,A_s\})$ at $s$. If $A \in \mathcal{A}_\mu(P,\{P_s,A_s\})$ is a conically singular $\SU(3)$-instanton, the tangent cones $A_s$ necessarily need to satisfy~\eqref{equ: cone reduction of instanton equation}.
\end{remark}

\begin{example}
Let $(Z,\omega,\Omega)$ be a compact Calabi--Yau 3-fold\footnote{Note that the results in this example hold for general compact Kähler manifolds (possibly without $\SU(3)$-structure) of any dimension. The relevant equation in this case is the Hermitian Yang--Mills equation $F_A^{0,2}=0$ and $i\Lambda_{\omega}F_A = \lambda \cdot \textup{id}$ for some $\lambda\in \mathbb{R}$ (when $\lambda = 0$ this is equivalent to \eqref{equ: SU(3)-instanton} in the presence of an $\SU(3)$-structure).} (i.e. $\diff \omega=0$ and $\diff \Omega=0$) and $\mathcal{E}$ be a reflexive sheaf on $Z$. Assume that $\mathcal{E}$ is the sum of slope-stable reflexive sheaves, which are locally free on the complement of a finite set $S \subset Z$. Furthermore, assume that around each $s\in S$ there are holomorphic coordinates $\varphi \col B_R(0) \to Z$ with $\varphi(0)=s$ such that $\varphi^*\mathcal{E} = \pr_{\mathbb{P}^2}^* \mathcal{F}$, where $\pr_{\mathbb{P}^2} \col B_R(0)\setminus \{0\} \hookrightarrow \mathbb{C}^3\setminus \{0\} \to \mathbb{P}^2$ is the canonical projection and $\mathcal{F}$ is a locally free sheaf on $\mathbb{P}^2$ which is the sum of slope-stable locally free sheaves. Bando and Siu \cite[Theorem~3]{BandoSiu-PHYM-over-reflexive-sheaves} proved that the holomorphic vector bundle $E \to Z\setminus S$ associated to $\mathcal{E}_{\vert Z\setminus S}$ admits an Hermitian inner product $h_E$ and a projective unitary connection $A \in \mathcal{A}(\mathbb{P}\U(E,h_E))$, which is an $\SU(3)$-instanton. Moreover,  Jacob, Sá Earp, and Walpuski proved \cite[Theorem~1.2]{AdamSaEarpWalpuski-tangent-cones-of-HYM-connections} that for every $s\in S$ there exists a rate $\mu_s>-1$ and a connection $A_s \in \mathcal{A}(\mathbb{P}\U(F,h_F))$, where $F \to \mathbb{P}^2$ is the vector bundle associated to $\mathcal{F}$ and $\mathbb{P}\U(F,h_F)$ is the projective unitary bundle assoicated to a suitable Hermitian inner product $h_F$, such that $A$ is conically singular with rate $\{\mu_s\}$ and tangent cones given by the respective pullbacks of $\{(\mathbb{P}\U(F,h_F),A_s)\}$ to $S^5$.
\end{example}

Assume that $A \in \mathcal{A}(P)$ is conically singular and that $s \in S$ is a singular point. The following proposition shows that the set of framings $(\Upsilon_s,\tilde{\Upsilon}_s)$ at $s$ (as in \autoref{def: framed CS connections}), is up to terms of order $\mathcal{O}(r^{\mu_s+1})$ an $\Stab_{\SU(3)}(A_s)$-torsor, where 
\begin{equation} \label{equ: definition Stab_SU(3)(A_s)}
\Stab_{\SU(3)}(A_s) \coloneqq \{ \tilde{U} \col P_s \xrightarrow{\sim} P_s \mid \textup{$\tilde{U}$ covers an element in $\SU(3)$ and $\tilde{U}^*A_s = A_s$}\}. 
\end{equation}

\begin{proposition}\label{prop: change of framing}
Let $A\in \mathcal{A}_\mu(P,\{P_s,A_s\})$ be a conically singular connection. Furthermore, let $\Upsilon_i \col B_R(0) \to Z$ for $i=1,2$ be two $\SU(3)$-coordinate systems both centered at (the same) $s\in S$ and $\Tilde{\Upsilon}_i \col \pr_{S^5}^*P_s \to \Upsilon_i^*P$ be two isomorphisms that both satisfy~\eqref{equ: decay condition on framing}. Then there exists a bundle isomorphism $\Tilde{U} \col P_s \to P_s$ covering $U \coloneqq D_0(\Upsilon_2^{-1} \circ \Upsilon_1) \in \SU(3)$ such that $\Tilde{U}^*A_s=A_s$ and  
\begin{equation}\label{equ: estimate on change of framing}
\big\vert \nabla_{\pr_{S^5}^*A_s}^k\big(\Tilde{\Upsilon}_2^{-1}\circ \Tilde{\Upsilon}_1 - \pr_{S^5}^*\Tilde{U}\big)\big\vert = \mathcal{O}(r^{\mu_s+1-k}) \quad \textup{for every $k \in \mathbb{N}_0$ as $r \to 0$.}
\end{equation}
Here, we assumed that $G$ is a subgroup of $\GL(W)$ for some vector space $W$ and $\tilde{\Upsilon}_2^{-1} \circ \tilde{\Upsilon}_1$ and $\pr_{S^5}^*\tilde{U}$ can therefore both be regarded as vector bundle homomorphisms $P_s \times_G W\to P_s \times_G W$. Moreover, we concatenate $\tilde{\Upsilon}_2^{-1} \circ \tilde{\Upsilon}_1$ with the parallel transport over the straight line that connects any $(\Upsilon_2^{-1}\circ \Upsilon_1)(z)$ with $Uz$. Then both homomorphisms are sections of the linear bundle $\Hom(P_s \times_G W,U^*P_s \times_G W)$, so that the difference and the covariant derivatives are indeed well-defined.
\end{proposition}
\begin{proof}
By pre-composing $\tilde{\Upsilon}_2$ with the parallel transport over the straight line that connects any $Uz \in \mathbb{C}^3\setminus \{0\}$ with $(\Upsilon_2^{-1} \circ \Upsilon_1)(z)$, we may in the following assume that $\tilde{\Upsilon}_2^{-1} \circ \tilde{\Upsilon}_1$ covers $U$. Note that since $\mu_s<1$, the modified $\tilde{\Upsilon}_2$ still satisfies~\eqref{equ: decay condition on framing}. 

Since the isomorphisms $\tilde{\Upsilon}_1$ and $\tilde{\Upsilon}_2$ both satisfy~\eqref{equ: decay condition on framing}, we have for every $k \in \mathbb{N}_0$
\begin{align*}
\big\vert \nabla_{\pr_{S^5}^*A_s}^k \big((\Tilde{\Upsilon}_2^{-1}\circ \Tilde{\Upsilon}_1)^*\pr_{S^5}^*A_s-\pr_{S^5}^*A_s\big)\big\vert &= \big\vert \nabla_{\pr_{S^5}^*A_s}^k \big(\Tilde{\Upsilon}_1^*(A+a)-\pr_{S^5}^*A_s\big)\big\vert \\
&\leq \big\vert \nabla_{\pr_{S^5}^*A_s}^k \big(\Tilde{\Upsilon}_1^*A-\pr_{S^5}^*A_s\big)\big\vert + \big\vert \nabla_{\pr_{S^5}^*A_s}^k \big(\Tilde{\Upsilon}_1^*a\big)\big\vert
\end{align*}
where $a \coloneqq (\tilde{\Upsilon}_2^{-1})^*\pr_{S^5}^*A_s - A$ satisfies $\vert \nabla^k(\Tilde{\Upsilon}_2^*a)\vert= \mathcal{O}(r^{\mu_s-k})$. Therefore, 
\begin{equation}\label{equ: auxiliary estimate on change of framing}
\big\vert \nabla_{\pr_{S^5}^*A_s}^k \big((\Tilde{\Upsilon}_2^{-1}\circ \Tilde{\Upsilon}_1)^*\pr_{S^5}^*A_s-\pr_{S^5}^*A_s\big)\big\vert = \mathcal{O}(r^{\mu_s-k}) \quad \textup{for every $k \in \mathbb{N}_0$ as $r \to 0$.}
\end{equation}

Next, we define for sufficiently small $r>0$ the following 1-parameter family of bundle isomorphisms on $P_s \to S^5$: \[\Tilde{g}_r \coloneqq \Tilde{\delta}_r^* (\Tilde{\Upsilon}_2^{-1} \circ \Tilde{\Upsilon}_1)_{\vert S^5_r} = \Tilde{\delta}^{-1}_r \circ (\Tilde{\Upsilon}_2^{-1} \circ \Tilde{\Upsilon}_1)_{\vert S^5_r} \circ \Tilde{\delta}_r \] where $\delta_r\col S_1^5\to S^5_r$ for $r\in (0,R)$ denotes the dilation map from the sphere of radius 1 onto the sphere of radius $r$ and $\Tilde{\delta}_r$ is its canonical lift to $\pr_{S^5}^*P_s$. Equation~\eqref{equ: auxiliary estimate on change of framing} implies that for every $k\geq 1$: \[ \big\vert \nabla^k_{A_s} \Tilde{g}_r \big\vert = \mathcal{O}(r^{\mu_s+1}).\] Since the $C^0$-norm of $\Tilde{g}_r$ is bounded due to the compactness of $G$, the Arzelà--Ascoli Theorem implies that there exists an isomorphism $\Tilde{U}\col P_s \to P_s$ covering $U=D_0(\Upsilon_2^{-1}\circ \Upsilon_1)$ such that $\Tilde{g}_r \to \Tilde{U}$ in $C^\infty$ on $S^5$ as $r\to 0$. Furthermore, since $\mu_s+1>0$, the isomorphism $\Tilde{U}$ is parallel with respect to $A_s$. 

Equation~\eqref{equ: auxiliary estimate on change of framing} with $k=0$ implies $\vert \partial_r \Tilde{g}_r \vert = \mathcal{O}(r^{\mu_s})$ and therefore \[\big\vert \Tilde{g}_r - \Tilde{U} \big\vert \leq \int_0^r \big\vert \partial_t\Tilde{g}_t \big\vert \diff t = \mathcal{O}(r^{\mu_s+1}). \] Dilating back, this implies~\eqref{equ: estimate on change of framing} with $k=0$. Since $\Tilde{U}$ is parallel, Equation~\eqref{equ: estimate on change of framing} with $k\geq 1$ follows from~\eqref{equ: auxiliary estimate on change of framing} with $k-1$.
\end{proof}
If $(W_\rho,\rho)$ is any $G$-representation, then we will denote in the following by $P\times_{\rho}W_\rho$ its associated vector bundle.
\begin{definition}
Let $\{(\Upsilon_s,\tilde{\Upsilon}_s)\}_{s\in S}$ be a set of framings as in \autoref{def: framed CS connections}. For any set of rates $\lambda \coloneqq \{\lambda_s\}_{s\in S}\in \mathbb{R}^{ S }$ and any $G$-representation $(W_\rho,\rho)$ let 
\begin{align*}
\Omega^\ell_\lambda(Z\setminus S,P\times_\rho W_\rho;\{\Upsilon_s,\tilde{\Upsilon}_s\}) &\coloneqq \big\{ \eta \in \Omega^\ell(Z\setminus S, P\times_\rho W_\rho)  \ \big\vert \\
& \qquad \qquad  \qquad \quad \big\vert \nabla^k_{\pr_{S^5}^*A_s} (\tilde{\Upsilon}_s^*\eta) \big\vert = \mathcal{O}(r^{\lambda_s-k}) \textup{ for $\forall s\in S$} \big\}
\end{align*} 
where the norm and the connection on $\Upsilon_s^*(\Lambda^\ell T^*Z) \cong \Lambda^\ell T^*(B_R(0)\setminus \{0\})$ are induced by the flat metric on $B_R(0)\setminus \{0\} \subset \mathbb{C}^3 \setminus \{0\}$.
\end{definition}
\begin{remark}
Whenever the set $\{(\Upsilon_s,\tilde{\Upsilon}_s)\}_{s\in S}$ of $\SU(3)$-coordinate systems and bundle isomorphisms in the previous definition is clear from the context, then we will remove it from the notation and simply write $\Omega^\ell_\lambda(Z\setminus S,P\times_\rho W_\rho)$.
\end{remark}
The following proposition shows that the definition of $\Omega^\ell_\lambda(Z\setminus S,P\times_\rho W_\rho;\{\Upsilon_s,\tilde{\Upsilon}_s\})$ only depends on the $\times_{s\in S} \Stab_{\SU(3)}(A_s)$-orbit of $\{(\Upsilon_s,\tilde{\Upsilon}_s)\}_{s\in S}$, where $\Stab_{\SU(3)}(A_s)$ was defined in~\eqref{equ: definition Stab_SU(3)(A_s)}.
\begin{proposition}
Let $\Upsilon_i \col B_R(0) \to Z$ for $i=1,2$ be two $\SU(3)$-coordinate systems both centered at (the same) $s\in S$ and $\Tilde{\Upsilon}_i \col \pr_{S^5}^*P_s \to \Upsilon_i^*P$ be two isomorphisms such that there exists a bundle isomorphism $\Tilde{U} \col P_s \to P_s$ covering $U \coloneqq D_0(\Upsilon_2^{-1} \circ \Upsilon_1) \in \SU(3)$ with $\Tilde{U}^*A_s=A_s$ and  
\begin{equation}
\big\vert \nabla_{\pr_{S^5}^*A_s}^k\big(\Tilde{\Upsilon}_2^{-1}\circ \Tilde{\Upsilon}_1 - \pr_{S^5}^*\Tilde{U}\big)\big\vert = \mathcal{O}(r^{\mu_s+1-k}).
\end{equation}
Assume that $\eta \in \Omega^\ell(Z\setminus S,P\times_\rho W_\rho)$ satisfies $\vert \nabla^k_{\pr_{S^5}^*A_s} (\Tilde{\Upsilon}_1^*\eta) \vert = \mathcal{O}(r^{\lambda_s-k})$ for every $k\in \mathbb{N}_0$ and some $\lambda_s\in \mathbb{R}$. Then $\eta$ satisfies also $\vert \nabla^k_{\pr_{S^5}^*A_s} (\Tilde{\Upsilon}_2^*\eta) \vert = \mathcal{O}(r^{\lambda_s-k})$ for every $k\in \mathbb{N}_0$.
\end{proposition}

\begin{proof}
Dropping the subscript $\pr^*A_s$ from the covariant derivative for notational convenience, we obtain
\begin{align*}
\vert \nabla^k(\Tilde{\Upsilon}_2^*\eta)\vert &\leq \vert \nabla^k((\Tilde{\Upsilon}_1^{-1}\circ\Tilde{\Upsilon}_2- \Tilde{U})^*\Tilde{\Upsilon}_1^*\eta)\vert + \vert \nabla^k(\Tilde{U}^*\Tilde{\Upsilon}_1^*\eta) \vert \\
&= \mathcal{O}(r^{\mu_s+1+\lambda_s-k}) + \mathcal{O}(r^{\lambda_s-k}).
\end{align*}
Since $\mu_s+1>0$, this term is $\mathcal{O}(r^{\lambda_s-k})$.
\end{proof}

We end this section by showing that whenever the $\SU(3)$-structure $(\omega,\Omega)$ on $Z$ satisfies $\diff^* \omega= 0$ and  $\diff \Omega = w_1 \omega^2$ for $w_1 = -\frac{i}{6} \langle \diff \omega , \overline{\Omega} \rangle_\mathbb{C} \in \mathbb{C}$, then the $\SU(3)$-instanton equation can be augmented to an elliptic equation modulo gauge.

\begin{proposition}\label{prop: augmented instanton equation}
Fix a conically singular connection $A \in \mathcal{A}_\mu(P,\{P_s,A_s\})$ and define $\Omega^\ell_\mu(Z\setminus S,\mathfrak{g}_P)$ with respect to any framing $\{(\Upsilon_s,\tilde{\Upsilon}_s)\}_{s\in S}$ of $A$. Assume that $(\omega,\Omega)$ satisfies $\diff^*\omega=0$ and $\diff \Omega = w_1 \omega^2$. Then $A+a$ for $a\in \Omega^1_\mu(Z\setminus S,\mathfrak{g}_P)$ is a conically singular $\SU(3)$-instanton in Coulomb gauge relative to $A$ if and only if there are $\xi_1,\xi_2 \in \Omega^0_\mu(Z\setminus S, \mathfrak{g}_P)$ such that \begin{equation} \label{equ: augmented instanton equations}
\Lambda_{\omega} F_{A+a} = 0, \quad *(F_{A+a} \wedge \Im \Omega) + \diff_{A+a} \xi_1 + J^*(\diff_{A+a} \xi_2) = 0, \quad \textup{and} \quad \diff_A^*a = 0 
\end{equation}
where $J$ is the almost complex structure associated to $(\omega,\Omega)$ (cf. \autoref{remark: almost complex structure associated to (omega,Omega)}). Moreover, the two sections $\xi_1,\xi_2$ in the latter case satisfy $\diff_{A+a} \xi_1 = \diff_{A+a}\xi_2 =0$.
\end{proposition}
\begin{remark}
Equation~\eqref{equ: augmented instanton equations} is up to a zeroth-order term, which we have chosen to discard, the dimensional reduction of the $\textup{G}_2$-monopole equation in Coulomb gauge.
\end{remark}
\begin{proof}
First note that the previous proposition and \autoref{prop: change of framing} imply that the definition of $\Omega^\ell_\mu(Z\setminus S,\mathfrak{g}_P)$ is independent of the chosen $\{(\Upsilon,\Tilde{\Upsilon})\}_{s \in S}$. 

If $A+a$ is an $\SU(3)$-instanton in Coulomb gauge relative to $A$, then it satisfies the three equations with $\xi_1=\xi_2=0$. Conversely, assume that $a,\xi_1,$ and $\xi_2$ satisfy the equations in~\eqref{equ: augmented instanton equations}. Applying $\diff_{A+a}^*$ to the second equation and using $*(F_{A+a} \wedge \diff \Im \Omega) = 2 \Im(w_1) \Lambda_\omega F_{A+a} = 0$ by the first equation, $\frac{1}{2}* (\diff_{A+a} \xi_2 \wedge \omega \wedge \omega) = -J^*\diff_{A+a} \xi_2$, and $\diff (\omega \wedge \omega) = * \diff^*\omega = 0$ we obtain \[ 0 = \diff_{A+a}^*\diff_{A+a} \xi_1 - [\Lambda_\omega F_{A+a},\xi_2]= \diff_{A+a}^*\diff_{A+a} \xi_1. \] An integration by parts argument, which is justified since $\langle \diff_{A+a}\xi_1, \xi_1 \rangle = \mathcal{O}(r^{2 \mu_s -1})$ at any $s\in S$ with $\mu_s>-1$, then gives $\diff_{A+a} \xi_1 =0$. After further multiplication by $J^*$, the second equation reduces to \[ *(F_{A+a} \wedge \Re \Omega) - \diff_{A+a}\xi_2 =0. \] Again applying $\diff_{A+a}^*$ and using integration by parts as above gives $\diff_{A+a}\xi_2=0$.
\end{proof}

Now define the map
\begin{align*}
\Theta_A \col &\Omega^0_\mu(Z\setminus S, \mathfrak{g}_P \oplus \mathfrak{g}_P \oplus T^*Z\otimes \mathfrak{g}_P) \to \Omega^0_\mu(Z\setminus S, \mathfrak{g}_P \oplus \mathfrak{g}_P \oplus T^*Z\otimes \mathfrak{g}_P) \\
&(\xi_1,\xi_2,a) \mapsto \big(\diff_A^*a, \Lambda_{\omega} F_{A+a}, *(F_{A+a} \wedge \Im \Omega) + \diff_{A+a} \xi_1 + J^*(\diff_{A+a} \xi_2) \big).
\end{align*}
Then \[\Theta_A(\xi_1,\xi_2,a) = \Theta_A(0) + L_A(\xi_1,\xi_2,a) + Q_A(\xi_1,\xi_2,a) \] for
\begin{align*}
L_A \coloneqq 
\begin{pmatrix}
0 & 0 & \diff_A^* \\
0 & 0 & \Lambda_{\omega} \diff_A \\
\diff_A & J^* \diff_A & *(\Im \Omega \wedge\diff_A) 
\end{pmatrix}
\end{align*}
and \[Q_A(\xi_1,\xi_2,a) \coloneqq \big(0, \tfrac{1}{2} \Lambda_\omega [a\wedge a], *(\tfrac{1}{2}[a\wedge a] \wedge \Im \Omega) + [a,\xi_1]+J^*[a,\xi_2]\big). \] The following proposition follows from a straight forward calculation similar to the proof of \autoref{prop: augmented instanton equation}.
\begin{proposition}\label{prop: adjoint of model operator}
Let $(\omega,\Omega)$ be an $\SU(3)$-structure with $\diff^* \omega = 0$ and $\diff \Omega = w_1 \omega^2$. The linear operator $L_A$ is elliptic and its formal adjoint is given by \[L_A^*(\xi_1,\xi_2,a) = L_A(\xi_1,\xi_2,a) + (0,0,2 \Im (w_1) J^*a).\] Thus, if $w_1$ is real, then $L_A$ is formally self-adjoint.
\end{proposition}
\begin{remark}
In the following, we will always assume that $\Im(w_1) =0$ so that $L_A$ is formally self-adjoint. By \autoref{prop: modifying the SU(3)-structure such that dIm Omega = 0} this does not pose a (significant) additional restriction on the $\SU(3)$-structure. 
\end{remark}

\section{The moduli space of framed conically singular connections}
\label{sec: moduli space of framed conically singular connections}

The aim of the upcoming sections is to define the moduli space of conically singular $\SU(3)$-instantons with prescribed tangent cones and to describe its local structure. As a first step, this section investigates the local structure of the space of conically singular connections in which the framing $\{(\Upsilon_s,\tilde{\Upsilon}_s)\}$ in \autoref{def: framed CS connections} is taken as part of the collected data.

\subsection{Definition of the space and its topology}
\label{sec: space of framed connections and its topology}

Throughout this section, $Z^6$ is a compact 6-manifold with an $\SU(3)$-structure $(\omega,\Omega)$. Furthermore, $G$ is a compact Lie group whose Lie algebra $\mathfrak{g}$ has been equipped with an $\Ad$-invariant inner product.

First, we define an equivalence relation on the set of framings $\{(\Upsilon_s,\tilde{\Upsilon}_s)\}_{s\in S}$ appearing in \autoref{def: framed CS connections}.

\begin{definition}\label{def: equivalence relation framing}
Let $\pi \col P \to Z\setminus \{s\}$ be a principal $G$-bundle and let $\pi_s \col P_s \to S^5$ and $A_s \in \mathcal{A}(P_s)$ be chosen. Furthermore, let $\Upsilon_1 \col B_{R_1}(0) \to Z$ and $\Upsilon_2 \col B_{R_2}(0) \to Z$ be two $\SU(3)$-coordinate systems centered at $s$ and $\Tilde{\Upsilon}_i \col \pr_{S^5}^*P_s \to P$ be two framings covering $\Upsilon_i$, respectively. We call $(\Upsilon_1,\Tilde{\Upsilon}_1)$ and $(\Upsilon_2,\Tilde{\Upsilon}_2)$ equivalent at rate $\mu_s \in \mathbb{R}$ with respect to $(P_s,A_s)$ if $D_0\Upsilon_1 = D_0\Upsilon_2$ and \[ \vert \nabla^k_{\pr_{S^5}^*A_s}(\Tilde{\Upsilon}_2^{-1}\circ \Tilde{\Upsilon}_1 - \textup{Id} )\vert = \mathcal{O}(r^{\mu_s+1-k}) \quad \textup{for every $k\in \mathbb{N}_0$ as $r \to 0$.}\]
\end{definition}

\begin{definition}[{\textbf{Moduli space of framed connections}}] \label{def: moduli space of framed connections}
Let $N\in \mathbb{N}$ be the number of singular points, $\mu\coloneqq \{\mu_i\}_{i\in \{1,\dots,N\}}$ for $\mu_i \in (-1,0)$ be a set of rates, and $\{(P_i,A_i)\}_{i\in \{1,\dots,N\}}$ be a set of (prescribed) tangent cones. With these we define the following:
\begin{enumerate}
\item Let $\mathcal{A}_\mu^{\textup{Fr}}(\{P_i,A_i\})$ be the set consisting of elements of the form: \[\big(S,\pi \col P \to Z\setminus S,\{[\Upsilon_i,\tilde{\Upsilon}_i]\}_{i\in \{1,\dots,N\}},A \big)\] where 
\begin{itemize}
\item $S= \{s_1,\dots,s_N\} \subset Z$ is a totally ordered subset,
\item $\pi \col P \to Z \setminus S$ is a principal $G$-bundle,
\item for each $i=1,\dots,N$, $[\Upsilon_i,\tilde{\Upsilon}_i]$ is an equivalence class (with respect to the relation in \autoref{def: equivalence relation framing}) of
\begin{itemize}
\item an $\SU(3)$-coordinate system $\Upsilon_i \col B_R(0) \to Z$ centered at $s_i$,
\item a bundle isomorphism $\Tilde{\Upsilon}_i \col \pr_{S^5}^*P_{i} \to \Upsilon_i^*P$
\end{itemize} 
at rate $\mu_i$ with respect to $(P_i,A_i)$. 
\item $A \in \mathcal{A}_\mu^{\textup{Fr}}(P,\{\Upsilon_i,\tilde{\Upsilon}_i,A_i\})$ (as in \autoref{def: framed CS connections}, where $(\Upsilon_i,\tilde{\Upsilon}_i)$ is any representative of the equivalence class $[\Upsilon_i,\tilde{\Upsilon}_i]$).
\end{itemize}
\item Let $\mathcal{B}_\mu^{\textup{Fr}}(\{P_i,A_i\}) \coloneqq \mathcal{A}_\mu^{\textup{Fr}}(\{P_i,A_i\})/\sim$ where the equivalence relation $\sim$ is defined by \[\big(S,\pi\col P \to Z\setminus S, \{[\Upsilon_{i},\Tilde{\Upsilon}_{i}]\}_{i=1,\dots,N}, A\big) \sim \big(S^\prime,\pi^\prime\col P^\prime \to Z\setminus S^\prime, \{[\Upsilon_{i}^\prime,\Tilde{\Upsilon}_{i}^\prime]\}_{i=1,\dots,N}, A^\prime\big)\] if $S = S^\prime$ (as ordered sets), $D_{0}(\Upsilon_{i}^{-1}\circ \Upsilon_{i}^\prime)= \textup{Id}$ for all $i=1,\dots, N$, and there exists an isomorphism $F \col P^\prime \to P$ (covering the identity) that satisfies
\begin{itemize}
\item $\big\vert \nabla_{\pr^*_{S^5}A_{s_i}}^k \big( \Tilde{\Upsilon}_{i}^{-1} \circ F \circ \Tilde{\Upsilon}_{i}^\prime - \textup{Id} \big) \big\vert = \mathcal{O}(r^{\mu_i+1-k})$ for every $k\in \mathbb{N}_0$ and $i=1,\dots,N$,
\item $F^*A=A^\prime$.
\end{itemize}
In the formulation above we again assumed that $G$ is a subgroup of $\GL(W)$ for some vector space $W$ and $\Tilde{\Upsilon}_{i}^{-1} \circ F \circ \Tilde{\Upsilon}_{i}^\prime$ can therefore be regarded as a vector bundle homomorphism $P_i \times_G W\to P_i \times_G W$. Moreover, we use parallel transport over straight lines in order to regard $\Tilde{\Upsilon}_{i}^{-1} \circ F \circ \Tilde{\Upsilon}_{i}^\prime$ as a section of the linear bundle $\End(P_s \times_G W)$, so that the difference and the covariant derivatives are well-defined.
\end{enumerate}
\end{definition}

\begin{remark}
First, note that the definition of $\mathcal{A}_\mu^{\textup{Fr}}(P,\{\Upsilon_i,\tilde{\Upsilon}_i,A_i\})$ in \autoref{def: framed CS connections} remains unchanged when replacing any pair $(\Upsilon_i,\tilde{\Upsilon}_i)$ with an equivalent (in the sense of \autoref{def: equivalence relation framing}) pair $(\Upsilon_i^\prime,\tilde{\Upsilon}_i^\prime)$. For the definition of $\mathcal{A}^\Fr_\mu(\{P_i,A_i\})$ the choice of representative of $[\Upsilon_i,\tilde{\Upsilon}_i]$ is therefore irrelevant.
\end{remark}

\begin{remark}
In the following we will equip the spaces in the previous definition with a topology such that they are non-connected. If one wishes to restrict to a connected component, then one should additionally assume that all bundles in the definition above are isomorphic to a fixed bundle via an isomorphism that covers a diffeomorphism that is isotopic to the identity.
\end{remark}

\begin{remark}\label{rem: totally ordered singular set framed case}
In the previous definition we take the singular set $S$ to be totally ordered because we want to prescribe the tangent cone $(P_i,A_i)$ at each singularity in advance. A 'full' moduli theory of conically singular connections (and instantons) should of course take the tangent cone around any singularity as part of the data collected in $\mathcal{A}_\mu^{\textup{Fr}}$. As a step toward such a moduli space with variable tangent cones, one should also identify a singular connection \[ (S, \pi \col P \to Z \setminus S, \{[\Upsilon_i,\tilde{\Upsilon}_i]\}_{i\in \{1,\dots,N\}},A) \] in the definition of $\mathcal{B}_{\mu}^{\textup{Fr}}(\{P_i,A_i\})$ given above with \[(\sigma(S), \pi \col P \to Z \setminus S, \{[\Upsilon_i,\tilde{\Upsilon}_i]\}_{i\in \{1,\dots,N\}},A), \] where $\sigma \in S_N$ is any permutation that satisfies $\mu_i = \mu_{\sigma(i)}$ and $(P_i,A_i) = (P_{\sigma(i)},A_{\sigma(i)})$ for every $i=1,\dots,N$ and $\sigma(S) \coloneqq \{\sigma(s_1),\dots,\sigma(s_N)\}$ (as a totally ordered set). Note, however, that because the symmetric group $S_N$ is finite, dividing out this additional group action will not change the local structure of $\mathcal{B}_{\mu}^{\textup{Fr}}(\{P_i,A_i\})$. 
\end{remark}

Next, we define topologies on $\mathcal{A}_\mu^\Fr(\{P_i,A_i\})$ and $\mathcal{B}_\mu^\Fr(\{P_i,A_i\})$. For this we first define a topology on the set $\mathcal{A}_\mu^\Fr(P,\{\Upsilon_i,\tilde{\Upsilon}_i,A_i\})$ of all conically singular connections on a fixed bundle with respect to a fixed set of framings $\{(\Upsilon_i,\tilde{\Upsilon}_i)\}$.

\begin{definition}\label{def: C^infty_mu topology}
Let $S \coloneqq \{s_1,\dots,s_N\} \subset Z$ be a totally ordered set, and let $\{(P_{i},A_{i})\}_{i=1,\dots,N}$ and $\{(\Upsilon_i,\tilde{\Upsilon}_i)\}_{i=1,\dots,N}$ be as in \autoref{def: framed CS connections}. Furthermore, we assume that each coordinate system $\Upsilon_i$ is defined over a ball $B_R(0)\subset \mathbb{C}^3$ of the same radius $R$. For any $\lambda \coloneqq \{\lambda_1,\dots,\lambda_N\}$ we define $\rho \col Z \setminus S \to (0,\infty)$ (a distance function) and $w_\lambda \col Z \to \mathbb{R}$ (a possibly non-continuous rate function) via
\begin{align*}
\rho(x) \coloneqq \begin{cases}
\big\vert \Upsilon_i^{-1}(x) \big\vert \quad &\textup{if $x \in \Upsilon_i(B_R(0))$,} \\
R &\textup{else,}
\end{cases}
\qquad \textup{and} \qquad 
w_\lambda(x) \coloneqq \begin{cases}
\lambda_i \quad &\textup{if $x \in \Upsilon_i(B_R(0))$,} \\
-1 &\textup{else.}
\end{cases}
\end{align*}
For some fixed $A_0 \in \mathcal{A}_\mu^\Fr(P,\{\Upsilon_i,\tilde{\Upsilon}_i,A_i\})$, we define the $C^\infty_\mu$-topology on $\mathcal{A}_\mu^\Fr(P,\{\Upsilon_i,\tilde{\Upsilon}_i,A_i\})$ to be the topology generated by the following set of semi-norms: \[ \vert A_0+ a \vertWC{k}{\mu}{} \coloneqq \sup \vert \rho^{-w_{\mu-k}} \nabla_{A_0}^k a \vert \quad \textup{where $k\in \mathbb{N}_0$ and $A_0+a \in \mathcal{A}_\mu^\Fr(P,\{\Upsilon_i,\tilde{\Upsilon}_i,A_i\})$} \] (where $\mu-k \coloneqq \{\mu_1-k,\dots,\mu_N-k\}$).
\end{definition}

\begin{remark}\label{rmk: independence of C^infty_mu-topology from choices}
A moment's thought reveals that the definition of the $C^\infty_\mu$-topology is independent of the choice of radius $R>0$ and base connection $A_0 \in \mathcal{A}_\mu^\Fr(P,\{\Upsilon_i,\tilde{\Upsilon}_i,A_i\})$. Moreover, if $\{(\Upsilon_i^\prime,\tilde{\Upsilon}_i^\prime)\}_{i=1,\dots,N}$ is another set of framings which is equivalent to $\{(\Upsilon_i,\tilde{\Upsilon}_i)\}_{i=1,\dots,N}$ in the sense of \autoref{def: equivalence relation framing}, then $\mathcal{A}_\mu^\Fr(P,\{\Upsilon_i,\tilde{\Upsilon}_i,A_i\}) = \mathcal{A}_\mu^\Fr(P,\{\Upsilon_i^\prime,\tilde{\Upsilon}_i^\prime,A_i\})$ as topological spaces (equipped with their respective $C^\infty_\mu$-topologies).
\end{remark}

With the $C^\infty_\mu$-topology at hand, we now define topologies on $\mathcal{A}_\mu^\Fr(\{P_i,A_i\})$ and $\mathcal{B}_\mu^\Fr(\{P_i,A_i\})$.

\begin{definition}\label{def: base of topology}
For fixed rates $\mu = \{\mu_i\}_{i=1,\dots,N}$ and tangent cones $\{(P_i,A_i)\}_{i=1,\dots,N}$, let $\mathcal{A}_\mu^{\textup{Fr}}(\{P_i,A_i\})$ and $\mathcal{B}_\mu^{\textup{Fr}}(\{P_i,A_i\})$ be as in \autoref{def: moduli space of framed connections}. We first define the following collection $\mathcal{C}$ of subsets of $\mathcal{A}_\mu^\Fr(\{P_i,A_i\})$ which will subsequently serve as the basis for a topology.

Let $\mathbb{A} \coloneqq (S, \pi \col P \to Z\setminus S, \{[\Upsilon_i,\Tilde{\Upsilon}_i]\} ,A) \in \mathcal{A}_\mu^{\textup{Fr}}(\{P_i,A_i\})$ be any element. Furthermore, assume that we have have chosen
\begin{itemize}
\item An open neighbourhood $V_1 \subset \{f \in \textup{Diff}(Z) \mid f^*(\omega,\Omega)_{f(s_i)}=(\omega,\Omega)_{s_i}$ \textup{ for all $i=1,\dots,N$ and $S=\{s_1,\dots,s_N\}$}\} of the identity (with respect to the $C^\infty$-topology). Furthermore, we assume $\textup{dist}(s_i,f(s_i))<\varepsilon$ for every $f\in V_1$ and $s_i\in S$ where $\varepsilon\ll \textup{dist}(s_i,s_j)$ for all $s_i\neq s_j \in S$.
\item An open neighbourhood $V_2 \subset \times_{i=1}^N \textup{Stab}_{\SU(3)}(A_i)$ of the identity (with respect to the $C^\infty$-topology) where $\Stab_{\SU(3)}{(A_i)}$ is as in \eqref{equ: definition Stab_SU(3)(A_s)}.
\item An open neighbourhood $V_3 \subset \mathcal{A}_\mu^\Fr(P,\{\Upsilon_i,\tilde{\Upsilon}_i,A_i\})$ of $A$ with respect to the $C^\infty_{\mu}$-topology.
\end{itemize}
We then define $V_{\mathbb{A}}(V_1,V_2,V_3) \subset \mathcal{A}_\mu^\Fr(\{P_i,A_i\})$ as
\begin{align*}
V_{\mathbb{A}}(V_1,V_2,V_3) \coloneqq \Big\{\big(S^\prime, f \circ \pi \col P \to Z\setminus S^\prime, \{[f \circ \Upsilon_i\circ U, \Tilde{\Upsilon}_i & \circ \Tilde{U}_i]\}_{i} ,A^\prime\big) \in \mathcal{A}^\Fr_\mu(\{P_i,A_i\}) \ \Big\vert \\
&(f,\{U_i,\Tilde{U}_i\}_i,A^\prime) \in V_1\times V_2 \times V_3 \Big\}
\end{align*} 
and $\mathcal{C}$ as the collection of all such subsets, i.e. $ \mathcal{C} \coloneqq \cup V_{\mathbb{A}}(V_1,V_2,V_3)$ (where the union is taken over all $\mathbb{A}$, $V_1$, $V_2$, and $V_3$ as above).

We now equip $\mathcal{A}_\mu^\Fr$ with the topology generated by $\mathcal{C}$ and $\mathcal{B}_\mu^{\textup{Fr}}$ with the quotient topology.
\end{definition}

\begin{remark}
Let $S\subset Z$ be a subset and $\pi \col P \to Z\setminus S$ be a principal $G$-bundle. Furthermore, let $f \col Z \to Z$ be a diffeomorphism with $S^\prime \coloneqq f(S)$. The principal $G$-bundle $f \circ \pi \col P \to Z\setminus S^\prime$ (where the total space $P$ and the $G$-action are the same as for $\pi \col P \to Z\setminus S$ and the projection is concatenated with $f$) is isomorphic to the push-forward bundle $f_*\pi \col f_*P \to Z\setminus S^\prime$. This isomorphism identifies the (same) connection $A$ (now considered over $f \circ \pi \col P \to Z\setminus S^\prime$) with the push-forward connection over $f_*P$. Thus, the neighbourhood $V_\mathbb{A}(V_1,V_2,V_3) \in \mathcal{C}$ defined above consists of conically singular connections on the push-forward bundles of $P$ by certain diffeomorphisms. Moreover, note that by \autoref{rmk: independence of C^infty_mu-topology from choices}, the definition of $V_\mathbb{A}(V_1,V_2,V_3)$ is independent of the choice of representative for the equivalence classes $\{[\Upsilon_i,\tilde{\Upsilon}_i]\}_{i=1,\dots,N}$.
\end{remark}

\begin{proposition}
The collection $\mathcal{C}$ in the previous definition is closed under finite intersections and defines therefore a basis for the topology on $\mathcal{A}_\mu^\Fr(\{P_i,A_i\})$.   
\end{proposition}

\begin{proof}
Let $\mathbb{A}\coloneqq (S, \pi \col P \to Z\setminus S, \{[\Upsilon_i,\Tilde{\Upsilon}_i]\},A)$ be an element in $\mathcal{A}_\mu^{\textup{Fr}}(\{P_i,A_i\})$ and let $V_{\mathbb{A}}(V_1,V_2,V_3)\subset \mathcal{A}_\mu^{\textup{Fr}}(\{P_i,A_i\})$ be an open neighbourhood (associated to open sets $V_1,V_2,V_3$) of the form described in the previous definition. We will first show that for any $\mathbb{A}^\prime \in V_{\mathbb{A}}(V_1,V_2,V_3)$ there exist $V_1^\prime$, $V_2^\prime$, $V_3^\prime$ as in the previous definition such that $V_{\mathbb{A}}(V_1,V_2,V_3)=V_{\mathbb{A}^\prime}(V_1^\prime,V_2^\prime,V_3^\prime)$.

For this, note that by the definition of $V_{\mathbb{A}}(V_1,V_2,V_3)$ any $\mathbb{A}^\prime \in V_{\mathbb{A}}(V_1,V_2,V_3)$ can be written as \[ \mathbb{A}^\prime = \big(f(S),f\circ \pi \col P \to Z \setminus S, \{[f \circ\Upsilon_i\circ U_i,\tilde{\Upsilon}_i \circ \tilde{U}_i]\},A^\prime\big)\] where $(f,\{(U_i,\tilde{U}_i)\},A^\prime) \in V_1\times V_2 \times V_3$. Since any element in $V_1$ can be written as $(f^\prime \circ f^{-1}) \circ f$ and any element in $V_2$ can be written as $\{(U_i, \tilde{U}_i)\} \circ \{(U_i^{-1}\circ U_i^\prime, \tilde{U}_i^{-1} \circ \tilde{U}_i^\prime) \}$, we immediately have $V_{\mathbb{A}}(V_1,V_2,V_3)=V_{\mathbb{A}^\prime}(V_1^\prime,V_2^\prime,V_3^\prime)$ for $V_1^\prime \coloneqq V_1 \circ f^{-1}$, $V_2^\prime \coloneqq \{(U_i^{-1},\tilde{U}_i^{-1})\} \circ V_2$, and $V_3^\prime \coloneqq V_3$.

In order to prove that $\mathcal{C}$ is closed under intersections take two elements in $\mathcal{C}$ which are not disjoint from one another. By the previous argument we may write these two sets as $V_{\mathbb{A}}(V_1,V_2,V_3)$ and  $V_{\mathbb{A}}(V_1^\prime,V_2^\prime,V_3^\prime)$ for some $\mathbb{A} \in \mathcal{A}_\mu^\Fr(\{P_i,A_i\})$ and $V_1,V_1^\prime,V_2,V_2^\prime,V_3,V_3^\prime$ as in \autoref{def: base of topology}. Then $V_{\mathbb{A}}(V_1,V_2,V_3)\cap V_{\mathbb{A}}(V_1^\prime,V_2^\prime,V_3^\prime) = V_{\mathbb{A}}(V_1\cap V_1^\prime,V_2\cap V_2^\prime,V_3\cap V_3^\prime)$, and the result follows.
\end{proof}

\subsection{The local structure of $\mathcal{B}_\mu^\Fr(\{P_i,A_i\})$}

In this section we prove that $\mathcal{B}_\mu^{\Fr}(\{P_i,A_i\})$ (as defined in the previous section) is locally homeomorphic to an open neighbourhood in a product of the form \[(\mathbb{C}^3)^N\times (\times_i\mathfrak{m}_i) \times \big(\times_i \Stab_{\SU(3)}(A_i) \times \mathcal{A}_\mu^{\textup{Fr}}(P,\{\Upsilon_i,\Tilde{\Upsilon}_i,A_i\})\big)/\mathcal{G}_{\mu+1}.\] Here, the framed bundle $\pi \col P \to Z \setminus S$ is fixed (but depends on the neighbourhood in $\mathcal{B}_\mu^{\Fr}(\{P_i,A_i\})$), the groups $\Stab_{\SU(3)}(A_i)$ are as in \eqref{equ: definition Stab_SU(3)(A_s)}, $\mathfrak{m}_i\subset \mathfrak{su}(3)$ is a subspace complementary to the image of (the Lie algebra) $ \mathfrak{stab}_{\SU(3)}(A_i)$ under the canonical projection $\mathfrak{stab}_{\SU(3)}(A_i) \to \mathfrak{su}(3)$, and $\mathcal{G}_{\mu+1}$ is defined by
\begin{align*}
\mathcal{G}_{\mu+1} \coloneqq \big\{ g \in \mathcal{G}(P) &\mid \vert \nabla^k (\tilde{\Upsilon}_i^{-1} \circ g \circ \tilde{\Upsilon}_i-\pr_{S^5}^*\tilde{U}_i) \vert = \mathcal{O}(r^{\mu_i+1-k}) \textup{ for every $i=1,\dots,N$,}  \\ 
& \quad \textup{$k \in \mathbb{N}_0$, and a $\tilde{U}_i \in \mathcal{G}(P_i)$ that preserves $A_i$} \big\}
\end{align*}
(where we again assume that $G\subset \GL(W)$ for some finite dimensional vector space $W$).

Geometrically, the elements in $\mathbb{C}^3$ in the product above correspond to (locally) moving the singular points of $\pi \col P \to Z \setminus S$. Similarly, the elements in $\mathfrak{m}_i \subset \mathfrak{su}(3)$ 'rotate' the bundle around $s_i$ and elements in $\Stab_{\SU(3)}(A_i)$ correspond to changing the framing at $s_i$.

We begin in \autoref{sec: construction of bundle isomorphisms} by showing that bundles of the form $f_1 \circ \pi \col P \to Z\setminus f_1(S)$ and $f_2\circ \pi \col P \to Z \setminus f_2(S)$, where $f_1,f_2 \col Z \to Z$ are diffeomorphisms satisfying certain properties, may be identified in a way that is compatible with their conically singular structure. In \autoref{sec: local parametrisation of framed cs connections} we then define a suitable parametrisation of $\mathcal{B}_\mu^{\textup{Fr}}(\{P_i,A_i\})$ and prove that this parametrisation indeed defines a local homeomorphism.

\subsubsection{Construction of suitable bundle isomorphisms}\label{sec: construction of bundle isomorphisms}
 
Recall from \autoref{def: base of topology} that the (framed) bundles $\pi \col P \to Z\setminus S$ in a sufficiently small neighbourhood in $\mathcal{A}^\Fr_\mu(\{P_i,A_i\})$ are the push-forwards of a fixed bundle by certain diffeomorphisms close to the identity. When going to the quotient $\mathcal{B}^\Fr_\mu(\{P_i,A_i\})$, we first have to determine which of these framed bundles are identified by an isomorphism that respects their respective conically singular structure.

It is well-known that two diffeomorphisms $f_0,f_1 \col Z \to Z$ which are both sufficiently close to the identity and agree at the singular points $S$ give rise to isomorphic push-forward bundles $f_i \circ \pi \col P \to Z \setminus f_i(S)$. The following proposition shows that when $f_0$ and $f_1$ agree at $S$ to first order, then this bundle isomorphism may be chosen to respect their singular structure.

\begin{proposition}\label{prop: isomorphism from isotopy}
Let $S=\{s_1,\dots,s_N\}$ be a totally ordered set and $\pi \col P \to Z \setminus S$ be a bundle together with a set of framings $\{(\Upsilon_i,\Tilde{\Upsilon}_i)\}_{i=1,\dots,N}$ around each $s_i\in S$ as in \autoref{def: framed CS connections}. Assume that $f_0,f_1\col Z \to Z$ are two diffeomorphisms which are both sufficiently close with respect to the $C^1$-norm to the identity and satisfy \[(f_0)_{\vert S} = (f_1)_{\vert S} \quad \textup{and} \quad (Df_0)_{\vert S}=(Df_1)_{\vert S}.\] Then there exists an isomorphism \[F \col (f_0\circ \pi \col P \to Z\setminus f_0(S)) \to (f_1\circ \pi \col P \to Z\setminus f_1(S))\] covering the identity with the property that \[ \vert \nabla^k(\Tilde{\Upsilon}_i^{-1} \circ F \circ \Tilde{\Upsilon}_i - \textup{Id})\vert = \mathcal{O}(r^{1-k}) \] for all $k\in \mathbb{N}_0$ and every $i=1,\dots,N$.
\end{proposition}

\begin{proof}
Since $f_0$ and $f_1$ are sufficiently close with respect to the $C^1$-norm and agree up to first order at $S$, there exists an isotopy of diffeomorphisms $f \col [0,1] \times Z\to Z$ such that $(f_t)_{\vert S}$ and $(Df_t)_{\vert S}$ is constant in $t$ (see for example \cite[Proof of Theorem~6.26]{Lee-SmoothManifolds}).

We consider the principal $G$-bundle 
\begin{align*}
\pi^\prime \col [0,1] \times P &\to [0,1] \times Z\setminus f_0(S) \\
(t,p) & \mapsto (t,f_t(\pi(p))).
\end{align*}
(Note that this bundle is isomorphic to the pullback $h^*P$ where $h\col [0,1]\times Z\setminus f_0(S) \to Z\setminus S$ is the isotopy defined by $h_t\coloneqq f_t^{-1}$ for every $t\in [0,1]$.) Equip $P$ with a connection $A \in \mathcal{A}(P)$ that satisfies $\Tilde{\Upsilon}_i^*A = \pr_{S^5}^*A_i$ for every $i=1,\dots,N$. Its pullback $A^\prime \coloneqq \pr_{P}^*A$ defines a connection on $\pi^\prime \col [0,1] \times P \to [0,1] \times Z\setminus f_0(S)$. 
    
Now define $F \col P \to P$ by \[F(p) \coloneqq \pr_P(\textup{tra}_{[0,1]}^{A^\prime}(0,p))\] where $\textup{tra}_{[0,1]}^{A^\prime}(0,p)$ denotes the parallel transport of $(0,p)$ on $\pi^\prime \col [0,1]\times P \to [0,1]\times Z\setminus f_0(S)$ with respect to $A^\prime$ over the path \[[0,1] \ni t \mapsto (t,f_0(\pi(p))) \in [0,1]\times Z\setminus f_0(S).\] 

In order to prove \[ \vert \nabla^k(\Tilde{\Upsilon}_i^{-1} \circ F \circ \Tilde{\Upsilon}_i - \textup{Id})\vert = \mathcal{O}(r^{1-k}) \] we first make the following definitions: Let $\pi_i^\prime\col [0,1]\times \pr_{S^5}^*P_i \to [0,1]\times B_R(0)\setminus \{0\}$ be the pullback bundle which we equip with the pullback connection $A^\prime_i \coloneqq \pr_{P_i}^*A_i$. Furthermore, define for $z\in B_R(0)$ and $t\in[0,1]$
\begin{align*}
f_0^\prime(z) &\coloneqq (\Upsilon_i^{-1} \circ f_0 \circ \Upsilon_i)(z) \\
h_t^\prime(z) &\coloneqq (\Upsilon_i^{-1} \circ h_t \circ \Upsilon_i)(z)
\end{align*}
where $h_t =f_t^{-1}$. Then \[(\Tilde{\Upsilon}_i^{-1} \circ F \circ \Tilde{\Upsilon}_i)(p) = \pr_{\pr_{S^5}^*P_i} (\textup{tra}^{A^\prime_i}_{[0,1]}(0,p))\] where $\textup{tra}^{A^\prime_i}_{[0,1]}$ is the parallel transport with respect to $A^\prime_i$ over the path $t \mapsto (t,h_t^\prime(f_0^\prime(\pi(p))))$. Since $f_t$ agrees for all $t\in [0,1]$ with $f_0$ up to first order at $S$, we have $h_t^\prime(f_0^\prime(z)) = z + \mathcal{O}(t \vert z\vert^2)$. This implies that the $A^\prime_i$-horizontal lift of $\partial_t (t,h^\prime_t(f^\prime_0(z)))\in T_{(t,h^\prime_t(f^\prime_0(z)))}([0,1]\times B_R(0))$ to $([0,1] \times \pr_{S^5}^*P_i) \cong ([0,1]\times (0,R) \times P_i)$ satisfies 
\begin{align*}
\textup{Lift}^{A^\prime_i} \big(\partial_t ,(\partial_t h^\prime)_t(f^\prime_0(z))\big) &= \bigg(\partial_t, \partial_t \vert h^\prime_t(f^\prime_0(z)) \vert, \textup{Lift}^{A_i}\big(\partial_t \tfrac{h^\prime_t(f^\prime_0(z))}{\vert h^\prime_t(f^\prime_0(z)) \vert}\big)\bigg) \\
&=  (\partial_t,0,0) + \mathcal{O}(\vert z \vert). 
\end{align*}
Since $ \textup{tra}^{A^\prime_i}_{[0,1]}(0,p) $ is the flow of this horizontal lift, we obtain $\vert \Tilde{\Upsilon}_i^{-1} \circ F \circ \Tilde{\Upsilon}_i - \textup{Id}\vert = \mathcal{O}(r)$. In order to estimate the derivatives, we note that 
\begin{align*}
\vert \nabla(\Tilde{\Upsilon}_i^{-1} \circ F \circ \Tilde{\Upsilon}_i) \vert &\leq c \vert (\textup{tra}_{[0,1]}^{A^\prime_i}(0,\cdot))^*A_i^\prime - A_i^\prime \vert \\
& \leq c \int_{0}^1 \vert (\pr_{S^5}^*F_{A_i})(\partial_t h^\prime_t(f_0^{\prime}(\cdot)),\cdot) \vert \diff t \\
&\leq \mathcal{O}(r^{-2+2}).
\end{align*}
The higher derivatives can be estimated analogously.
\end{proof}
\begin{remark}\label{rmk: family of isomorphisms from isotopy}
Note that the isomorphism $F$ constructed in the previous proof is the time-1 flow of a time-dependent vector field $(X^H_t)_{t\in[0,1]} \in \Gamma(TP)$ over $P$ given as follows: First define the time-dependent vector field $(X_t)_{t\in[0,1]} \in \Gamma(TZ)$ for $t_{0}\in [0,1]$ and $z\in Z$ as \[X_{t_0}(z) \coloneqq -Df_{t_0}^{-1}\big(\partial_t f_t(z)_{\vert t={t_0}}\big), \] where $f_t \col Z \to Z$ is the isotopy between $f_0$ and $f_1$ used in the previous proof. Next, let $X^H_t \in \Gamma(TP)$ be the $A$-horizontal lift of $X_t$ to the bundle $\pi \col P \to Z\setminus S$ (where $A \in \mathcal{A}(P)$ is also as in the previous proof). Then $F(p) = \textup{Flow}_{[0,1]}^{X^H_t}(p)$ (where $\textup{Flow}_{[0,1]}^{X^H_t}$ denotes the time-dependent flow of $X^H_t$ that starts at $t=0$ and ends at $t=1$).

Since the solutions of ordinary differential equations depend continuously on the right-hand side, this description of $F$ as a time dependent flow implies the following strengthening of the previous proposition: Let $V$ be any topological space and $f_i \col V \times Z \to Z$ for $i=1,2$ be two continuous families of diffeomorphisms (i.e. $f_i(v,\cdot)$ is a diffeomorphism for every $v\in V$) such that all $f_i(v,\cdot)$ are sufficiently close to the identity with respect to the $C^1$-norm and satisfy \[f_0(v,\cdot)_{\vert S} = f_1(v,\cdot)_{\vert S} \quad \textup{and} \quad D(f_0(v,\cdot))_{\vert S}=D(f_1(v,\cdot))_{\vert S}.\] Furthermore, we assume that all derivatives of $f_i(v,\cdot)$ depend uniformly on $v\in V$ (in the sense that $\Vert f_i(v_n,\cdot) - f_i(v,\cdot)\VertC{k}{} \to 0$ for all $k\in \mathbb{N}$ whenever $v_n \to v$). Then the collection of bundle isomorphisms constructed in the previous proposition yields a continuous map \[ F \col V \times P \to P\] such that all derivatives of $F(v,\cdot)$ depend uniformly on $v\in V$. In fact, since all $f_i(v,\cdot)$ agree on $S$ up to first order, one can prove that $\Vert F(v_n, \cdot) - F(v,\cdot) \VertWC{k}{1}{} \to 0$ for every $k \in \mathbb{N}$, whenever $v_n \to v$, where $\Vert \cdot \VertWC{k}{1}{}$ is as in \autoref{def: C^infty_mu topology}.
\end{remark}

The next proposition addresses the situation when two diffeomorphisms $f_0,f_1 \col Z \to Z$ agree only to zeroth order at any $s\in S$ but differs at first order by a rotation in $\SU(3)$ which lies in the image of the canonical map $\Stab_{\SU(3)}(A_i) \to \SU(3)$.

\begin{proposition}\label{prop: isomorphism from stabiliser}
Assume that we are in the same situation as in \autoref{prop: isomorphism from isotopy} with the exception that we only assume that the diffeomorphisms $f_0,f_1\col Z \to Z$ agree to zeroth order at $S$. Additionally assume that there exists now a collection $\{(U_i,\tilde{U}_i)\}_{i=1,\dots,N}$ of $A_i$-preserving bundle isomorphisms $\tilde{U}_i\in \textup{Stab}_{\SU(3)}(A_i)$ (as defined in \eqref{equ: definition Stab_SU(3)(A_s)}) covering $U_i \in \SU(3)$ such that \[ D_0(\Upsilon_i^{-1} \circ f_1^{-1} \circ f_0 \circ \Upsilon_i) = U_i^{-1} \quad \textup{for every $i=1,\dots,N$.}\] Furthermore, we assume that each $\tilde{U}_i$ lies in the image of the (Lie group) exponential map on $\textup{Stab}_{\SU(3)}(A_i)$. Then there exists an isomorphism \[F \col (f_0\circ \pi \col P \to Z\setminus f_0(S)) \to (f_1\circ \pi \col P \to Z\setminus f_1(S))\] covering the identity with the property that \[ \vert \nabla^k(\Tilde{\Upsilon}_i^{-1} \circ F \circ \Tilde{\Upsilon}_i \circ \tilde{U}_i - \textup{Id})\vert = \mathcal{O}(r^{1-k}) \] for every $k\in \mathbb{N}_0$ and every $i=1,\dots,N$.
\end{proposition}

\begin{proof}
For simplicity we will assume $N=1$ and drop the subscripts to ease notation. The general case can be proven analogously. In the following we will show that there exists a diffeomorphism $f_0^\prime$ (sufficiently close to the identity) and an isomorphism \[F\col (f_0 \circ \pi \col P \to Z\setminus f_0(S)) \to (f_0 \circ f_0^\prime \circ \pi \col P \to X\setminus f_0(S))\] that satisfy
\begin{itemize}
\item $f_0^\prime(s) = s$ and $D_s(f_0\circ f_0^\prime) = D_s f_1$ (where $S=\{s\}$)
\item $\tilde{\Upsilon}^{-1} \circ F \circ \tilde{\Upsilon} = \tilde{U}^{-1}$ on $B_{R/2}(0)$.
\end{itemize}
The proposition then follows from \autoref{prop: isomorphism from isotopy} applied to $(f_0\circ f_0^\prime)$ and $f_1$.

Let $t \mapsto \tilde{U}_t\in \textup{Stab}(A_s)$ for $t\in [0,1]$ be a path connecting $\tilde{U}$ to $\textup{Id}$. Furthermore, since $\tilde{U}$ lies in the image of the exponential map, we may assume that there is a path $t \mapsto u_t\in \mathfrak{su}(3)$ such that $\tilde{U}_t$ covers $\exp(u_t)$. We then define $f_t^\prime \col \Upsilon(B_R(0))\to \Upsilon(B_R(0))$ via \[(\Upsilon^{-1}\circ f_t^\prime\circ \Upsilon)(z)) \coloneqq \exp(\chi(\vert z \vert) \cdot u_t) \cdot z \] where $\chi$ is a non-increasing cut-off function with $\chi(r)=1$ for $r\leq R/2$ and $\chi(r)=0$ for $r\geq 3R/4$. Furthermore, We extend $f_t^\prime$ to be the identity map on $Z\setminus \Upsilon(B_R(0))$.

The isomorphism $F^{-1}$ is now constructed as the isomorphism in the proof of \autoref{prop: isomorphism from isotopy}. The isotopy of diffeomorphisms is hereby given by $f_t^\prime$ and the parallel transport is taken with respect to a connection $A^\prime\in \mathcal{A}([0,1]\times P)$ that satisfies $ \tilde{\Upsilon}^*A^\prime = (\tilde{U}^\prime)^* A_s$ where \[\tilde{U}^\prime \col [0,1]\times P_s \to [0,1]\times P_s, \quad (t,p) \mapsto (t,\tilde{U}_t(p)). \qedhere \]
\end{proof}

\begin{remark}\label{rem: diffeos only contribute zeroth order and first order up to stabiliser}
The previous two propositions may be interpreted as follows: (small) open neighbourhoods in $\mathcal{A}^\Fr_\mu(\{P_i,A_i\})$ are parametrised  (among other data) by diffeomorphisms $f\col Z \to Z$ close to the identity. \autoref{prop: isomorphism from isotopy} shows that when going to the quotient $\mathcal{B}^\Fr_\mu(\{P_i,A_i\})$ one only needs to remember the behaviour of such $f$ to first order around $S$. Moreover, \autoref{prop: isomorphism from stabiliser} shows that the framed bundles $(f_0 \circ \pi \col P \to Z\setminus f_0(S),\{(\Upsilon_i \circ U_i,\tilde{\Upsilon}_i\circ \tilde{U}_i)\})$ and $(f_1 \circ \pi \col P \to Z\setminus F_1(S),\{(\Upsilon_i,\tilde{\Upsilon}_i)\})$ are isomorphic whenever $f_0$ and $f_1$ are sufficiently close, agree to zeroth order at $S$ and differ to first order by elements $U_i \in \SU(3)$ that lie in the image of the canonical map $\textup{Stab}_{\SU(3)}(A_i) \to \SU(3)$. Thus, when going to the quotient $\mathcal{B}^\Fr_\mu(\{P_i,A_i\})$, one only needs remember the zeroth order term of such a diffeomorphism $f \col Z \to Z$ at any $s_i\in S$ and the first order term at the same $s_i \in S$ up to elements in $\textup{image}(\Stab_{\SU(3)}(A_i) \to \SU(3))$.
\end{remark}

\subsubsection{A local parametrisation of $\mathcal{B}_\mu^{\textup{Fr}}(\{P_i,A_i\})$}\label{sec: local parametrisation of framed cs connections}

In this section we will define a local parametrisation of $\mathcal{B}_\mu^{\textup{Fr}}(\{P_i,A_i\})$ and prove that it is indeed a (local) homeomorphism. Recall from \autoref{def: base of topology} that open subsets in $\mathcal{A}^\Fr_\mu(\{P_i,A_i\})$ are parametrised by certain diffeomorphisms (among other data). Moreover, we have seen in the previous section (cf. \autoref{rem: diffeos only contribute zeroth order and first order up to stabiliser}) that when going to the quotient $\mathcal{B}_\mu^{\textup{Fr}}(\{P_i,A_i\})$, one only remembers the behaviour of such a diffeomorphism at every $s_i \in S$ to first order and the first order term only up to elements in $\textup{image}(\Stab_{\SU(3)}(A_i) \to \SU(3))$. In the following we will therefore define a family of diffeomorphisms that realises any fixed translation of $s_i\in S$  (as zeroth order term) and rotation around $s_i$ transverse to $\textup{image}(\Stab_{\SU(3)}(A_i) \to \SU(3))$ (as first order term). This family of diffeomorphisms will subsequently be used to parametrise small neighbourhoods in $\mathcal{B}_\mu^\Fr(\{P_i,A_i\})$. All diffeomorphisms will be the time-1 flow of the following vector fields:

\begin{proposition}\label{prop: construction of families of vector fields}
Let $S\coloneqq \{s_1,\dots,s_N\}$ be a totally ordered set and for each $i=1,\dots,N$ let $\Upsilon_i\col B_R(0) \to Z$ be a $\SU(3)$-coordinate system centered around $s_i$. Moreover, let $\mathfrak{m}\subset (\mathfrak{su}(3))^N$ be any linear subspace. There exists an $\varepsilon>0$ and three smooth maps 
\begin{align*}
\mathfrak{vec}_0 \col \mathfrak{m} &\to \Gamma(TZ)\\
\mathfrak{vec}_k \col (B_\varepsilon(0) \subset \mathbb{C}^3)^N &\to \Gamma(TZ) \quad k=1,2
\end{align*}
with the following properties:
\begin{enumerate}
\item $B_{4\varepsilon}(s_i)\cap B_{4\varepsilon}(s_j)= \emptyset$ for all $i \neq j$,
\item $\textup{supp}(\mathfrak{vec}_0(\Vec{u})) \subset \cup_i B_{4\varepsilon}(s_i)$ and $\textup{supp}(\mathfrak{vec}_k(\Vec{v})) \subset \cup_i B_{4\varepsilon}(s_i)$ for $k=1,2$ and every $\vec{u}\in \mathfrak{m}$ and $\Vec{v} \in B_{\varepsilon}(0)^N$,
\item \label{bul: properties of vec0} The family $\mathfrak{vec}_0$ satisfies the following:
\begin{itemize}
\item for each $\vec{u} = (u_1,\dots,u_N) \in \mathfrak{m}$ and $t \in [-1,1]$ we have \[ \mathfrak{vec}_0(t \cdot\vec{u}) = t\cdot \mathfrak{vec}_0(\vec{u}).\]
\item for every $\vec{u}=(u_1,\dots,u_N) \in \mathfrak{m}$ and $i\in \{1,\dots,N\}$ we have \[ \mathfrak{vec}_0(\vec{u})(s_i)=0\] and, more generally, \[ \mathfrak{vec}_0(\vec{u})_{\vert B_{2\varepsilon}(s_i)} = (\Upsilon_i)_* \hat{u}_i \quad \textup{in the neighbourhood $B_{2\varepsilon}(s_i)$ of $s_i$,} \] where $\hat{u}_i \in \Gamma(T\mathbb{C}^3)$ is the vector field induced by the infinitesimal rotation $u_i \in \mathfrak{su}(3)$ (i.e. $\hat{u}_i(z) \coloneqq u_i \cdot z \in \mathbb{C}^3 = T_z \mathbb{C}^3$). This implies \[ D_0(\Upsilon_i^{-1} \circ \textup{Flow}_1^{\mathfrak{vec}_0(\vec{u})} \circ \Upsilon_i) = \exp(u_i) \] where $\textup{Flow}^{\mathfrak{vec}_0(\Vec{u})}_t$ denotes the flow of of $\mathfrak{vec}_0(\vec{u})$ at time $t$ and $\exp(u_i) \in \SU(3)$ denotes the ordinary (matrix) exponential.
\end{itemize}
\item \label{bul: properties of vec1} The family $\mathfrak{vec}_1$ satisfies the following:
\begin{itemize}
\item for each $\Vec{v}=(v_1,\dots,v_N)\in B_\varepsilon(0)^N$ and $t \in [-1,1]$ we have \[\mathfrak{vec}_1(t \cdot \Vec{v}) = t \cdot \mathfrak{vec}_1(\Vec{v}).\]
\item for every $\vec{v} = (v_1,\dots,v_N) \in B_\varepsilon(0)^N$ and $i=1,\dots,N$ we have \[\mathfrak{vec}_1(\vec{v})_{\vert B_{2\varepsilon}(s_i)} = (\Upsilon_i)_* \hat{v}_i \quad \textup{in the neighbourhood $B_{2\varepsilon}(s_i)$ of $s_i$,}\] where $\hat{v}_i \in \Gamma(T\mathbb{C}^3)$ denotes the constant vector field in the direction of $v_i \in \mathbb{C}^3$ (i.e. $\hat{v}_i(z) = v_i \in \mathbb{C}^3 = T_z \mathbb{C}^3$ at $z \in \mathbb{C}^3$). This implies that the flow of $\mathfrak{vec}_1(\vec{v})$ satisfies \[\textup{Flow}^{\mathfrak{vec}_1(\Vec{v})}_1(s_i) = \Upsilon_i(v_i) \quad \textup{for every $i=1,\dots,N$.}\]
\end{itemize}  
\item \label{bul: properties of vec2} The family $\mathfrak{vec}_2$ satisfies the following:
\begin{itemize}
\item $\mathfrak{vec}_2(\Vec{0})=0$ where $\vec{0}=(0,\dots,0)$,
\item for each $\vec{v}=(v_1,\dots,v_N)\in (\mathbb{C}^3)^N$ and any $s_i \in S$ we have \[ (\partial_{\vec{v}}\mathfrak{vec}_2(0))(s_i) \coloneqq (D_0 \mathfrak{vec}_2)(\vec{v})(s_i) = 0\] where we regard the derivative $\partial_{\vec{v}}\mathfrak{vec}_2(0)=(D_0 \mathfrak{vec}_2)(\vec{v}) \in \Gamma(TZ)$ again as a vector field on $Z$.
\item for each $\Vec{v}=(v_1,\dots,v_N)\in B_\varepsilon(0)^N$ and $i\in \{1,\dots,N\}$ we have 
\begin{align*}
\textup{Flow}^{\mathfrak{vec}_2(\Vec{v})}_1(\Upsilon_i(v_i)) &= \Upsilon_i(v_i)
\end{align*} 
and \[ (\textup{Flow}_1^{\mathfrak{vec}_2(\Vec{v})} \circ \textup{Flow}_1^{\mathfrak{vec}_1(\Vec{v})})^*(\omega,\Omega)_{\Upsilon_i(v_i)}= (\omega,\Omega)_{s_i}.\]
\end{itemize}
\end{enumerate}
\end{proposition}
\begin{remark}
The time-1 flow of these vector fields gives rise to a fixed family of diffeomorphisms. Their respective roles are as follows: $\textup{Flow}_1^{\mathfrak{vec}_0(\vec{u})}$ gives rise to a rotation by $\exp(u_i) \in \SU(3)$ around any $s_i \in S$. The flow $\textup{Flow}_1^{\mathfrak{vec}_1(\vec{v})}$ translates $s_i$ by $v_i \in \mathbb{C}^3$, and $\textup{Flow}_1^{\mathfrak{vec}_2(\vec{v})}$ ensures that the translated coordinate system $\textup{Flow}_1^{\mathfrak{vec}_2(\vec{v})}\circ \textup{Flow}_1^{\mathfrak{vec}_1(\vec{v})} \circ \Upsilon_i$ centered at the translated singular point $\textup{Flow}_1^{\mathfrak{vec}_1(\vec{v})}(s_i) = \Upsilon_i(v_i)$ still pulls-back $(\omega,\Omega)$ at $(\textup{Flow}_1^{\mathfrak{vec}_1(\vec{v})}(s_i))$ to the flat $\SU(3)$-structure $(\omega_0,\Omega_0)$ on $\mathbb{C}^3$.
\end{remark}
The existence of such vector fields is well-known (see, for example, \cite[Theorem~5.2]{Joyce-Moduli_of_cs-slag}, \cite[Section~6.2]{Lotay-cs_coassociatives}, \cite[Section~4.3]{Englebert-cs_cayleys}, or \cite[Definition~5.8]{Bera-cs_associatives}). For the convenience of the reader we have included a proof.
\begin{proof}
For simplicity we assume $N=1$ and drop the subscripts so that $S = \{s\}$ together with the $\SU(3)$-coordinate system $\Upsilon\col B_R(0) \to Z$ around $s$. The proof for a general $N$ is analogous. 
    
For $\varepsilon < R/3$ we define $\mathfrak{vec}^\prime_1 \col B_\varepsilon(0) \to \Gamma(TB_R(0))$ where the vector field $\mathfrak{vec}^\prime_1(v)\in \Gamma(TB_R(0))$ for $v\in B_{\varepsilon}(0)$ at the point $z\in B_R(0)$ is given by $ (\mathfrak{vec}^\prime_1(v))(z) \coloneqq \chi(\vert z \vert) \cdot \hat{v}$ where $\chi$ is a fixed non-increasing cut-off function with $\chi(r)=1$ for $r \leq 2 \varepsilon$ and $\chi(r)= 0$ for $r>3\varepsilon$ and where $\hat{v}\in \Gamma(T\mathbb{C}^3)$ denotes the constant vector field $\hat{v}(z) = v \in \mathbb{C}^3=T_z\mathbb{C}^3$ for every $z\in \mathbb{C}^3$. The map $\mathfrak{vec}_1$ is then defined by $\Upsilon_* \circ \mathfrak{vec}^\prime_1$ on $\Upsilon(B_R(0))$ and extended by zero outisde of $\Upsilon(B_R(0))$.

Pulling $(\omega,\Omega)$ back via $\Upsilon$ gives rise to a smooth map $B_R(0) \to \Lambda^2 (\mathbb{C}^3) \oplus \Lambda^3_\mathbb{C}(\mathbb{C}^3)$ with $\Upsilon^*(\omega,\Omega)_p=(\omega_0,\Omega_0)$. The Implicit Function Theorem implies (after possibly shrinking $\varepsilon$) that there exists a smooth map $A \col B_\varepsilon(0) \to \mathfrak{gl}_{\mathbb{R}}(\mathbb{C}^3)$ (into the space of real $6\times 6$-matrices) with $A(0)=0$ such that \[ \exp(-A(v))^*(\omega_0,\Omega_0) = \Upsilon^*(\omega,\Omega)_{\Upsilon(v)}\] for every $v\in B_{\varepsilon}(0)$. 

As above, we first define $\mathfrak{vec}_2^\prime \col B_{\varepsilon}(0) \to \Gamma(TB_R(0))$, where the vector field $\mathfrak{vec}_2^\prime(v)\in \Gamma(TB_R(0))$ for $v\in B_{\varepsilon}(0)$ at the point $z\in B_R(0)$ is given by \[(\mathfrak{vec}_2^\prime(v))(z) \coloneqq \chi(\vert z \vert) \cdot (\tfrac{\diff}{\diff t} (\exp(tA(v)) \cdot (z-v)+v)_{\vert t=0}) \in T_zB_R(0). \] The map $\mathfrak{vec}_2$ is then again defined as the push-forward $\Upsilon_*\circ \mathfrak{vec}_2^\prime$ on $\Upsilon(B_R(0))$ and extended by zero outside of $\Upsilon(B_R(0))$).

\noindent The third map $\mathfrak{vec}_0 \col \mathfrak{m} \to \Gamma(TZ)$ is constructed analogously.

We will only verify the property \[ (\textup{Flow}_1^{\mathfrak{vec}_2(\Vec{v})} \circ \textup{Flow}_1^{\mathfrak{vec}_1(\Vec{v})})^*(\omega,\Omega)_{\Upsilon_i(v_i)}= (\omega,\Omega)_{s_i}.\]
For this, note that the flow at time $t\in[0,1]$ of $\mathfrak{vec}_2^\prime(v)\in \Gamma(TB_R(0))$ at $z\in B_{R}(0)$ with $\vert z -v \vert$ sufficiently small is given by \[ \textup{Flow}_t^{\mathfrak{vec}_2^\prime(v)}(z) = \exp(t  A(v)) \cdot (z-v) + v \] where $\exp$ in this context denotes the ordinary matrix-exponential. Thus, by the construction of $A(v)$ \[ (\textup{Flow}_1^{\mathfrak{vec}_2^\prime(v)})^*(\Upsilon^*(\omega,\Omega)_{\Upsilon(v)}) = (\omega_0,\Omega_0)_v\] which together with $(\textup{Flow}^{\mathfrak{vec}_1^{\prime}(v)}_1)^*(\omega_0,\Omega_0)_{v} = (\omega_0,\Omega_0)_0$ and $\Upsilon^*(\omega,\Omega)_{s}=(\omega_0,\Omega_0)_0$ implies the claim.
\end{proof}

\begin{definition}\label{def: family of diffeomorphism parametrising open neighbourhood}
Let $\mathfrak{m}_i \subset \mathfrak{su}(3)$ for every $i=1,\dots,N$ be a complement of the image of $\mathfrak{stab}_{\SU(3)}(A_i)$ (the Lie algebra to the Lie group defined in ~\eqref{equ: definition Stab_SU(3)(A_s)}) in $\mathfrak{su}(3)$ under the canonical projection $\mathfrak{stab}_{\SU(3)}(A_i) \to \mathfrak{su}(3)$. Moreover, let $\varepsilon>0$ be as in the previous proposition. For any  $\vec{v} = (v_1,\dots,v_N) \in (B_\varepsilon(0))^N$ and any $\vec{u}=(u_1,\dots,u_N)\in (\times_i \mathfrak{m}_i)$, we denote by $f_{\vec{v},\vec{u}}\col Z \to Z$ the diffeomorphism given by \[f_{\Vec{v},\vec{u}} \coloneqq \textup{Flow}_1^{\mathfrak{vec}_2(\Vec{v})} \circ \textup{Flow}_1^{\mathfrak{vec}_1(\Vec{v})}\circ \textup{Flow}_1^{\mathfrak{vec}_0(\Vec{u})},\] where $\mathfrak{vec}_k$ for $k=0,1,2$ are the families of vector fields constructed in the previous proposition.
\end{definition}

\begin{definition}\label{def: local structure of B definition of Psi}
Assume that $\mathbb{A}\coloneqq (S,\pi \col P \to Z \setminus S, \{[\Upsilon_i,\tilde{\Upsilon}_i]\},A) \in \mathcal{A}_\mu^{\textup{Fr}}(\{P_i,A_i\})$ is a fixed element. Moreover, let $\mathfrak{m}_i \subset \mathfrak{su}(3)$, $\varepsilon>0$, and $f_{\vec{v},\vec{u}} \col Z \to Z$ for $\vec{v} \in (B_\varepsilon(0))^N$ and $\vec{u}\in (\times \mathfrak{m}_i)$ be as in the previous definition (where all vector fields $\mathfrak{vec}_k$ are constructed with respect to $\{\Upsilon_i\}$ around $s_i \in S$). In the following we will denote by $B_\varepsilon^{\mathfrak{m}_i}(0) \subset \mathfrak{m}_i$ the $\varepsilon$-ball around $0$ in $\mathfrak{m}_i$. With these notions at hand, we define the following map:
\begin{align*}
\Psi_{\mathbb{A}} \col (B_{\varepsilon}(0))^N \times \big(\times_i B_{\varepsilon}^{\mathfrak{m}_i}(0)\big) \times (\times_i \textup{Stab}_{\SU(3)}(A_i)) \times \mathcal{A}_\mu^{\textup{Fr}}(P,\{\Upsilon_i,\tilde{\Upsilon}_i,A_i\}) \to \mathcal{A}_\mu^{\textup{Fr}}(\{P_i,A_i\})\\
\big(\vec{v},\vec{u}, \{(U_i,\tilde{U}_i)\},A^\prime\big) \mapsto \big(f_{\Vec{v},\vec{u}}(S), f_{\Vec{v},\vec{u}} \circ \pi \col P \to Z\setminus f_{\vec{v},\vec{u}}(S), \{[f_{\vec{v},\vec{u}}\circ \Upsilon_i \circ U_i,\Tilde{\Upsilon}_i\circ \tilde{U}_i]\}, A^\prime\big)
\end{align*}
\end{definition}

The following is the main result of this section.

\begin{theorem}\label{thm: local structure of B Psi is homeo}
Let $\mathbb{A}\coloneqq (S,\pi \col P \to Z \setminus S, \{[\Upsilon_i,\tilde{\Upsilon}_i]\},A) \in \mathcal{A}_\mu^{\textup{Fr}}(\{P_i,A_i\})$ be a fixed element and $\Psi_{\mathbb{A}}$ be as in the previous definition. Furthermore, let $\mathfrak{q}\col \mathcal{A}_\mu^{\textup{Fr}}(\{P_i,A_i\}) \to \mathcal{B}_\mu^{\textup{Fr}}(\{P_i,A_i\})$ be the quotient map. Then $\mathfrak{q}\circ \Psi_\mathbb{A}$ descends to a map $\overline{\mathfrak{q} \circ \Psi_\mathbb{A}}$ on the quotient
\begin{align*}
(B_{\varepsilon}(0))^N &\times \big(\times_i B_\varepsilon^{\mathfrak{m}_i}(0)\big) \times \big(\times_i \textup{Stab}_{\SU(3)}(A_i) \times \mathcal{A}_\mu^{\textup{Fr}}(P,\{\Upsilon_i,\tilde{\Upsilon}_i,A_i\})/\mathcal{G}_{0,\mu+1})\big) \big/\big(\mathcal{G}_{\mu+1}/\mathcal{G}_{0,\mu+1}\big) \\
&\to \mathcal{B}_\mu^{\textup{Fr}}(\{P_i,A_i\}).
\end{align*} 
Here,
\begin{align*}
\mathcal{G}_{0,\mu+1} \coloneqq \big\{ g \in \mathcal{G}(P) &\mid \vert \nabla^k (\tilde{\Upsilon}_i^{-1} \circ g \circ \tilde{\Upsilon}_i-\textup{Id}) \vert = \mathcal{O}(r^{\mu_i+1-k}) \textup{ for every $i=1,\dots,N$} \\
& \qquad \qquad \qquad \qquad \qquad\qquad\qquad\qquad\qquad\qquad\quad \qquad \textup{and $k \in \mathbb{N}_0$}\big\} \\
\mathcal{G}_{\mu+1} \coloneqq \big\{ g \in \mathcal{G}(P) &\mid \vert \nabla^k (\tilde{\Upsilon}_i^{-1} \circ g \circ \tilde{\Upsilon}_i-\pr_{S^5}^*\tilde{U}_i) \vert = \mathcal{O}(r^{\mu_i+1-k}) \textup{ for every $i=1,\dots,N$,} \\
& \quad \textup{$k \in \mathbb{N}_0$, and a $\tilde{U}_i \in \mathcal{G}(P_i)$ that preserves $A_i$} \big\}
\end{align*} 
(where $\mathcal{G}(P) = \Gamma(P\times_G G) \subset \Gamma(P \times_G\End(W))$ for $G \subset \GL(W)$ denotes the group of bundle-isomorphisms) and where $g\in \mathcal{G}_{\mu+1}$ acts on \[\big(\{(U_i,\tilde{U}_i)\},A\big) \in \big(\times_i \Stab_{\SU(3)}(A_i)\big) \times \mathcal{A}_\mu^{\textup{Fr}}(P,\{\Upsilon_i,\tilde{\Upsilon}_i,A_i\}) \] via \[ \big(\{(U_i, {\textstyle{\lim_{\tilde{\Upsilon}_i}(g)^{-1}}} \circ \tilde{U}_i)\},g^*A\big), \] where the asymptotic limit $\lim_{\tilde{\Upsilon}_i}(g) \in \mathcal{G}(P_i)$ at $s_i \in S$ is defined by $\vert \tilde{\Upsilon}_i^{-1} \circ g \circ \tilde{\Upsilon}_i-\lim_{\tilde{\Upsilon}_i}(g) \vert = \mathcal{O}(r^{\mu_i+1}).$ Moreover, if $\varepsilon>0$ in the definition of $\Psi_{\mathbb{A}}$ is sufficiently small, then $\overline{\mathfrak{q} \circ \Psi_\mathbb{A}}$ is a homeomorphism onto an open subset of $\mathcal{B}^{\textup{Fr}}_{\mu}(\{P_i,A_i\})$.
\end{theorem}

The proof of the previous theorem consists of two steps: First, we show that $\mathfrak{q} \circ \Psi_{\mathbb{A}}$ is an open map and second, that it descends to an injection $\overline{\mathfrak{q} \circ \Psi_{\mathbb{A}}}$ once $\varepsilon>0$ is sufficiently small. The following two results serve as preparation of the first step. Note, that as an alternative approach, one could simply \textit{define} the topology on $\mathcal{B}_{\mu}^{\textup{Fr}}(\{P_i,A_i\})$ via the (injective) map $\overline{q \circ \Psi_{\mathbb{A}}}$ (cf. \cite[Paragraph below Definition~5.4]{Joyce-Moduli_of_cs-slag} or \cite[Definition~5.12]{Bera-cs_associatives}) and then use \autoref{prop: isomorphism from isotopy} and \autoref{prop: isomorphism from stabiliser} to argue why this is a reasonable choice for a topology. The reader may therefore prefer to skip the proof of the openness of $\mathfrak{q} \circ \Psi_{\mathbb{A}}$ and go directly to the proof of the injectivity of $\overline{\mathfrak{q} \circ \Psi_{\mathbb{A}}}$.

We begin with the following proposition whose proof is left to the reader.
\begin{proposition}
Let $\mathbb{A}\coloneqq(S, \pi \col P \to Z \setminus S, \{[\Upsilon_{i},\Tilde{\Upsilon}_{i}]\}, A)$ be an element in $\mathcal{A}_\mu^{\textup{Fr}}(\{P_i,A_i\})$ and let $V_{\mathbb{A}}(V_1,V_2,V_3) \in \mathcal{C}$ be an element in the neighbourhood basis of $\mathbb{A}$ (associated to open subsets $V_1$, $V_2$, $V_3$ as in \autoref{def: base of topology}). Assume further that $\mathbb{A}^\prime \coloneqq (S^\prime, \pi^\prime \col P^\prime \to Z \setminus S^\prime, \{[\Upsilon_{i}^\prime,\Tilde{\Upsilon}_{i}^\prime]\}, A^\prime)$ is another element of $\mathcal{A}_\mu^{\textup{Fr}}(\{P_i,A_i\})$ with $S^\prime=S$ (as totally ordered sets) and $D_{0}(\Upsilon_{i}^{-1}\circ \Upsilon_{i}^\prime)= \textup{Id}$ for every $i=1,\dots, N$. Moreover, let $F \col P^\prime \to P$ be an isomorphism (covering the identity) that satisfies
\begin{itemize}
\item $\big\vert \nabla_{\pr^*_{S^5}A_{i}}^k \big( \Tilde{\Upsilon}_{i}^{-1} \circ F \circ \Tilde{\Upsilon}_{i}^\prime - \textup{Id} \big) \big\vert = \mathcal{O}(r^{\mu_i+1-k})$ for every $k\in \mathbb{N}_0$,
\item $F^*A=A^\prime$.
\end{itemize}
Define
\begin{align*} 
F^*V_{\mathbb{A}}(V_1,V_2,V_3) \coloneqq \Big\{\! \big(S^{\prime\prime}, f \circ \pi^\prime \col P^\prime \to Z \setminus S^{\prime\prime}, \{[f \circ \Upsilon_i^\prime \circ U_i, \Tilde{\Upsilon}_i^\prime \circ \tilde{U}_i]\}, F^*A^{\prime\prime}\big)\! \in\! \mathcal{A}_\mu^{\textup{Fr}}(\{P_i,A_i\}) \Big\vert&  \\
\textup{ where $ \big(S^{\prime\prime}, f \circ \pi \col P \to Z \setminus S^{\prime\prime}, \{[f \circ \Upsilon_i \circ U_i,\Tilde{\Upsilon}_i\circ \tilde{U}_i]\}, A^{\prime\prime})\in V_{\mathbb{A}}(V_1,V_2,V_3) $} \Big\}&
\end{align*}
(where $(U_i,\tilde{U}_i) \in \Stab_{\SU(3)}(A_i)$ as in \autoref{def: base of topology}). Then $F^*V_{\mathbb{A}}(V_1,V_2,V_3) \in \mathcal{C}$.
\end{proposition}

\begin{corollary}\label{cor: quotient map is open}
The quotient map $\mathfrak{q}\col \mathcal{A}_\mu^{\textup{Fr}}(\{P_i,A_i\}) \to \mathcal{B}_\mu^{\textup{Fr}}(\{P_i,A_i\})$ is an open map.
\end{corollary}
\begin{proof}
Let $V \subset \mathcal{A}_\mu^{\textup{Fr}}(\{P_i,A_i\})$ be an element of $\mathcal{C}$. Then $\mathfrak{q}^{-1}(\mathfrak{q}(V)) = \cup_F F^*V$ where we take the union over all isomorphisms $F$ that satisfy the requirements in the equivalence relation in $\mathcal{B}_\mu^{\textup{Fr}}(\{P_i,A_i\})$. Since $\mathcal{C}$ is a basis of the topology on $\mathcal{A}_\mu^{\textup{Fr}}(\{P_i,A_i\})$, the result follows.
\end{proof}

\begin{proof}[{Proof of \autoref{thm: local structure of B Psi is homeo}}]
In order to ease the notation we will again assume that $N=1$ and $S= \{s\}$ and drop all subscripts. Furthermore, we will assume that $\mathfrak{m}_s=\mathfrak{su}(3)$. The general case is similar and uses \autoref{prop: isomorphism from stabiliser}.
    
It is clear that $\Psi_{\mathbb{A}}$ and therefore also $\mathfrak{q} \circ \Psi_\mathbb{A}$ are continuous. Next, we prove that $\mathfrak{q} \circ \Psi_\mathbb{A}$ is an open map. For this, let $V_1 \subset B_{\varepsilon}(0)$, $V_2 \subset B_\varepsilon^{\mathfrak{m}_s}(0)$, $V_3 \subset \textup{Stab}_{\SU(3)}(A_s)$, and $V_4 \subset \mathcal{A}_\mu^{\textup{Fr}}(P,\{\Upsilon_s,\tilde{\Upsilon}_s,A_s\})$ be open subsets. We define the subset \[V^\prime \subset \{ f \in \textup{Diff}(Z) \mid f^*(\omega,\Omega)_{f(s)}=(\omega,\Omega)_s\} \] to consist of all diffeomorphisms $f\col Z \to Z$ whose $C^1$-norm is sufficiently close to the identity and which satisfy \[ v_f \coloneqq \Upsilon^{-1}(f(s)) \in V_1 \subset B_\varepsilon(0) \quad \textup{and} \quad D_0(\Upsilon^{-1} \circ f_{v_f,0}^{-1} \circ f \circ \Upsilon) \in \exp(V_2) \subset \SU(3)\] where $f_{v_f,0}$ is the diffeomorphism constructed in \autoref{def: family of diffeomorphism parametrising open neighbourhood}. Then $V^\prime$ is open (with respect to the subspace topology). Furthermore, 
\begin{align*}
V^\prime \times Z &\to Z \\
(f,x) & \mapsto f(x) \\
(f,x) & \mapsto f_{u_f,v_f}(x),
\end{align*}
where $u_f \in V_2$ satisfies $\exp(u_f)= D_0(\Upsilon^{-1} \circ f_{v_f,0}^{-1} \circ f \circ \Upsilon)$, are two continuous families of diffeomorphisms on $Z$ such that $f$ and $f_{u_f,v_f}$ are both sufficiently close to the identity and satisfy $f(s)=f_{u_f,v_f}(s)$ and $D_sf = D_sf_{u_f,v_f}$. \autoref{prop: isomorphism from isotopy} (see also \autoref{rmk: family of isomorphisms from isotopy}) gives therefore rise to a continuous map \[F \col V^\prime \times P \to P \] where for each $f\in V^\prime$ \[F_f\col (f_{u_f,v_f} \circ \pi\col P \to Z\setminus f_{u_f,v_f}(S)) \to (f \circ \pi\col P \to Z\setminus f(S)) \] is a bundle isomorphism compatible with $\{f_{u_f,v_f} \circ\Upsilon,\Tilde{\Upsilon}\}$ and $\{f \circ\Upsilon,\Tilde{\Upsilon}\}$. Moreover, all derivatives of $F_f$ depend (weighted) uniformly on $f$ (in the sense of \autoref{rmk: family of isomorphisms from isotopy}). We then obtain a continuous map
\begin{align*}
H \col V^\prime \times \mathcal{A}_\mu^{\textup{Fr}}(P,\{\Upsilon_s,\tilde{\Upsilon}_s,A_s\}) &\to \mathcal{A}_\mu^{\textup{Fr}}(P,\{\Upsilon_s,\tilde{\Upsilon}_s,A_s\}) \\
(f,A^\prime) \mapsto (F_f)^*A^\prime.
\end{align*} 
The set \[ \cup_{f\in V^\prime} (f,(F_f)_*V_4) = H^{-1}(V_4) \subset \{ f \in \textup{Diff}(Z) \mid f^*(\omega,\Omega)_{f(p)}=(\omega,\Omega)_p\} \times \mathcal{A}_\mu^{\textup{Fr}}(P,\{\Upsilon_s,\tilde{\Upsilon}_s,A_s\}) \] is therefore open and so is 
\begin{align*}
\Tilde{V} \coloneqq & \big\{\big(f(S),f \circ \pi \col P \to Z\setminus f(S), \{[\Upsilon\circ U_i,\Tilde{\Upsilon}\circ \tilde{U}_i]\},A^\prime\big) \ \big\vert\ \\
& \qquad \qquad \qquad \quad \{(U_i,\tilde{U}_i)\} \in V_3 \textup{ and } (f,A^\prime) \in H^{-1}(V_4) \big\} \subset \mathcal{A}^{\textup{Fr}}_\mu(\{P_s,A_s\}). 
\end{align*}
Furthermore, since by \autoref{cor: quotient map is open} $\mathfrak{q}$ is an open map,
\begin{align*}
(\mathfrak{q}\circ\Psi_\mathbb{A})(V_1\times V_2\times V_3\times V_4) &= \mathfrak{q} (\Tilde{V}) \subset \mathcal{B}_\mu^{\textup{Fr}}(\{P_s,A_s\})
\end{align*}
is open. The composition $\mathfrak{q}\circ \Psi_\mathbb{A}$ is therefore a continuous open map. In order to finish the proof, we are left to show that it descends to an injection $\overline{\mathfrak{q} \circ \Psi_\mathbb{A}}$ on the quotient
\begin{align*}
B_{\varepsilon}(0) &\times B_\varepsilon^{\mathfrak{m}_s}(0) \times \big(\textup{Stab}_{\SU(3)}(A_i) \times \mathcal{A}_\mu^{\textup{Fr}}(P,\{\Upsilon_i,\tilde{\Upsilon}_i,A_i\})\big) \big/ \mathcal{G}_{\mu+1} \\
= B_{\varepsilon}(0) &\times B_\varepsilon^{\mathfrak{m}_s}(0) \times \big(\textup{Stab}_{\SU(3)}(A_i) \times \mathcal{A}_\mu^{\textup{Fr}}(P,\{\Upsilon_i,\tilde{\Upsilon}_i,A_i\})/\mathcal{G}_{0,\mu+1})\big) \big/\big(\mathcal{G}_{\mu+1}/\mathcal{G}_{0,\mu+1}\big).
\end{align*} 

Assume therefore that $(\mathfrak{q} \circ \Psi_\mathbb{A})(v_1,u_1,(U_1,\tilde{U}_1),A^\prime_1) = (\mathfrak{q} \circ \Psi_\mathbb{A})(v_2,u_2,(U_2,\tilde{U}_2),A^\prime_2)$ for two $(v_i,u_i,(U_i,\tilde{U}_i),A^\prime_i) \in B_{\varepsilon}(0) \times B_{\varepsilon}^{\mathfrak{m}_s}(0) \times \textup{Stab}_{\SU(3)}(A_s) \times \mathcal{A}_\mu^{\textup{Fr}}(P,\{\Upsilon_s,\tilde{\Upsilon}_s,A_s\})$. Direct inspection of the equivalence relation divided out in the definition of $\mathcal{B}_\mu^{\textup{Fr}}(\{P_i,A_i\})$ (cf. \autoref{def: moduli space of framed connections}) implies \[f_{v_1,u_1}(S) = f_{v_2,u_2}(S) \quad \textup{and} \quad U_1^{-1} \circ  D_0(\Upsilon_s^{-1} \circ f_{v_1,u_1}^{-1} \circ f_{v_2,u_2} \circ \Upsilon_s) \circ U_2 = \id. \] The first point immediately leads to $v_1=v_2$. Since we have assumed that $\mathfrak{m}_s = \mathfrak{su}(3)$, we have that $U_1,U_2 \in \textup{image}(\Stab_{\SU(3)}(A_s) \to \SU(3)) \subset \SU(3)$ lie in a discrete subgroup. Thus, by choosing $\varepsilon>0$ sufficiently small, we must have (because $D_0(\Upsilon_s^{-1} \circ f_{v_1,u_1}^{-1} \circ f_{v_2,u_2} \circ \Upsilon_s) \in \SU(3)$ lies in a $2\varepsilon$-neighbourhood of $\id$) $U_1=U_2$ and therefore $u_1=u_2$.\footnote{If $\mathfrak{m}_s \neq \mathfrak{su}(3)$, one needs to use the fact that $\mathfrak{m}_s$ lies transverse to $\textup{image}(\mathfrak{stab}_{\SU(3)}(A_s) \to \mathfrak{su}(3))$.}

By the definition of the equivalence relation divided out in $\mathcal{B}^{\textup{Fr}}_{\mu}(\{P_s,A_s\})$, there exists a gauge transformation $F \col P \to P$ that satisfies
\begin{itemize}
\item $\big\vert \nabla^k \big(\Tilde{\Upsilon}_{s}^{-1} \circ F \circ \Tilde{\Upsilon}_{s} - \tilde{U}_1 \circ\tilde{U}_2^{-1} \big) \big\vert =\big\vert \nabla^k \big(\tilde{U}_1^{-1} \circ \Tilde{\Upsilon}_{s}^{-1} \circ F \circ \Tilde{\Upsilon}_{s} \circ \tilde{U}_2 - \textup{Id} \big) \big\vert = \mathcal{O}(r^{\mu_s+1-k})$ for every $k\in \mathbb{N}_0$, and
\item $F^*A^\prime_1=A^\prime_2$.
\end{itemize}
The first point implies that $F \equiv g \in \mathcal{G}_{\mu+1}$ and the proof follows.
\end{proof}

\begin{remark}
Taking the asymptotic limit $\lim_{\tilde{\Upsilon}_i}$ at each $s_i\in S$ embeds the group $\mathcal{G}_{\mu+1}/\mathcal{G}_{0,\mu+1}$ into the product $\times_i \Stab_{\mathcal{G}(P_i)}(A_i)$ consisting for each $i=1,\dots,N$ of gauge transformations $P_i \to P_i$ fixing $A_i$. If the center $Z(G)$ of the structure group is finite and all tangent cones $A_i$ are infinitesimally irreducible, then $\times_i \Stab_{\mathcal{G}(P_i)}(A_i)$ is discrete. Moreover, the action of $\mathcal{G}_{\mu+1}/\mathcal{G}_{0,\mu+1}$ on $\times_i \Stab_{\SU(3)}(A_i)$ is free. Thus, if $V \subset \times_i \Stab_{\SU(3)}(A_i)$ is a sufficiently small neighbourhood of the identity element, then \[ \overline{\mathfrak{q} \circ \Psi_\mathbb{A}} \col (B_{\varepsilon}(0))^N \times (\times_i B_\varepsilon^{\mathfrak{m}_i}(0)) \times V \times (\mathcal{A}_\mu^{\textup{Fr}}(P,\{\Upsilon_i,\tilde{\Upsilon}_i,A_i\})/\mathcal{G}_{0,\mu+1}) \to \mathcal{B}_\mu^{\textup{Fr}}(\{P_i,A_i\})\] is a homeomorphism onto an open subset.
\end{remark}

\section{The moduli space of (unframed) conically singular connections} 
\label{sec: space of (unframed) cs connections}

In this section, we consider the moduli space of conically singular connections in which the (ungeometric choice of) framing is removed from the collected data. We begin by giving the analogues of \autoref{def: moduli space of framed connections} and \autoref{def: base of topology} for unframed bundles. There exists a canonical $\times_i \Stab_{\SU(3)}(A_i)$-action on the space $\mathcal{B}_{\mu}^{\textup{Fr}}(\{P_i,A_i\})$ and we will prove in \autoref{prop: unframed moduli is orbit space of framed moduli} that the moduli space of unframed conically singular connections is homeomorphic to the orbit space of this action. This will directly lead to the analogue of \autoref{thm: local structure of B Psi is homeo} for unframed connections.

As in the previous sections, $Z^6$ is a compact 6-manifold with an $\SU(3)$-structure $(\omega,\Omega)$ and $G$ is a compact Lie group whose Lie algebra $\mathfrak{g}$ has been equipped with an $\Ad$-invariant inner product.

\begin{definition}[{\textbf{Moduli space of (unframed) connections}}]\label{def: spaces of unframed connections}
Let $N\in \mathbb{N}$ be the number of singular points, $\mu \coloneqq \{ \mu_i\}_{i \in \{1,\dots, N\}}$ for $\mu_i \in (-1,0)$ a set of rates, and $\{ (\pi_i \col P_i \to S^5,A_i) \}_{i \in \{1,\dots, N\}}$ be a set of tangent cones. We define the following:
\begin{enumerate}
\item Let $\mathcal{A}_\mu(\{P_i,A_i\})$ be the set consisting of elements of the form \[(S, \pi \col P \to Z\setminus S, A),\] where 
\begin{itemize}
\item $S= \{s_1,\dots,s_N\} \subset Z$ is a totally ordered subset,
\item $\pi \col P \to Z \setminus S$ is a principal $G$-bundle,
\item $A \in \mathcal{A}_\mu(P,\{P_i,A_i\})$, where $\mathcal{A}_\mu(P,\{P_i,A_i\})$ denotes the set of all conically singular connections on the fixed bundle $P$ with tangent connection $A_i$ (cf. \autoref{def: conically singular connections on fixed bundle}).
\end{itemize}
\item Let $\mathcal{B}_\mu(\{P_i,A_i\}) \coloneqq \mathcal{A}_{\mu}(\{P_i,A_i\})/ \sim$, where the equivalence relation $\sim$ is defined by \[(S,\pi\col P \to Z\setminus S, A) \sim (S^\prime,\pi^\prime\col P^\prime \to Z\setminus S^\prime, A^\prime)\]
if $S = S^\prime$ (as totally ordered sets) and there exists an isomorphism $F \col P^\prime \to P$ covering the identity that satisfies $F^*A=A^\prime$.
\end{enumerate}
\end{definition}
\begin{remark}
Assume that \[(S,\pi\col P \to Z\setminus S, A) \sim (S^\prime,\pi^\prime\col P^\prime \to Z\setminus S^\prime, A^\prime)\] via the isomorphism $F \col P^\prime \to P$ and that $(\Upsilon_i,\tilde{\Upsilon}_i)$ and $(\Upsilon^\prime_i,\tilde{\Upsilon}_i^\prime)$ are two framings of $P$ and $P^\prime$ at any $s_i \in S$, respectively, for which $A$ and $A^\prime$ satisfy~\eqref{equ: decay condition on framing}. \autoref{prop: change of framing} then implies that there exists a bundle isomorphism $\Tilde{U}_i\col P_{i} \to P_{i}$ covering $U_i\coloneqq D_0(\Upsilon_{i}^{-1} \circ \Upsilon_{i}^\prime)\in \SU(3)$ that preserves $A_{i}$ such that \[\big\vert \nabla_{\pr^*_{S^5}A_i}^k \big( \Tilde{\Upsilon}_{i}^{-1} \circ F \circ \Tilde{\Upsilon}_{i}^\prime - \Tilde{U}_i \big) \big\vert = \mathcal{O}(r^{\mu_i+1-k}) \quad \textup{for every $k\in \mathbb{N}_0$}. \]
\end{remark}

\begin{remark}
Once we topologies the spaces in the previous definition, they will again be non-connected. If one wishes to restrict to a connected component, then one should additionally assume that all bundles in the previous definition above are isomorphic to a fixed bundle via an isomorphism that covers a diffeomorphism that is isotopic to the identity.
\end{remark}

\begin{remark}\label{rem: full moduli theory for unframed connections}
We again take the singular set $S$ to be totally ordered because we want to prescribe the tangent cone $(P_i,A_i)$ at each singularity $s_i\in S$ in advance. As discussed in \autoref{rem: totally ordered singular set framed case}, we note once more that the previous definition is only a first step towards a 'full' moduli theory of conically singular connections (and instantons) in which the tangent cone at each singularity is a variable piece of data.
\end{remark}

We now equip these spaces with a topology analogously to \autoref{sec: moduli space and its topology}. An alternative approach would be to simply use \autoref{prop: unframed moduli is orbit space of framed moduli} to \textit{define} the topology on $\mathcal{A}_\mu(\{P_i,A_i\})$ and $\mathcal{B}_\mu(\{P_i,A_i\})$.

\begin{definition}\label{def: base of topology for unframed connections}
For fixed rates $\mu = \{\mu_i\}_{i=1,\dots,N}$ and tangent cones $\{(P_i,A_i)\}_{i=1,\dots,N}$, let $\mathcal{A}_\mu(\{P_i,A_i\})$ and $\mathcal{B}_\mu(\{P_i,A_i\})$ be as in \autoref{def: spaces of unframed connections}. We first define the following collection $\mathcal{C}$ of subsets of $\mathcal{A}_\mu(\{P_i,A_i\})$ which will subsequently serve as the basis for a topology.

Let $\mathbb{A} \coloneqq (S, \pi \col P \to Z\setminus S, A) \in \mathcal{A}_\mu(\{P_i,A_i\})$ be any element and let $\{(\Upsilon_i,\tilde{\Upsilon}_i)\}_{i\in \{1,\dots,N\}}$ be any set of framings such that $A \in \mathcal{A}_\mu^\Fr(P,\{\Upsilon_i,\tilde{\Upsilon}_i,A_i\})$. Furthermore, assume that we have have chosen
\begin{itemize}
\item An open neighbourhood $V_1 \subset \{f \in \textup{Diff}(Z) \mid f^*(\omega,\Omega)_{f(s_i)}=(\omega,\Omega)_{s_i}$ \textup{ for all $i=1,\dots,N$ and $S=\{s_1,\dots,s_N\}$}\} of the identity (with respect to the $C^\infty$-topology). Furthermore, we assume $\textup{dist}(s_i,f(s_i))<\varepsilon$ for every $f\in V_1$ and $s_i\in S$ where $\varepsilon\ll \textup{dist}(s_i,s_j)$ for all $s_i\neq s_j \in S$.
\item An open neighbourhood $V_2 \subset \mathcal{A}_\mu^\Fr(P,\{\Upsilon_i,\tilde{\Upsilon}_i,A_i\})$ of $A$ with respect to the $C^\infty_{\mu}$-topology.
\end{itemize}
We then define $V_{\mathbb{A}}(V_1,V_2) \subset \mathcal{A}_\mu(\{P_i,A_i\})$ as
\begin{align*}
V_{\mathbb{A}}(V_1,V_2) \coloneqq \Big\{\big(S^\prime, f \circ \pi \col P \to Z\setminus S^\prime ,A^\prime\big) \in \mathcal{A}^\Fr_\mu(\{P_i,A_i\}) \ \Big\vert\  (f,A^\prime) \in V_1\times V_2 \Big\}
\end{align*} 
and $\mathcal{C}$ as the collection of all such subsets, i.e. $ \mathcal{C} \coloneqq \cup V_{\mathbb{A}}(V_1,V_2)$ (where the union is taken over all $\mathbb{A}$, $V_1$, and $V_2$ as above).

We now equip $\mathcal{A}_\mu(\{P_i,A_i\})$ with the topology generated by $\mathcal{C}$ and $\mathcal{B}_\mu(\{P_i,A_i\})$ with the quotient topology.
\end{definition}
\begin{remark}
\autoref{prop: change of framing} and \autoref{rmk: independence of C^infty_mu-topology from choices} imply that the definition of $V_{\mathbb{A}}(V_1,V_2)$ above is independent of the particular choice of framings $\{(\Upsilon_i,\tilde{\Upsilon}_i)\}$. 
\end{remark}

Next, we show that the natural forgetful map $\mathcal{A}_{\mu}^{\textup{Fr}}(\{P_i,A_i\}) \to \mathcal{A}_{\mu}(\{P_i,A_i\})$ induces a homeomorphism between $\mathcal{B}_{\mu}(\{P_i,A_i\})$ and the quotient of $\mathcal{B}_{\mu}^{\textup{Fr}}(\{P_i,A_i\})$ by the canonical $\times_i \Stab_{\SU(3)}(A_i)$-action (where $\Stab_{\SU(3)}(A_i)$ was defined in \eqref{equ: definition Stab_SU(3)(A_s)}).

\begin{proposition}\label{prop: unframed moduli is orbit space of framed moduli}
Let $\mathcal{A}_{\mu}^{\textup{Fr}}(\{P_i,A_i\})$ be the space of framed conically singular connections from \autoref{sec: moduli space and its topology}. There exists a canonical $(\times_i \Stab_{\SU(3)}(A_i))$-action on $\mathcal{A}_{\mu}^{\textup{Fr}}(\{P_i,A_i\})$ where $\{(U_i,\tilde{U}_i)\}\in (\times_i \Stab_{\SU(3)}(A_i))$ acts on $(S,\pi \col P \to Z \setminus S, \{[\Upsilon_i,\tilde{\Upsilon}_i]\},A) \in \mathcal{A}_{\mu}^{\textup{Fr}}(\{P_i,A_i\})$ via \[ (S,\pi \col P \to Z \setminus S, \{[\Upsilon_i \circ U_i,\tilde{\Upsilon}_i \circ \tilde{U}_i]\},A). \] This action is free and the forgetful-map $\mathcal{A}_{\mu}^{\textup{Fr}}(\{P_i,A_i\}) \to \mathcal{A}_{\mu}(\{P_i,A_i\})$ induces a homeomorphism from its orbit-space $\mathcal{A}_{\mu}^{\textup{Fr}}(\{P_i,A_i\})/ (\times_i \Stab_{\SU(3)}(A_i))$ to $\mathcal{A}_{\mu}(\{P_i,A_i\})$. Moreover, this homeomorphism descends to a homeomorphism \[ \mathcal{B}_{\mu}^{\textup{Fr}}(\{P_i,A_i\})/ (\times_i \Stab_{\SU(3)}(A_i)) \cong \mathcal{B}_{\mu}(\{P_i,A_i\}).\]
\end{proposition}
\begin{proof}
\autoref{prop: change of framing} implies that the forgetful map induces a (set-theoretic) bijection between $\mathcal{A}_{\mu}^{\textup{Fr}}(\{P_i,A_i\})/ (\times_i \Stab_{\SU(3)}(A_i))$ and $\mathcal{A}_{\mu}(\{P_i,A_i\})$. That this is an homeomorphism follows directly from the definition of the topologies on $\mathcal{A}_{\mu}^{\textup{Fr}}(\{P_i,A_i\})$ and $\mathcal{A}_{\mu}(\{P_i,A_i\})$.

If $\mathbb{A}, \mathbb{A}^\prime \in \mathcal{A}_{\mu}^{\textup{Fr}}(\{P_i,A_i\})$ are equivalent in the sense of \autoref{def: moduli space of framed connections} (i.e. $[\mathbb{A}]=[\mathbb{A}^\prime] \in \mathcal{B}_{\mu}^{\textup{Fr}}(\{P_i,A_i\})$), then $\mathbb{A}\cdot \{(U_i,\tilde{U}_i)\}$ is equivalent to $\mathbb{A}^\prime\cdot \{(U_i,\tilde{U}_i)\}$ (in the sense of \autoref{def: moduli space of framed connections}) for any $\{(U_i,\tilde{U}_i)\} \in (\times_i \Stab_{\SU(3)}(A_i))$. The action of $(\times_i \Stab_{\SU(3)}(A_i))$ descends therefore to $\mathcal{B}_{\mu}^{\textup{Fr}}(\{P_i,A_i\})$. 

Since the concatenation of the forgetful-map with the quotient map  \[\mathcal{A}_{\mu}^{\textup{Fr}}(\{P_i,A_i\}) \to \mathcal{A}_{\mu}(\{P_i,A_i\}) \to \mathcal{B}_{\mu}(\{P_i,A_i\})\] is constant both along the equivalence classes under $\sim$ defined in \autoref{def: moduli space of framed connections} and along the $(\times_i \Stab_{\SU(3)}(A_i))$-orbits, we obtain by the universal property of the quotient topology an induced map \[(\mathcal{B}_{\mu}^{\textup{Fr}}(\{P_i,A_i\}))/(\times_i \Stab_{\SU(3)}(A_i)) \to \mathcal{B}_{\mu}(\{P_i,A_i\}).\] That this is a homeomorphism again follows from the universal property of the quotient topology.
\end{proof}

The previous proposition allows us to use \autoref{thm: local structure of B Psi is homeo} to describe the local structure of $\mathcal{B}_\mu(\{P_i,A_i\})$. In the following we will first define a local parametrisation and then prove that this indeed defines a local homeomorphism.

\begin{definition}\label{def: local structure of unframed B definition of Phi}
Let $\mathbb{A}\coloneqq (S,\pi \col P \to Z \setminus S,A) \in \mathcal{A}_\mu(\{P_i,A_i\})$ be a fixed element and let $\{(\Upsilon_i,\tilde{\Upsilon}_i)\}$ be a set of framings such that $A \in \mathcal{A}_\mu^{\textup{Fr}}(P,\{\Upsilon_i,\tilde{\Upsilon}_i,A_i\})$. Moreover, pick for every $i=1,\dots,N$ a complementary subspace $\mathfrak{m}_i \subset \mathfrak{su}(3)$ of $\textup{image}(\mathfrak{stab}_{\SU(3)}(A_i) \to \mathfrak{su}(3))$ under the natural projection map $\mathfrak{stab}_{\SU(3)}(A_i) \to \mathfrak{su}(3)$. In the following, we denote for any $\varepsilon>0$ by $B^{\mathfrak{m}_i}_{\varepsilon}(0)$ the $\varepsilon$-ball around $0$ in $\mathfrak{m}_i$ and by $B_{\varepsilon}(0)$ the $\varepsilon$-ball in $\mathbb{C}^3$. For the (fixed) choices of $\mathfrak{m}_i$ and $\varepsilon>0$ we define the following map:
\begin{align*}
\Phi_{\mathbb{A}} \col (B_{\varepsilon}(0))^N \times (\times_iB_{\varepsilon}^{\mathfrak{m}_i}(0)) \times \mathcal{A}_\mu^{\textup{Fr}}(P,\{\Upsilon_i,\tilde{\Upsilon}_i,A_i\}) \to \mathcal{A}_\mu(\{P_i,A_i\})\\
\big(\Vec{v},\vec{u}, A^\prime\big) \mapsto \big(f_{\Vec{v},\vec{u}}(S), f_{\Vec{v},\vec{u}} \circ \pi \col P \to Z\setminus f_{\Vec{v},\vec{u}}(S), A^\prime\big)
\end{align*}
where $f_{\Vec{v},\vec{u}} \col Z \to Z$ denotes the diffeomorphism (realising the translation by $v_i$ and rotation by $\exp(u_i)$ at every $s_i \in S$) from \autoref{def: family of diffeomorphism parametrising open neighbourhood}.
\end{definition}

\begin{theorem}\label{thm: local structure of unframed B Phi is homeo}
Let $\mathbb{A}\coloneqq (S,\pi \col P \to Z \setminus S,A) \in \mathcal{A}_\mu(\{P_i,A_i\})$ be any element and $\Phi_{\mathbb{A}}$ be as in the previous definition. Furthermore, let $\mathfrak{q}\col \mathcal{A}_\mu(\{P_i,A_i\}) \to \mathcal{B}_\mu(\{P_i,A_i\})$ be the quotient map. Then $\mathfrak{q}\circ \Phi_\mathbb{A}$ descends to \[ \overline{\mathfrak{q} \circ \Phi_\mathbb{A}} \col (B_{\varepsilon}(0))^N \times (\times_iB_\varepsilon^{\mathfrak{m}_i}(0)) \times \big((\mathcal{A}_\mu^{\textup{Fr}}(P,\{\Upsilon_i,\tilde{\Upsilon}_i,A_i\}))/\mathcal{G}_{\mu+1})\big) \to \mathcal{B}_\mu(\{P_i,A_i\}),\] where 
\begin{align*}
\mathcal{G}_{\mu+1} \coloneqq \big\{ g \in \mathcal{G}(P) &\mid \vert \nabla^k (\tilde{\Upsilon}_i^{-1} \circ g \circ \tilde{\Upsilon}_i-\tilde{U}_i) \vert = \mathcal{O}(r^{\mu_i+1-k}) \textup{ for every $i=1,\dots,N$, $k \in \mathbb{N}_0$,} \\
& \quad \textup{and a $\tilde{U}_i \in \mathcal{G}(P_i)$ that preserves $A_i$} \big\}
\end{align*} 
is as in \autoref{thm: local structure of B Psi is homeo}. For sufficiently small $\varepsilon>0$, the map $\overline{\mathfrak{q} \circ \Phi_\mathbb{A}}$ is a homeomorphism onto an open subset of $\mathcal{B}_{\mu}(\{P_i,A_i\})$.
\end{theorem}
\begin{proof}
We will again assume that $N=1$. Recall from \autoref{thm: local structure of B Psi is homeo} that the map $\overline{\mathfrak{q} \circ \Psi_\mathbb{A}}$ (as defined in mentioned theorem) is a local homeomorphism between $\mathcal{B}_{\mu}^{\textup{Fr}}(\{P_s,A_s\})$ and \[B_{\varepsilon}(0) \times B_\varepsilon^{\mathfrak{m}_s}(0) \times  \big(\big(\textup{Stab}_{\SU(3)}(A_s) \times \mathcal{A}_\mu^{\textup{Fr}}(P,\{\Upsilon_s,\tilde{\Upsilon}_s,A_s\})\big)\big/\mathcal{G}_{\mu+1}\big).\] In this description $g \in \mathcal{G}_{\mu+1}$ acts on any $((U,\tilde{U}),A)\in \textup{Stab}(A_s) \times \mathcal{A}_\mu^{\textup{Fr}}(P,\{\Upsilon_s,\tilde{\Upsilon}_s,A_s\})$ via $((U_s,\lim_{\tilde{\Upsilon}_s}(g)^{-1} \circ \tilde{U}_s),g^*A)$, where $\lim_{\tilde{\Upsilon}_s}(g)\in \mathcal{G}(P_s)$ is characterised by \[ \vert\nabla^k(\tilde{\Upsilon}_s^{-1}\circ g \circ \tilde{\Upsilon}_s - \pr_{S^5}^*(\textstyle{\lim_{\tilde{\Upsilon}_s}}(g))) \vert = \mathcal{O}(r^{\mu_s+1-k}) \quad \textup{for every $k\in \mathbb{N}_0$.}\] 

Moreover, the (right) $\textup{Stab}_{\SU(3)}(A_s)$-action on $\mathcal{B}_{\mu}^{\textup{Fr}}(\{P_s,A_s\})$ corresponds under $\overline{\mathfrak{q}\circ \Psi_{\mathbb{A}}}$ to \[ (v,u,[ (U,\tilde{U}),A]) \cdot (U^\prime,\tilde{U}^\prime) = (v,u,[(U\circ U^\prime,\tilde{U} \circ \tilde{U}^\prime ),A ]) \] where $(U^\prime,\tilde{U}^\prime ) \in \Stab_{SU(3)}(A_s)$ and $(v,u,[(U,\tilde{U}),A])$ lies in  \[B_{\varepsilon}(0) \times B_\varepsilon^{\mathfrak{m}_s}(0) \times  \big(\big(\textup{Stab}_{\SU(3)}(A_s) \times \mathcal{A}_\mu^{\textup{Fr}}(P,\{\Upsilon_s,\tilde{\Upsilon}_s,A_s\})\big)\big/\mathcal{G}_{\mu+1}\big).\] 

This shows that the actions of $\mathcal{G}_{\mu+1}$ and $\Stab_{\SU(3)}(A_s)$ commute, so that $\overline{q \circ \Psi_{\mathbb{A}}}$ induces a local homeomorphism between 
\begin{align*}
& B_{\varepsilon}(0) \times B_\varepsilon^{\mathfrak{m}_s}(0) \times \left(\left. \raisebox{.2em}{$\big(\textup{Stab}_{\SU(3)}(A_s) \times \mathcal{A}_\mu^{\textup{Fr}}(P,\{\Upsilon_s,\tilde{\Upsilon}_s,A_s\})\big)\big/\mathcal{G}_{\mu+1}$} \right/ \raisebox{-.2em}{$\Stab_{\SU(3)}(A_s)$} \right)\\
= & B_{\varepsilon}(0) \times B_\varepsilon^{\mathfrak{m}_s}(0) \times \big( \mathcal{A}_\mu^{\textup{Fr}}(P,\{\Upsilon_s,\tilde{\Upsilon}_s,A_s\})/\mathcal{G}_{\mu+1} \big)
\end{align*}
and \[\mathcal{B}_{\mu}^{\textup{Fr}}(\{P_s,A_s\})/\Stab_{\SU(3)}(A_s) = \mathcal{B}_{\mu}(\{P_s,A_s\}).\] A moment's thought reveals that this map is precisely $\overline{q \circ \Phi_{\mathbb{A}}}$.
\end{proof}

\begin{remark}\label{rem: alternative interpretation of rotations}
We now give an alternative interpretation of the deformations in the previous theorem obtained by 'rotating' the bundle via $U \in \SU(3)$ and give a heuristic on why one only considers rotations parametrised by $B_{\varepsilon}^{\mathfrak{m}_i}(0) \subset \mathfrak{m}_i$. For simplicity we assume hereby that $Z = \mathbb{C}^3$ and that we consider conically singular connections with one singular point modelled on $(\pi_0 \col P_0 \to S^5,A_0)$ (otherwise we have to perform the following discussion locally). 

If \[ \big(\{0\}, (\pr_{S^5}^*\pi_0) \col (\pr_{S^5}^*P_0) \to \mathbb{C}^3\setminus \{0\},A \big) \in \mathcal{A}_\mu(\{P_0,A_0\}) \] is a conically singular connection and $U \in \SU(3)$, then we obtain a new conically singular connection via 'rotation':  \[ \big(\{0\}, U \circ (\pr_{S^5}^*\pi_0) \col (\pr_{S^5}^*P_0) \to \mathbb{C}^3\setminus \{0\},A \big) \in \mathcal{A}_\mu(\{P_0,A_0\}). \] Since $\SU(3)$ is connected, there exists an isomorphism \[F \col (U\circ \pi_0 \col P_0 \to S^5) \to (\pi_0 \col P_0 \to S^5)\] covering $\id$ (which can be constructed via parallel transport as in \autoref{prop: isomorphism from isotopy} and \autoref{prop: isomorphism from stabiliser}). In a comprehensive moduli theory which allows for variable tangent connections (cf. \autoref{rem: totally ordered singular set framed case} and \autoref{rem: full moduli theory for unframed connections}), one should identify \[ \big[\{0\}, U \circ (\pr_{S^5}^*\pi_0) \col (\pr_{S^5}^*P_0) \to \mathbb{C}^3\setminus \{0\},A \big] \in \mathcal{B}_\mu(\{P_0,A_0\}) \] with \[ \big[\{0\},(\pr_{S^5}^*\pi_0) \col (\pr_{S^5}^*P_0) \to \mathbb{C}^3\setminus \{0\},F_*A \big] \in \mathcal{B}_\mu(\{P_0,F_*A_0\})\] (because both elements are related by an isomorphism compatible with the respective conical structure). One can therefore interpret rotations of the bundle as deformations of $A\in \mathcal{A}(\pr_{S^5}^*P_0)$ (on the fixed bundle) which also change the tangent connection within the class \[\big\{F_*A_0 \in \mathcal{A}(P_0) \ \big\vert \ \textup{$F \col P_0\to P_0$ is an isomorphism covering any $U \in \SU(3)$} \big\} \subset \mathcal{A}(P_0).\] Of course, if $F$ preserves $A_0$ (i.e. $(U,F) \in \Stab_{\SU(3)}(A_0)$), then one stays within $\mathcal{B}_\mu(\{P_0,A_0\})$. Thus, in order to only add deformations that truly change the tangent connection one needs to restrict to rotations $U \in \SU(3)$ not contained in $\textup{image}( \Stab_{\SU(3)}(A_0)\to \SU(3))$. These are locally parametrised by $\mathfrak{m}_0$. This discussion is of course analogous to its well-known counterpart for conically singular submanifolds (cf. \cite[Definition~5.1]{Joyce-Moduli_of_cs-slag}).
\end{remark}

Note that the asymptotic limit map $g \mapsto \lim_{\tilde{\Upsilon}_i}(g)$ as defined in \autoref{thm: local structure of B Psi is homeo} (see also the previous proof) gives an embedding of $\mathcal{G}_{\mu+1}/\mathcal{G}_{0,\mu+1}$ into $\times_i \Stab_{\mathcal{G}(P_i)}(A_i)$ (the product-group of $A_i$-preserving gauge transformations $g_i \col P_i\to P_i$). Thus, if the center $Z(G)$ of the structure group is trivial and all tangent connections $A_i$ are irreducible, then $\mathcal{G}_{\mu+1}/\mathcal{G}_{0,\mu+1}$ is trivial. This implies the following:

\begin{corollary}
Assume that $G$ has a trivial center and that all tangent connections $A_i$ are irreducible (that is, $\Stab_{\mathcal{G}(P_i)}(A_i)$ is trivial). Then $\mathcal{B}_{\mu}^{\textup{Fr}}(\{P_i,A_i\}) \to \mathcal{B}_{\mu}(\{P_i,A_i\})$ is a principal $(\times_i \Stab_{\SU(3)}(A_i))$-bundle.
\end{corollary}

\section{The moduli space of conically singular SU(3)-instantons}\label{sec: moduli space of cs SU(3)-instantons}

In this section we now come to the moduli space of conically singular $\SU(3)$-instantons with prescribed tangent connections. We first give a definition of this space and use \autoref{thm: local structure of unframed B Phi is homeo} to (locally) reduce its complexity. We then prove the existence of local Kuranishi charts. That is, this moduli space is locally given by the zero-set of a smooth function between finite dimensional vector spaces. Moreover, we give a formula for its virtual dimension and show that all moduli spaces of instantons with rates lying in a certain cube are homeomorphic to each other.

\subsection{Definition of the moduli space and first properties}
\label{sec: moduli space and its topology}

Throughout this section, $Z^6$ is a compact 6-manifold with an $\SU(3)$-structure $(\omega,\Omega)$. Furthermore, $G$ is a compact Lie group whose Lie algebra $\mathfrak{g}$ has been equipped with an $\Ad$-invariant inner product. Recall also the definition of $\mathcal{B}_\mu(\{P_i,A_i\})$ (as a topological space) given in \autoref{sec: space of (unframed) cs connections}.

\begin{definition}[{\textbf{Moduli space of conically singular instantons with prescribed tangent connections}}] \label{def: moduli space of cs instantons}
Let $N\in \mathbb{N}$ be the number of singular points, $\mu\coloneqq \{\mu_i\}_{i\in \{1,\dots,N\}}$ for $\mu_i \in (-1,0)$ be a set of rates, and $\{(P_i,A_i)\}_{i\in \{1,\dots,N\}}$ be a fixed set consisting of a principal $G$-bundle $\pi_i \col P_i \to S^5$ and a connection $A_i \in \mathcal{A}(P_i)$ satisfying~\eqref{equ: cone reduction of instanton equation}. The moduli space of conically singular $\SU(3)$-instantons with $N$ singularities and prescribed tangent cones $\{(P_i,A_i)\}_{i\in \{1,\dots,N\}}$ of rate $\mu$ is the topological space defined as \[\mathcal{M}_\mu(\{P_i,A_i\}) \coloneqq \{ [(S,\pi \col P \to Z\setminus S,A)]\in \mathcal{B}_\mu(\{P_i,A_i\})\mid \textup{A solves~\eqref{equ: SU(3)-instanton}}\} \subset \mathcal{B}_\mu(\{P_i,A_i\}) \] equipped with the subspace topology.
\end{definition}

\begin{remark}
Note that the equivalence relation $\sim$ in the definition of $\mathcal{B}^{\textup{Fr}}_\mu(\{P_i,A_i\})$ preserves~\eqref{equ: SU(3)-instanton}. The subset $\mathcal{M}_\mu(\{P_i,A_i\})$ is therefore well-defined.
\end{remark}

\begin{remark}
Here and in the following we will focus on unframed conically singular instantons. Note, however, that there exists an analogous definition of a moduli space of framed conically singular instantons with prescribed tangent cones $\mathcal{M}^\textup{Fr}_\mu(\{P_i,A_i\}) \subset \mathcal{B}^\textup{Fr}_\mu(\{P_i,A_i\})$. Moreover, using \autoref{thm: local structure of B Psi is homeo}, all results discussed in the following for $\mathcal{M}_\mu(\{P_i,A_i\})$ have straight forwards analogues for $\mathcal{M}^\textup{Fr}_\mu(\{P_i,A_i\})$.
\end{remark}

Recall from \autoref{thm: local structure of B Psi is homeo} the definition of the following gauge groups
\begin{align*}
\mathcal{G}_{0,\mu+1} \coloneqq \big\{ g \in \mathcal{G}(P) &\mid \vert \nabla^k (\tilde{\Upsilon}_i^{-1} \circ g \circ \tilde{\Upsilon}_i-\textup{Id}) \vert = \mathcal{O}(r^{\mu_i+1-k}) \textup{ for every $i=1,\dots,N$} \\
& \qquad \qquad \qquad \qquad \qquad\qquad\qquad\qquad\qquad\qquad\quad \qquad \textup{and $k \in \mathbb{N}_0$}\big\} \\
\mathcal{G}_{\mu+1} \coloneqq \big\{ g \in \mathcal{G}(P) &\mid \vert \nabla^k (\tilde{\Upsilon}_i^{-1} \circ g \circ \tilde{\Upsilon}_i-\pr_{S^5}^*\tilde{U}_i) \vert = \mathcal{O}(r^{\mu_i+1-k}) \textup{ for every $i=1,\dots,N$,} \\
& \quad \textup{$k \in \mathbb{N}_0$, and a $\tilde{U}_i \in \mathcal{G}(P_i)$ that preserves $A_i$} \big\}
\end{align*} 

The following theorem is a direct consequence of \autoref{thm: local structure of unframed B Phi is homeo}:

\begin{theorem}\label{thm: local structure for instanton moduli space 1}
Let $[\mathbb{A}] \coloneqq [(S,\pi \col P \to Z \setminus S,A)] \in \mathcal{M}_\mu(\{P_i,A_i\})$ be a fixed conically singular instanton and let $\{(\Upsilon_i,\tilde{\Upsilon}_i)\}$ be any choice of framing such that $A \in \mathcal{A}_\mu^{\textup{Fr}}(P,\{\Upsilon_i,\tilde{\Upsilon}_i,A_i\})$. The map $\overline{\mathfrak{q} \circ \Phi_{\mathbb{A}}}$ defined in \autoref{thm: local structure of unframed B Phi is homeo} induces a homeomorphism between an open neighbourhood of $(0,0,[A])$ in 
\begin{align*}
\big\{ (\vec{v},\vec{u},[A]) \in B_{\varepsilon}(0))^N \times (\times_iB_\varepsilon^{\mathfrak{m}_i}(0)) \times \big(\mathcal{A}_\mu^{\textup{Fr}}&(P,\{\Upsilon_i,\tilde{\Upsilon}_i,A_i\})/\mathcal{G}_{\mu+1}\big) \ \big\vert \ \textup{$A$ satisfies \eqref{equ: SU(3)-instanton}} \\
& \textup{with respect to the $\SU(3)$-structure $f_{\vec{v},\vec{u}}^*(\omega,\Omega)$} \big\} \\
\subset (B_{\varepsilon}(0))^N \times  & (\times_iB_\varepsilon^{\mathfrak{m}_i}(0)) \times \big(\mathcal{A}_\mu^{\textup{Fr}}(P,\{\Upsilon_i,\tilde{\Upsilon}_i,A_i\})/\mathcal{G}_{\mu+1}\big)
\end{align*}
(where $f_{\vec{v},\vec{u}}$ is the diffeomorphism from \autoref{def: family of diffeomorphism parametrising open neighbourhood}) and an open neighbourhood of $[\mathbb{A}]$ in $\mathcal{M}_\mu(\{P_i,A_i\})$.
\end{theorem}

\subsection{The local structure of $\mathcal{M}_\mu(\{P_i,A_i\})$}

In this section we prove the existence of local Kuranishi charts on $\mathcal{M}_\mu(\{P_i,A_i\})$. That is, for each $[\mathbb{A}]\in \mathcal{M}_\mu(\{P_i,A_i\})$ there exist two finite dimensional vector spaces $W_1,W_2$ and a smooth map $\ob_{\mathbb{A}} \col W_1 \to W_2$ such that a neighbourhood of $[\mathbb{A}]$ in $\mathcal{M}_\mu(\{P_i,A_i\})$ is homeomorphic to a neighbourhood of zero in $\ob^{-1}_{\mathbb{A}}(0)$. For this we first prove in \autoref{sub: Coulomb gauge} a slice theorem for the action of the (based) gauge group. In \autoref{sub: Kuranishi charts} we then establish the existence of such Kuranishi charts.

Throughout this section we restrict to compact structure groups $G$ that have a finite center and to tangent connections $A_i$ that are infinitesimally irreducible. This assumption makes the presentation in \autoref{sub: Coulomb gauge} a bit simpler but can be removed as in \cite[Chapter~I.5]{SoleFarre-thesis} (see also \autoref{rem: Coulomb gauge and slice for groups with positive center} and \autoref{rem: Kuranishi chart for groups with positive center}).

\subsubsection{Coulomb gauge as a slice for the gauge action}\label{sub: Coulomb gauge}

We have seen in \autoref{thm: local structure for instanton moduli space 1} that $\mathcal{M}_\mu(\{P_i,A_i\})$ is locally homeomorphic to an open neighbourhood inside a product involving \[\mathcal{A}_\mu^{\textup{Fr}}(P,\{\Upsilon_i,\tilde{\Upsilon}_i,A_i\})/\mathcal{G}_{\mu+1} = \big(\mathcal{A}_\mu^{\textup{Fr}}(P,\{\Upsilon_i,\tilde{\Upsilon}_i,A_i\})/\mathcal{G}_{0,\mu+1}\big)\big/\big(\mathcal{G}_{\mu+1}/\mathcal{G}_{0,\mu+1}\big)\] for some fixed bundle $\pi \col P \to Z\setminus S$, a set of framings $\{(\tilde{\Upsilon}_i \col \pr_{S^5}^*P_i \to \Upsilon_i^*P)\}$, and tangent connections $\{(\pi_i \col P_i \to S^5,A_i)\}$. In this section, we establish a slice theorem for the action of the based gauge group on (a Banach space version of) $\mathcal{A}_\mu^\Fr(P,\{\Upsilon_i,\tilde{\Upsilon}_i,A_i\})$. This implies that the subspace of $\mathcal{A}_\mu^{\textup{Fr}}(P,\{\Upsilon_i,\tilde{\Upsilon}_i,A_i\})/\mathcal{G}_{0,\mu+1}$ consisting of (the equivalence classes of) $\SU(3)$-instantons is (locally) given by the zero-set of the (non-linear) elliptic equation~\eqref{equ: augmented instanton equations} and is another key step in establishing the existence of local Kuranishi charts on $\mathcal{M}_\mu(\{P_i,A_i\})$.

As in the previous sections, we assume that $G$ is a compact Lie group (and a subgroup of $\GL(W)$ for some vector space $W$). In contrast to the previous sections, however, we additionally assume that the center of $G$ is 0-dimensional (and therefore finite). This assumption is used in \autoref{prop: rates of invertibility of Laplacian} (and consequently in \autoref{thm: Coulomb Gauge} and \autoref{thm: Slice Theorem}) to conclude that the Laplacian associated to an irreducible connection is invertible at certain rates. However, we note that by working with the larger group $\mathcal{G}_{\mu+1}$, this assumption can be removed (cf. \autoref{rem: Coulomb gauge and slice for groups with positive center}) and has been worked out in more detail in \cite[Chapter~I.5]{SoleFarre-thesis}.

We begin by considering a principal $G$-bundle $\pi_0 \col P_0 \to S^5$ equipped with a connection $A_0 \in \mathcal{A}(P_0)$ whose pullback to $\mathbb{C}^3\setminus \{0\}$ defines an $\SU(3)$-instanton (i.e. $A_0$ satisfies \eqref{equ: cone reduction of instanton equation} by \autoref{prop: dilation invariant instantons}). In order to ease the notation we will in the following denote by $\pi^\prime_0 \col P_0^\prime \to \mathbb{C}^3\setminus \{0\}$ and $A_0^\prime \coloneqq \pr_{S^5}^*A_0$ their respective pullbacks.

We now identify the rates at which the Laplacian $\Delta_{A^\prime_0}\coloneqq \diff_{A_0^\prime}^*\diff_{A_0^\prime}$ associated to $A_0^\prime$ acting on sections of the adjoint bundle $ \mathfrak{g}_{P_0^\prime} \to \mathbb{C}^3\setminus \{0\}$ is an isomorphism. For this we need the following definitions.

\begin{definition}
A section $\mathfrak{\xi}\in \Omega^{0}(\mathbb{C}^3\setminus \{0\},\mathfrak{g}_{P_0^\prime})$ is called homogeneous of degree $\lambda \in \mathbb{R}$ if it satisfies $\delta_r^*\xi = r^{\lambda} \xi$ for every $r\in(0,\infty)$, where $\delta_r \col \mathbb{C}^3\setminus \{0\} \to \mathbb{C}^3\setminus \{0\}$ denotes the dilation by $r$ and where we use parallel transport (with respect to $A_0^\prime$) in radial direction to identify different fibers. Similarly, a 1-form $a \in \Omega^1(\mathbb{C}^3\setminus \{0\},\mathfrak{g}_{P_0^\prime})$ is called homogeneous of degree $\lambda \in \mathbb{R}$ if $\delta_r^*a = r^{\lambda+1} a$ for every $r \in (0,\infty)$.
\end{definition}

\begin{remark}
A 1-form $a \in \Omega^1(\mathbb{C}^3\setminus \{0\},\mathfrak{g}_{P_0^\prime})$ is homogeneous of degree $\lambda \in \mathbb{R}$ if and only if it is of the form \[ a = \textstyle{\sum} \xi_i \diff x^{i},\] where $\xi_i \in \Omega^0(\mathbb{C}^3\setminus \{0\},\mathfrak{g}_{P_0^\prime})$ are homogeneous sections of degree $\lambda$ and $\diff x^1, \dots, \diff x^6$ are the canonical (dual-) basis elements of $\mathbb{C}^3\equiv \mathbb{R}^6$.
\end{remark}

\begin{definition}\label{def: critical rates}
For the operators $\Delta_{A^\prime}$ (as above) and $L_{A_0^\prime}$ (associated to the connection $A_0^\prime$ as defined prior to \autoref{prop: adjoint of model operator}) we define the following sets:
\begin{align*}
\mathcal{D}(\Delta_{A_0^\prime}) \coloneqq \big\{ \lambda \in \mathbb{R}\ \big\vert   &\textup{ $\exists$ non-trivial homogeneous $\mathfrak{\xi}\in \Omega^0(\mathbb{C}^3\setminus\{0\}, \mathfrak{g}_{P_0^\prime})$ of degree $\lambda$} \\
& \textup{ with $\diff_{A_0^\prime}^*\diff_{A_0^\prime} \xi=0$} \big\} \\
\mathcal{D}(L_{A_0^\prime}) \coloneqq \big\{ \lambda \in \mathbb{R} \ \big\vert  &\textup{ $\exists$ non-trivial homogeneous $\underline{a}\in \Omega^1(\mathbb{C}^3\setminus\{0\}, \mathfrak{g}_{P_0^\prime}\oplus \mathfrak{g}_{P_0^\prime} \oplus T^*\mathbb{C}^3\otimes \mathfrak{g}_{P_0^\prime})$} \\
& \textup{ of degree $\lambda$ with $L_{A_0^\prime} \underline{a}=0$} \big\}.
\end{align*}
\end{definition}
\begin{proposition}\label{prop: critical rates of model Laplacian}
Let $A_0\in \mathcal{A}(P_0)$ be a connection whose pullback $A_0^\prime$ is an $\SU(3)$-instanton. Then $\mathcal{D}(\Delta_{A_0^\prime}) \cap (-4,0) = \emptyset$. Furthermore, if the center of $G$ is finite and $A_0$ is infinitesimally irreducible (that is, the only section $\xi \in \Omega^0(S^5,\mathfrak{g}_{P_0})$ with $\diff_{A_0} \xi =0$ is $\xi=0$), then \[\mathcal{D}(\Delta_{A_0^\prime}) \cap (-4,1) \subset \{ \lambda+1 \mid  \lambda \in \mathcal{D}(L_{A_0}) \cap (-1,0)\}.\]
\end{proposition}

\begin{proof}
Let $\lambda \in \mathcal{D}(\Delta_{A_0^\prime}) \cap (-4,1)$ and let $\xi \in \Omega^0(\mathbb{C}^3\setminus \{0\},\mathfrak{g}_{P^\prime_0})$ be a homogeneous section of degree $\lambda$ which satisfies $\diff_{A_0^\prime}^*\diff_{A_0^\prime} \xi =0$. Using parallel transport in radial direction to identify different fibers, we can write $\xi$ as $\xi = r^\lambda \xi_0$ for $\xi_0 \in \Omega^0(S^5,P_0)$. The equation $\diff_{A_0^\prime}^*\diff_{A_0^\prime}\xi =0$ is then equivalent to \[ - \lambda (\lambda+4) \xi_0 + \diff_{A_0}^*\diff_{A_0} \xi_0 =0 \] where $\diff_{A_0}^*\diff_{A_0}$ denotes the Laplacian of $A_0$ over the sphere $S^5$. Since $\diff_{A_0}^*\diff_{A_0}$ is a positive operator, this equation does not have a non-trivial solution for $\lambda \in (-4,0)$. Furthermore, the solutions for $\lambda=0$ are precisely given by parallel sections $\xi_0$ of $\diff_{A_0}$. Hence, if $A_0$ is infinitesimally irreducible and the center of $G$ discrete, $\lambda$ needs to be positive. Since $A_0$ is an $\SU(3)$-instanton, $\diff_{A_0^\prime} \xi$ is a homogeneous element in the kernel of $L_{A_0}$ which is homogeneous of degree $\lambda-1$. So either $\diff_{A_0^\prime} \xi =0$ in which case $\xi$ vanishes (because it is parallel and vanishes at $0$) or $(\lambda-1) \in \mathcal{D}(L_{A_0})\cap (-1,0)$.
\end{proof}

In the following, let $Z$ be a 6-manifold equipped with an $\SU(3)$-structure $(\omega,\Omega)$. Furthermore, let $S \coloneqq \{s_1,\dots,s_N\} \subset Z$ and $\pi \col P \to Z \setminus S$ be a bundle together with a framed conically singular connection $A\in \mathcal{A}_\mu^{\textup{Fr}}(P,\{\Upsilon_i,\tilde{\Upsilon}_i,A_i\})$. We will now define weighted Hölder spaces of sections of $\mathfrak{g}_P$ and use the previous discussion on homogeneous kernel elements to show that $\diff_A^*\diff_A$ is an isomorphism for a certain range of rates. 

\begin{definition}\label{def: weighted Hölder norms}
Let $\pi \col P \to Z\setminus S$ and $A \in \mathcal{A}_\mu^{\textup{Fr}}(P,\{\Upsilon_i,\tilde{\Upsilon}_i,A_i\})$ be as above and let $P\times_{\nu}W_\nu$ be the associated bundle to any (fixed) representation $(W_\nu,\nu)$ of $G$. Moreover, let $\rho \col Z \setminus S \to (0,\infty)$ and $w_\lambda \col Z \to \mathbb{R}$ be the distance and rate functions of \autoref{def: C^infty_mu topology} and set $\rho(x,y)\coloneqq \min\{\rho(x),\rho(y)\}$ for any $x,y \in Z \setminus S$. For any $k\in \mathbb{N}_0$, $\alpha \in (0,1)$, and $\lambda = (\lambda_1,\dots,\lambda_N) \in \mathbb{R}^N$ we define the following weighted Hölder (semi-) norms acting on any $\eta\in C^{k,\alpha}_{\textup{loc}}(Z \setminus S,\Lambda^\ell T^* Z \otimes (P\times_{\nu}W_\nu))$:
\begin{align*}
[\eta\WsH{0}{\lambda}{} &\coloneqq \sup_{2\diff(x,y) < \rho(x,y)} \rho(x,y)^{w_{\lambda-\alpha}(x)} \frac{\vert\eta(x)- \eta(y)\vert}{\vert x-y\vert^\alpha} \\
\Vert \eta \VertWH{0}{\lambda}{} & \coloneqq \Vert \rho^{-w_{\lambda}} \eta \VertC{0}{} + [\eta\WsH{0}{\lambda}{}  \\
\Vert \eta \VertWH{k}{\lambda}{} &\coloneqq \sum_{i=0}^k  \Vert \nabla^i_{A} \eta \VertWH{0}{\lambda-i}{},
\end{align*} 
where $\lambda-i \coloneqq (\lambda_1-i,\dots,\lambda_N-i)$ and where all covariant derivatives are taken with respect to $A$ and the Levi--Civita connection on $T^*Z$. To compare $\eta(x)$ and $\eta(y)$ which lie over different fibers we use parallel transport over the shortest geodesic connecting $x$ and $y$. 
\end{definition}

\begin{definition}\label{def: weighted Hölder spaces}
With $\pi \col P \to Z\setminus S$, $A \in \mathcal{A}_\mu^{\textup{Fr}}(P,\{\Upsilon_i,\tilde{\Upsilon}_i,A_i\})$, and $(W_\nu,\nu)$ as in the previous definition, we define $C^{k,\alpha}_\lambda(Z\setminus S, \Lambda^\ell T^*Z \otimes (P\times_{\nu} W_\nu))$ as the Banach space consisting of all sections $\eta\in C^{k,\alpha}_{\textup{loc}}(Z \setminus S,\Lambda^\ell T^*Z \otimes (P\times_{\nu} W_\nu))$, for which $\Vert \eta \VertWH{k}{\lambda}{}$ is finite, equipped with the norm $\Vert \cdot \VertWH{k}{\lambda}{}$. Moreover, for any fixed (smooth) $A^\prime \in \mathcal{A}_\lambda^\Fr(P,\{\Upsilon_i,\tilde{\Upsilon}_i,A_i\})$, we define \[\mathcal{A}_\lambda^{k,\alpha}(P,\{\Upsilon_i,\tilde{\Upsilon}_i,A_i\}) \coloneqq A^\prime + C^{k,\alpha}_\lambda(Z\setminus S, T^*Z \otimes \mathfrak{g}_P) \] as an affine Banach space. Ultimately, we define \[ \mathcal{G}^{k,\alpha}_{0,\lambda} \coloneqq \{ \textup{Id} + g \mid g\in C^{k,\alpha}_\lambda(Z\setminus S,P\times_{G} \End(W)) \textup{ and } \textup{Id} + g \in P\times_{G} G\},\] where we again assumed that $G\subset \GL(W)$ for some vector space $W$.
\end{definition}

\begin{remark}
Since all $\mu_i>-1$, it is straight forward to see that a different choice of $A \in \mathcal{A}_\mu^\Fr(P,\{\Upsilon_i\tilde{\Upsilon}_i,A_i\})$ leads to an equivalent norm $\Vert \cdot \VertWH{k}{\lambda}{}$. Furthermore, a moment's thought shows that the definition of $\mathcal{A}_\lambda^{k,\alpha}(P,\{\Upsilon_i,\tilde{\Upsilon}_i,A_i\})$ (as an affine Banach space) is also independent of the choice of base-connection $A^\prime$. It is well-known that $\mathcal{G}^{k+1,\alpha}_{0,\lambda+1}$ is a Banach Lie group with Lie algebra $C^{k+1,\alpha}_{\lambda+1}(Z\setminus S, \mathfrak{g}_P)$ which acts smoothly on $\mathcal{A}_\lambda^{k,\alpha}(P,\{\Upsilon_i,\tilde{\Upsilon}_i,A_i\})$. 
\end{remark}

\begin{proposition}\label{prop: rates of invertibility of Laplacian}
Let $A \in \mathcal{A}_\mu^\Fr(P,\{\Upsilon_i,\tilde{\Upsilon}_i,A_i\})$ be as above and assume that the center of $G$ is finite and that all tangent cones $\{(P_i,A_i)\}$ are infinitesimally irreducible. Define for every $i=1,\dots,N$ \[\Bar{\mu}_i \coloneqq \min \{((-1,0)  \cap \mathcal{D}(L_{A_i})) \cup \{0\} \}\] and fix $k\geq 2$. Then \[\diff_{A}^*\diff_A \col C^{k,\alpha}_{\lambda}(Z\setminus S,\mathfrak{g}_P) \to C^{k-2,\alpha}_{\lambda-2}(Z\setminus S, \mathfrak{g}_P) \] is an isomorphism for all $\lambda \in \mathbb{R}^N$ with $\lambda_i \in (-4-\Bar{\mu}_i-1,\Bar{\mu}_i+1)$.
\end{proposition}
\begin{proof}
By \autoref{prop: critical rates of model Laplacian}, the critical rates $\mathcal{D}(\Delta_{A_i^\prime})$ of $\diff_{A_i}^*\diff_{A_i}$ for any $i=1,\dots,N$ are disjoint from the interval $(-4,\Bar{\mu}_i+1)$. The operator $\diff_{A}^*\diff_A \col C^{k,\alpha}_{\lambda}(Z\setminus S, \mathfrak{g}_P) \to C^{k-2,\alpha}_{\lambda-2}(Z\setminus S, \mathfrak{g}_P)$ is therefore Fredholm of constant Fredholm index for all $\lambda \in \times_i (-4-\Bar{\mu}_i-1,\Bar{\mu}_i+1)$ (cf. \autoref{app-prop: Fredholm away from critical rates} and \autoref{app-prop: ker and coker are locally constrant + index change formula}). Furthermore, for $\lambda = (-2,\dots,-2)$ (and therefore for all $\lambda \in \times_i(-4-\Bar{\mu}_i-1,\Bar{\mu}_i+1)$) the formal self-adjointness of $\diff_A^*\diff_A$ implies that this index is zero (cf. \autoref{app-prop: cokernel kernel pairing}). 
    
Next, we show that $\diff_{A}^*\diff_A \col C^{k,\alpha}_{\varepsilon}(Z\setminus S, \mathfrak{g}_P) \to C^{k-2,\alpha}_{\varepsilon-2}(Z\setminus S, \mathfrak{g}_P)$ for $\varepsilon \in  \times_i(0,\Bar{\mu}_i+1)$ is injective. For this assume that $\xi \in C^{k,\alpha}_{\varepsilon}(Z\setminus S, \mathfrak{g}_P)$ satisfies $\diff_A^*\diff_A \xi=0$. Integration by parts then gives $\diff_A \xi=0$ which implies that $\vert \xi \vert$ is constant and since $\vert \xi \vert = \mathcal{O}(r^{\varepsilon_i})$ around any $s_i\in S$, that $\vert \xi \vert$ vanishes everywhere. Since the kernel of $\diff_{A}^*\diff_A$ is independent of $\lambda \in \times_i (-4-\Bar{\mu}_i-1,\Bar{\mu}_i+1)$ (cf. \autoref{app-cor: kernel is locally constant}) the proposition follows.
\end{proof}

We come now to the main results of this section, in which we first show that connections which lie nearby a fixed connection $A\in\mathcal{A}^{k,\alpha}_\mu(P,\{\Upsilon_i,\tilde{\Upsilon}_i,A_i\})$ can be put into Coulomb gauge relative to $A$. We then use this gauge to construct a slice to the action of $\mathcal{G}^{k+1,\alpha}_{\mu+1,0}$ on $\mathcal{A}^{k,\alpha}_\mu(P,\{\Upsilon_i,\tilde{\Upsilon}_i,A_i\})$.

\begin{theorem}\label{thm: Coulomb Gauge}
Let $\pi \col P \to Z\setminus S$ be a framed principal $G$-bundle, where we assume that $G$ has a finite center. Moreover, assume that all tangent connections $\{(P_i,A_i)\}$ are infinitesimally irreducible and that $-1<\mu_i< \bar{\mu}_i$ for every $i=1,\dots,N$ (with $\bar{\mu}_i$ as in the previous proposition). For every fixed $k\geq 1$, $\alpha\in(0,1)$, and $A\in \mathcal{A}^{k,\alpha}_\mu(P,\{\Upsilon_i,\tilde{\Upsilon}_i,A_i\})$ there exists an open neighbourhood $V\subset \mathcal{A}^{k,\alpha}_\mu(P,\{\Upsilon_i,\tilde{\Upsilon}_i,A_i\})$ of $A$ and a smooth map $s \col V \to \mathcal{G}^{k+1,\alpha}_{0,\mu+1}$ with $s(A)=\textup{Id}$ such that \[ \diff_A^*(s(A^\prime)^*A^\prime-A) = 0 \] for all $A^\prime \in V$. Moreover, the map 
\begin{align*}
\Psi \col V &\to \ker\big(\diff_A^* \col C^{k,\alpha}_{\mu}(Z\setminus S,T^*Z\otimes\mathfrak{g}_P) \to C^{k-1,\alpha}_{\mu-1}(Z\setminus S,\mathfrak{g}_P)\big) \times \mathcal{G}^{k+1,\alpha}_{0,\mu+1} \\
A^\prime &\mapsto \big(s(A^\prime)^*A^\prime-A,s(A^\prime)\big)
\end{align*}
is a diffeomorphism onto a neighbourhood of $(0,\textup{Id})$.
\end{theorem}

The proof of this theorem is the same as in \cite[Theorem~3.2]{FreedUhlenbeck--Instantons}. We have included it here nevertheless, for the convenience of the reader.

\begin{proof}
In order to construct $s$, we consider 
\begin{align*}
\mathfrak{C} \col \mathcal{G}^{k+1,\alpha}_{0,\mu+1} \times &C^{k,\alpha}_{\mu}(Z \setminus S, T^*Z \otimes \mathfrak{g}_P) \to C^{k-1,\alpha}_{\mu-1}(Z\setminus S, \mathfrak{g}_P)\\
(g,a) & \mapsto \diff_A^*(g^{-1}\diff_A g + g^{-1} a g)
\end{align*}
where we again assumed that $G\subset \GL(W)$ and regarded $g$ and $a$ as a respective section and 1-form with values in $P\times_{G} \End(W)$. The derivative of $\mathfrak{C}$ with respect to the first variable at $(\textup{Id},0)$ is given by \[ (\partial_1 \mathfrak{C})_{(\textup{Id},0)} = \diff_A^*\diff_A \col C^{k+1,\alpha}_{\mu+1}(Z\setminus S, \mathfrak{g}_P) \to C^{k-1,\alpha}_{\mu-1}(Z\setminus S, \mathfrak{g}_P) \] and is an isomorphism by the previous proposition. The Implicit Function Theorem implies therefore that there exists an open neighbourhood $V^\prime\subset C^{k,\alpha}_\mu(Z\setminus S, T^*Z\otimes \mathfrak{g}_P)$ of the zero-section and a smooth function $s^\prime \col V^\prime \to \mathcal{G}^{k+1,\alpha}_{0,\mu+1}$ with $s^\prime(0)=\textup{Id}$ and $\mathfrak{C}(s^\prime(a),a)=0$ for every $a\in V^\prime$. The mapping $s(A+a)\coloneqq s^\prime(a)$ defined on $V \coloneqq A + V^\prime$ satisfies the properties of the theorem.

In order to see that the corresponding map $\Psi$ is a diffeomorphism onto a neighbourhood of $(0,\textup{Id})$ note that \[(a,g)\mapsto A+ gag^{-1} -(\diff_A g)g^{-1} \] is a local inverse to $\Psi$.
\end{proof}

\begin{theorem}\label{thm: Slice Theorem}
Assume that we are in the same situation as in the previous theorem (in particular, we assume again that $G$ has a finite center and that all tangent connections are infinitesimally irreducible). For every $k\geq 1$, $\alpha\in (0,1)$, and $\mu =(\mu_1,\dots,\mu_N) \in \mathbb{R}^N$ with $-1<\mu_i< \bar{\mu}_i$ for every $i=1,\dots,N$ (for $\bar{\mu}_i$ as in \autoref{prop: rates of invertibility of Laplacian}) the following holds: The quotient $\mathcal{A}^{k,\alpha}_\mu(P,\{\Upsilon_i,\tilde{\Upsilon}_i,A_i\})/\mathcal{G}^{k+1,\alpha}_{0,\mu+1}$ (equipped with its quotient topology) is Hausdorff and carries the structure of a Banach manifold, where a neighbourhood of $[A]\in \mathcal{A}^{k,\alpha}_\mu(P,\{\Upsilon_i,\tilde{\Upsilon}_i,A_i\})/\mathcal{G}^{k+1,\alpha}_{0,\mu+1}$ is homeomorphic to \[ \mathcal{S}_{[A]} \coloneqq \big\{ A^\prime \in \mathcal{A}^{k,\alpha}_\mu(P,\{\Upsilon_i,\tilde{\Upsilon}_i,A_i\})\ \big\vert\ \textup{where $\Vert A^\prime-A\VertWH{k}{\mu}{}<\varepsilon$ and $\diff_{A}^*(A^\prime - A ) =0$}\big\} \] for some $\varepsilon>0$.
\end{theorem}

The proof of this theorem is again analogously to its counterpart for non-singular connection \cite[Corollary to Theorem~3.2]{FreedUhlenbeck--Instantons}.

\begin{proof}
We prove the Hausdorff property of $\mathcal{A}^{k,\alpha}_\mu(P,\{\Upsilon_i,\tilde{\Upsilon}_i,A_i\})/\mathcal{G}^{k+1,\alpha}_{0,\mu+1}$ by showing that \begin{align*}
\Big\{(A,g^*A) \mid A\in \mathcal{A}^{k,\alpha}_\mu(P,\{\Upsilon_i,\tilde{\Upsilon}_i,A_i\}) \textup{ and } &g\in \mathcal{G}^{k+1,\alpha}_{0,\mu+1} \Big\} \\
&\subset \mathcal{A}^{k,\alpha}_\mu(P,\{\Upsilon_i,\tilde{\Upsilon}_i,A_i\})\times\mathcal{A}^{k,\alpha}_\mu(P,\{\Upsilon_i,\tilde{\Upsilon}_i,A_i\}) 
\end{align*} 
is closed. For this assume that $(A_n)_{n\in \mathbb{N}} \subset \mathcal{A}^{k,\alpha}_\mu(P,\{\Upsilon_i,\tilde{\Upsilon}_i,A_i\})$ and $(g_n)_{n\in \mathbb{N}}\subset \mathcal{G}^{k+1,\alpha}_{0,\mu+1}$ are sequences with 
\[ A_n \to A \quad \textup{and} \quad g_n^*A_n \to A^\prime \quad \textup{in $\mathcal{A}^{k,\alpha}_\mu(P,\{\Upsilon_i,\tilde{\Upsilon}_i,A_i\})$} \] for some $A, A^\prime \in \mathcal{A}^{k,\alpha}_\mu(P,\{\Upsilon_i,\tilde{\Upsilon}_i,A_i\})$. We write all connections as \[A_n=A_0 + a_n, \quad A= A_0+a, \quad \textup{and} \quad A^\prime= A_0 +a^\prime\] where $A_0\in \mathcal{A}_\mu^\Fr(P,\{\Upsilon_i,\tilde{\Upsilon}_i,A_i\})$ is a fixed base connection that agrees for every $i=1,\dots,N$ with $(\tilde{\Upsilon}_i)_*A_i$ on $\Upsilon_i(B_{R/2}(0))$. This leads to    
\begin{align}
a_n &\to a \quad \textup{and} \nonumber \\
g_n^{-1}\diff_{A_0}g_n +g_n^{-1}a_n g_n  &\to a^\prime \quad \textup{in $C^{k,\alpha}_\mu(Z\setminus S, T^*Z\otimes \mathfrak{g}_P)$} \label{equ: bootstrapping gauge transformations}
\end{align}
where we again assumed that $G\subset \GL(W)$ and regarded $g$ and $a$ as a respective section and 1-form with values in $P\times_{G} \End(W)$. Since $G$ is compact, we obtain an $n$-independent bound on $\vert \rho^{-w_\mu-1} (g_n-\textup{Id}) \vert$ over $Z\setminus (\cup_i B_{R/4}(s_i))$. Furthermore, identifying a neighbourhood of any of the $s_i\in S$ with $B_R(0)\subset \mathbb{C}^3$ via $\Upsilon_i$, we obtain from~\eqref{equ: bootstrapping gauge transformations} \[ \vert g_n- \textup{Id} \vert(z) \leq c \int_0^1 \vert z \vert \vert \partial_r g_n \vert(rz) \diff r \leq c \vert z \vert^{\mu_i+1} \] for any $z\in B_{R/2}(0)$ and a constant $c>0$ which is independent of $n$. Therefore, there exists a constant $C>0$ such that $\Vert (g_n-\textup{Id}) \VertWC{0}{\mu+1}{}<C$ independently of $n$. Bootstrapping via~\eqref{equ: bootstrapping gauge transformations} gives $\Vert g_n - \textup{Id} \VertWH{k+1}{\mu+1}{}<C$. Since the embedding $C^{k+1,\alpha}_{\mu+1} \subset C^{k,\alpha}_{\mu+1-\varepsilon}$ for any $0<\varepsilon< \tfrac{\mu_i+1}{2}$ is compact (cf. \autoref{app-prop: weighted Hölder spaces are Banach and compact embedding}), we obtain a converging subsequence $g_n\to g$ in $\mathcal{G}^{k,\alpha}_{\mu+1-\varepsilon}$. Since $k\geq 1$, this satisfies $g^*A=A^\prime$ and by using \eqref{equ: bootstrapping gauge transformations} for $g$ as above we can then conclude $g \in \mathcal{G}^{k+1,\alpha}_{0,\mu+1}$. This shows that the quotient $\mathcal{A}^{k,\alpha}_\mu(P,\{\Upsilon_i,\tilde{\Upsilon}_i,A_i\})/\mathcal{G}^{k+1,\alpha}_{0,\mu+1}$ is Hausdorff. 

The previous theorem shows that for any $[A]\in \mathcal{A}^{k,\alpha}_\mu(P,\{\Upsilon_i,\tilde{\Upsilon}_i,A_i\})/\mathcal{G}^{k+1,\alpha}_{0,\mu+1}$ the map 
\begin{align*}
\mathcal{S}_{[A]} & \to \mathcal{A}^{k,\alpha}_\mu(P,\{\Upsilon_i,\tilde{\Upsilon}_i,A_i\})/\mathcal{G}^{k+1,\alpha}_{0,\mu+1} \\
A^\prime &\mapsto [A^\prime]
\end{align*}
is open and surjective onto a neighbourhood of $[A]$. We are therefore left to show injectivity for sufficiently small $\varepsilon>0$. Assume that this is not the case. Then there exist $(A^\prime_{n})_{n\in \mathbb{N}}$ and $(\tilde{A}_n)_{n\in \mathbb{N}}$ such that $A_n^\prime \neq \tilde{A}_n$ but $\Vert A^\prime_n -A \VertWH{k}{\mu}{}<\frac{1}{n}$ and $\Vert \tilde{A}_n -A \VertWH{k}{\mu}{}<\frac{1}{n}$ as well as $\diff_A^*(A^\prime_n -A) =0$ and $\diff_A^*(\tilde{A}_n -A) =0$ for all $n\in \mathbb{N}$. Furthermore, there exist $g_n \in \mathcal{G}^{k+1,\alpha}_{0,\mu+1}$ such that $g_n^*A^\prime_n = \tilde{A}_n$. As above we can conclude that there exists a $g\in \mathcal{G}^{k+1,\alpha}_{0,\mu+1}$ such that $g_n \to g$ and $g^*A=A$. Thus, $g$ is constant and since $\vert g - \textup{Id} \vert = \mathcal{O}(r^{\mu_i+1})$ around any $s_i$, we have $g=\textup{Id}$. For sufficiently large $n\in \mathbb{N}$ we therefore obtain that the pairs $(A^\prime_n-A,\textup{Id})$ and $(\tilde{A}_n-A,g_n)$ lie in the open neighbourhood $\Psi(V)$ of the previous theorem. Since the map $\Psi$ of the previous theorem is a diffeomorphism, we obtain $g_n= \textup{Id}$ and therefore $A^\prime_n=\tilde{A}_n$ for sufficiently large $n\in \mathbb{N}$, which contradicts $A_n^\prime \neq \tilde{A}_n$.
\end{proof}

\begin{remark}
We note here that one can prove in a similar way as in the previous theorem that $\mathcal{B}_{\mu}^{\textup{Fr}}(\{P_i,A_i\})$ and $\mathcal{B}_{\mu}(\{P_i,A_i\})$ with their respective topologies defined in \autoref{sec: space of framed connections and its topology} and \autoref{sec: space of (unframed) cs connections} are Hausdorff as well.
\end{remark}

The following is a corollary of the previous theorem and \autoref{thm: local structure of unframed B Phi is homeo} (and \autoref{thm: local structure for instanton moduli space 1}).

\begin{corollary}\label{cor: local structure for instanton moduli space irred tangent cones 1}
Assume that we are in the situation of the previous theorem and let $[\mathbb{A}]\coloneqq [(S,\pi \col P \to Z \setminus S,A)] \in \mathcal{B}_\mu(\{P_i,A_i\})$. Assume further that $\{(\Upsilon_i,\tilde{\Upsilon}_i)\}$ is a set-of framings of $P$ such that $A \in \mathcal{A}_\mu^{\textup{Fr}}(P,\{\Upsilon_i,\tilde{\Upsilon}_i,A_i\})$ and that $A$ is irreducible (i.e. the only gauge transformations on $P$ that preserve $A$ lie in the (finite) center of $G$). Then $\overline{\mathfrak{q}\circ \Phi_{\mathbb{A}}}$ as defined in \autoref{thm: local structure of unframed B Phi is homeo} induces a homeomorphism between an open neighbourhood of $(0,0,[A])$ in \[(B_{\varepsilon}(0))^N \times (\times_iB_\varepsilon^{\mathfrak{m}_i}(0)) \times \big(\mathcal{A}_\mu^{\textup{Fr}}(P,\{\Upsilon_i,\tilde{\Upsilon}_i,A_i\})/\mathcal{G}_{0,\mu+1}\big)\] and an open neighbourhood of $[\mathbb{A}]$ in $\mathcal{B}_\mu(\{P_i,A_i\})$. Consequently, if $[\mathbb{A}]\in \mathcal{M}_\mu(\{P_i,A_i\})$, then $\overline{\mathfrak{q}\circ \Phi_{\mathbb{A}}}$ restricts to a homeomorphism between an open neighbourhood of $(0,0,[A])$ in 
\begin{align*}
\big\{ (\vec{v},\vec{u},[A]) \in B_{\varepsilon}(0))^N \times (\times_iB_\varepsilon^{\mathfrak{m}_i}(0)) \times \big(\mathcal{A}_\mu^{\textup{Fr}}&(P,\{\Upsilon_i,\tilde{\Upsilon}_i,A_i\})/\mathcal{G}_{0,\mu+1}\big) \ \big\vert \ \textup{$A$ satisfies \eqref{equ: SU(3)-instanton}} \\
& \textup{with respect to the $\SU(3)$-structure $f_{\vec{v},\vec{u}}^*(\omega,\Omega)$} \big\} \\
\subset (B_{\varepsilon}(0))^N \times  & (\times_iB_\varepsilon^{\mathfrak{m}_i}(0)) \times \big(\mathcal{A}_\mu^{\textup{Fr}}(P,\{\Upsilon_i,\tilde{\Upsilon}_i,A_i\})/\mathcal{G}_{0,\mu+1}\big)
\end{align*} 
and an open neighbourhood of $[\mathbb{A}]$ in $\mathcal{M}_\mu(\{P_i,A_i\})$.
\end{corollary}

\begin{proof}
The second statement follows directly from the first statement. In order to prove the first, we note that by \autoref{thm: local structure of unframed B Phi is homeo}, $\overline{\mathfrak{q} \circ \Phi_{\mathbb{A}}}$ induces a homeomorphism between an open neighbourhood of $(0,0,[A])$ in \[(B_{\varepsilon}(0))^N \times (\times_iB_\varepsilon^{\mathfrak{m}_i}(0)) \times \big((\mathcal{A}_\mu^{\textup{Fr}}(P,\{\Upsilon_i,\tilde{\Upsilon}_i,A_i\})/\mathcal{G}_{0,\mu+1})\big/\big(\mathcal{G}_{\mu+1}/\mathcal{G}_{0,\mu+1}\big)\big)\] and an open neighbourhood of $[\mathbb{A}]$ in $\mathcal{B}_\mu(\{P_i,A_i\})$. Taking the asymptotic limit $\lim_{\tilde{\Upsilon}_i}$ at each $s_i\in S$ (as defined in \autoref{thm: local structure of B Psi is homeo}) embeds the group $\mathcal{G}_{\mu+1}/\mathcal{G}_{0,\mu+1}$ into the product $\times_i \Stab_{\mathcal{G}(P_i)}(A_i)$ consisting for each $i=1,\dots,N$ of gauge transformations $P_i \to P_i$ fixing $A_i$. If the center $Z(G)$ of the structure group is finite and all tangent cones $A_i$ are infinitesimally irreducible, then $\times_i \Stab_{\mathcal{G}(P_i)}(A_i)$ and therefore $\mathcal{G}_{\mu+1}/\mathcal{G}_{0,\mu+1}$ are finite (because $G$ is compact).

As in the previous theorem one can show that $\mathcal{A}_\mu^{\textup{Fr}}(P,\{\Upsilon_i,\tilde{\Upsilon}_i,A_i\})/\mathcal{G}_{0,\mu+1}$ is Hausdorff. Because $\mathcal{G}_{\mu+1}/\mathcal{G}_{0,\mu+1}$ is finite and the only gauge transformations fixing $A$ lie in $Z(G)$, there exists an open neighbourhood $V \subset \mathcal{A}_\mu^{\textup{Fr}}(P,\{\Upsilon_i,\tilde{\Upsilon}_i,A_i\})/\mathcal{G}_{0,\mu+1}$ of $[A]$ such that $[g]\cdot V \cap V \neq \emptyset$ if and only if $[g] \in Z(G) \subset \mathcal{G}_{\mu+1}/\mathcal{G}_{0,\mu+1}$. This implies the statement.
\end{proof}

\begin{remark}\label{rem: Coulomb gauge and slice for groups with positive center}
In the following we explain how the assumption that $G$ has a finite center and that the tangent connections $A_i$ are infinitesimally irreducible can be removed from \autoref{thm: Coulomb Gauge} and \autoref{thm: Slice Theorem} (cf. \cite[Chapter~I.5]{SoleFarre-thesis}). We begin with the generalisation of \autoref{thm: Coulomb Gauge}: For this, we first assume that the tangents $A_i$ are still infinitesimally irreducible but $G$ has a positive dimensional center $Z(G)$.  As in \autoref{prop: rates of invertibility of Laplacian} one can prove that if $\lambda_i \in (-1,\bar{\mu}_i)$ for all $i=1,\dots,N$, then \[ \textup{index}\big(\Delta_A \col C^{k+1,\alpha}_{\lambda+1} \to C^{k-1,\alpha}_{\lambda-1}\big) = -N \dim(Z(G)) \] (cf. \autoref{app-prop: ker and coker are locally constrant + index change formula}). Moreover, the kernel of $\Delta_A \col C^{k+1,\alpha}_{\lambda+1} \to C^{k-1,\alpha}_{\lambda-1}$ is still trivial (because elements in $\Delta_A$ are constant and $\mathcal{O}(r^{\lambda_i+1})$ around any singularity). Thus, \[ \dim \coker\big(\Delta_A \col C^{k+1,\alpha}_{\lambda+1} \to C^{k-1,\alpha}_{\lambda-1}\big) = N \dim Z(G). \] 

For $N=1$ this gives (by \autoref{app-prop: cokernel kernel pairing}) \[ \coker\big(\Delta_A \col C^{k+1,\alpha}_{\lambda+1} \to C^{k-1,\alpha}_{\lambda-1}\big) \cong \ker\big(\Delta_A \col C^{k+1,\alpha}_{-5-\lambda} \to C^{k-1,\alpha}_{-7-\lambda}\big) = \mathfrak{z} \] where $\mathfrak{z}$ denotes the Lie algebra of $Z(G)$ (canonically embedded into $\Gamma(\mathfrak{g}_P)$). Integration by parts now shows (cf. \autoref{app-prop: cokernel kernel pairing}) that for $N=1$ \[\image\big(\mathfrak{C}\big) \subset \image\big(\Delta_A \col C^{k+1,\alpha}_{\lambda+1} \to C^{k-1,\alpha}_{\lambda-1}\big) \] where $\mathfrak{C}$ denotes the non-linear map appearing in the proof of \autoref{thm: Coulomb Gauge}. This shows that the proof of \autoref{thm: Coulomb Gauge} still holds for structure groups with a positive dimensional center whenever $A$ has a single singularity. 

To address a general number of singular points, we first observe that we are actually interested in the quotient of $\mathcal{A}_\mu^{\textup{Fr}}(P,\{\Upsilon_i,\tilde{\Upsilon}_i,A_i\})$ by the strictly larger gauge group $\mathcal{G}_{\mu+1}$ (cf. \autoref{thm: local structure for instanton moduli space 1}) and that $\mathcal{G}_{\mu+1}/\mathcal{G}_{0,\mu+1} \cong \times_{i} \textup{Stab}_{\mathcal{G}(P_i)}(A_i) \cong \times_i Z(G)$ (via the asymptotic limit map defined in \autoref{thm: local structure of unframed B Phi is homeo} and where the last isomorphism holds because we still assume that the tangents are infinitesimally irreducible).  Thus, considering the Lie algebra (of a suitable Banach version) of $\mathcal{G}_{\mu+1}$ gives $N\dim(Z(G))$ additional dimensions compared to $\mathcal{G}_{0,\mu+1}$ that can be used to overcome the $(N\dim(Z(G)))$-dimensional cokernel of $\Delta_A \col C^{k+1,\alpha}_{\mu+1} \to C^{k-1,\alpha}_{\mu-1}$ (cf. \cite[Proposition~5.6]{SoleFarre-thesis}). One can then show that the cokernel of the linearisation of this extended gauge action is isomorphic to $\mathfrak{z}$, the Lie algebra of $Z(G)$. As for $N=1$ one can use integration by parts to show that the image of the non-linear map $\mathfrak{C}$ is contained in the image of its linearisation and apply the proof of \autoref{thm: Coulomb Gauge}.

This shows that the condition that the center of $G$ is finite can be dropped, if we divide out (a suitable Banach version of) the larger gauge group $\mathcal{G}_{\mu+1}$. Similarly, one can show that in this situation the condition that the tangents are infinitesimally irreducible can also be dropped, when one assumes $A$ to be infinitesimally irreducible instead (cf. \cite[Theorem~5.7]{SoleFarre-thesis}).

The generalisation of \autoref{thm: Slice Theorem} to structure groups whose center is not finite and to tangent-connections that are not infinitesimally irreducible  (if one assumes $A$ to be infinitesimally irreducible instead) is proven similarly: for instantons with a single singularity, this is proven as in \autoref{thm: Slice Theorem} (using the discussion on the extension of \autoref{thm: Coulomb Gauge} in the previous paragraph). For connections with a larger number of singular points, one again needs to divide $\mathcal{A}^{k,\alpha}_\mu(P,\{\Upsilon_i,\tilde{\Upsilon}_i,A_i\})$ by (a suitable Banach version of) the larger group $\mathcal{G}_{\mu+1}$ and adapt the proof of \autoref{thm: Slice Theorem}.
\end{remark}

\subsubsection{Kuranishi charts for the moduli space}\label{sub: Kuranishi charts}
Throughout this section, $Z^6$ is a compact 6-manifold with an $\SU(3)$-structure $(\omega,\Omega)$ that satisfies $\diff^* \omega=0$ and $\diff \Omega = w_1 \omega^2$ for some $w_1 \in \mathbb{R}$ (cf. \autoref{rmk: convenient class of SU(3)-structures} and \autoref{prop: modifying the SU(3)-structure such that dIm Omega = 0}). Moreover, fix $N\in \mathbb{N}$ and for every $i=1,\dots, N$ a bundle with connection $(\pi_i \col P_i \to S^5,A_i)$ where each $A_i \in \mathcal{A}(P_i)$ is infinitesimally irreducible and satisfies \eqref{equ: cone reduction of instanton equation}. In this section, we show that $\mathcal{M}_{\mu}(\{P_i,A_i\})$ is for any rate $\mu \in \times_i (-1,\bar{\mu}_i)$ (where $\bar{\mu}_i$ is as in \autoref{prop: rates of invertibility of Laplacian}) locally homeomorphic to the zero set of a smooth map between finite dimensional vector spaces.  To be consistent with the previous section we will again assume that the center of $G$ is discrete, but we note once more that this assumption (and similarly, the assumption that all $A_i$ are infinitesimally irreducible) can be removed (cf. \cite[Chapter~I.5]{SoleFarre-thesis}). We first need the following auxiliary proposition which allows us to replace the (Fréchet) space of smooth connections $\mathcal{A}_\mu^{\textup{Fr}}(P,\{\Upsilon,\tilde{\Upsilon}_i,A_i\})$ by the Banach space $\mathcal{A}^{k,\alpha}_\mu(P,\{\Upsilon,\tilde{\Upsilon}_i,A_i\})$.

\begin{proposition}\label{prop: local structure of moduli space for fixed bundle}
Let $\pi \col P \to Z \setminus S$ be a principal $G$-bundle and let $A \in \mathcal{A}_\mu^{\textup{Fr}}(P,\{\Upsilon_i,\tilde{\Upsilon}_i,A_i\})$ be a conically singular $\SU(3)$-instanton, where the rate $\mu=(\mu_1,\dots,\mu_N)$ is chosen such that $-1< \mu_i< \bar{\mu}_i$ for every $i= 1, \dots, N$ (with $\bar{\mu}_i$ as in \autoref{prop: rates of invertibility of Laplacian}). There exists an open neighbourhood of $[A]$ in \[\mathcal{M}_\mu^\Fr(P,\{\Upsilon_i,\tilde{\Upsilon}_i,A_i\})\coloneqq \big\{ [A] \in \mathcal{A}_\mu^\Fr(P,\{\Upsilon_i,\tilde{\Upsilon}_i,A_i\})/\mathcal{G}_{0,\mu+1} \mid \textup{$A$ satisfies \eqref{equ: SU(3)-instanton}} \big\} \] that is homeomorphic to 
\begin{align*}
\mathcal{S}^{k,\alpha}_A &\coloneqq \big\{A^\prime \in \mathcal{A}_\mu^{k,\alpha}(P,\{\Upsilon_i,\tilde{\Upsilon}_i,A_i\}) \mid \textup{$A^\prime$ satisfies~\eqref{equ: SU(3)-instanton}, $\Vert A^\prime -A \VertWH{k}{\mu}{}<\varepsilon,$} \\
& \qquad \qquad \qquad \qquad \qquad \qquad\qquad\qquad\qquad\qquad\qquad\qquad \textup{and $\diff_A^*(A^\prime-A)=0$} \big\} \\
&\subset \mathcal{A}_\mu^{k,\alpha}(P,\{\Upsilon_i,\tilde{\Upsilon}_i,A_i\})
\end{align*}
for any $k\geq 1$ and $\alpha \in (0,1)$ and a sufficiently small $\varepsilon\equiv \varepsilon(A,k,\alpha)>0$.
\end{proposition}

\begin{proof}
Any connection $A^\prime \in \mathcal{A}_\mu^{k,\alpha}(P,\{\Upsilon_i,\tilde{\Upsilon}_i,A_i\})$ can be written as $A+a^\prime$ where $a^\prime \in C^{k,\alpha}_{\mu}(Z\setminus S,T^*Z\otimes \mathfrak{g}_P)$. If such $A^\prime$ satisfies the equations~\eqref{equ: SU(3)-instanton} and $\diff_A^*(A^\prime-A)=0$, then by the discussion prior to \autoref{prop: adjoint of model operator} (with $L_A$ and $Q_A$ as defined in said discussion) \[ L_{A}(0,0,a^\prime) = -Q_A(0,0,a^\prime) = -\tfrac{1}{2}(0,\Lambda_\omega [a^\prime\wedge a^\prime], *([a^\prime \wedge a^\prime] \wedge \Im \Omega)). \] Since $L_A$ is elliptic, bootstrapping and elliptic estimates imply that $a^\prime \in C^\infty_\mu(Z\setminus S, T^*Z \otimes \mathfrak{g}_P)$ (cf. \autoref{app-prop: weighted Schauder estimate}). Hence, there exists a well-defined map $\Psi$ from \[ \big\{A^\prime \in \mathcal{A}_\mu^{k,\alpha}(P,\{\Upsilon_i,\tilde{\Upsilon}_i,A_i\}) \mid \textup{$A^\prime$ satisfies~\eqref{equ: SU(3)-instanton}, $\Vert A^\prime -A \VertWH{k}{\mu}{}<\varepsilon,$ and $\diff_A^*(A^\prime-A)=0$}\big\} \] to $\mathcal{M}_\mu^{\textup{Fr}}(P,\{\Upsilon_i,\tilde{\Upsilon}_i,A_i\})$ mapping $A^\prime$ to $[A^\prime]$. By \autoref{thm: Slice Theorem} this map is injective for small enough $\varepsilon>0$. In order to prove that $\Psi$ is surjective, we first note that again by \autoref{thm: Slice Theorem} there exists a neighbourhood $ V\subset \mathcal{M}_\mu^{\textup{Fr}}(P,\{\Upsilon_i,\tilde{\Upsilon}_i,A_i\})$ of $[A]$ such that for every $A^\prime \in [A^\prime]\in V$ there exists a unique $g^\prime \in \mathcal{G}^{k+1,\alpha}_{0,\mu +1}$ such that $(g^\prime)^*A^\prime$ satisfies $\Vert (g^\prime)^*A^\prime - A \VertWH{k}{\mu}{}<\varepsilon$ and $\diff_A^*((g^\prime)^*A^\prime -A)=0$. Since $(g^\prime)^*A^\prime$ still satisfies~\eqref{equ: SU(3)-instanton}, elliptic regularity implies $(g^\prime)^*A^\prime\in \mathcal{A}^{\textup{Fr}}_\mu(P,\{\Upsilon_i,\tilde{\Upsilon}_i,A_i\}).$ Bootstrapping via \eqref{equ: bootstrapping gauge transformations} as in the proof of \autoref{thm: Slice Theorem} then gives $g^\prime\in \mathcal{G}_{0,\mu+1}$ and therefore $\Psi((g^\prime)^*A^\prime)=[A^\prime]$. This proves that $\Psi$ is bijective. That $\Psi$ is, in fact, an homeomorphism is again a consequence of elliptic estimates (cf. \autoref{app-prop: weighted Schauder estimate}) and \autoref{thm: Slice Theorem}.
\end{proof}

\begin{remark}\label{rem: extension of local stucture of moduli space for fixed bundle to non instanton slice}
The condition that $A$ is an $\SU(3)$-instanton in the previous proposition was only used for convenience so that $[A] \in \mathcal{M}_\mu^{\Fr}(P,\{\Upsilon_i,\tilde{\Upsilon}_i,A_i\})$. More generally, one can also center the slice $\mathcal{S}_A^{k,\alpha}$ around any conically singular connection $A \in \mathcal{A}_\mu^{\Fr}(P,\{\Upsilon_i,\tilde{\Upsilon}_i,A_i\})$ that is not an instanton and obtain via the same proof a homeomorphism onto a (possibly empty) open subset of $\mathcal{M}_\mu^{\Fr}(P,\{\Upsilon_i,\tilde{\Upsilon}_i,A_i\})$.
\end{remark}

The following equivariant version of the previous proposition follows by direct inspection of the homeomorphism constructed in the previous proof.

\begin{corollary}\label{cor: local equivariant structure of moduli space for fixed bundle}
Let $[A] \in \mathcal{M}_\mu^{\Fr}(P,\{\Upsilon_i,\tilde{\Upsilon}_i,A_i\})$ and 
\begin{align*}
\mathcal{S}^{k,\alpha}_A &\coloneqq \big\{A^\prime \in \mathcal{A}_\mu^{k,\alpha}(P,\{\Upsilon_i,\tilde{\Upsilon}_i,A_i\}) \mid \textup{$A^\prime$ satisfies~\eqref{equ: SU(3)-instanton}, $\Vert A^\prime -A \VertWH{k}{\mu}{}<\varepsilon,$} \\
& \qquad \qquad \qquad \qquad \qquad \qquad\qquad\qquad\qquad\qquad\qquad\qquad \textup{and $\diff_A^*(A^\prime-A)=0$} \big\} \\
&\subset \mathcal{A}_\mu^{k,\alpha}(P,\{\Upsilon_i,\tilde{\Upsilon}_i,A_i\})
\end{align*}
be as in the previous proposition. Moreover, let $\Stab_{\mathcal{G}}(A) \subset \mathcal{G}_{\mu+1}$ be the group of gauge transformations $g \col P \to P$ preserving $A$. We may choose $\mathcal{S}_A^{k,\alpha}$ to be $\Stab_{\mathcal{G}}(A)$-invariant (either by restricting to an open subset of the slice or by using the connection $A$ in the definition of the $\Vert \cdot \VertWH{k}{\mu}{}$-norm in \autoref{def: weighted Hölder norms}). The open neighbourhood $V_{[A]} \subset \mathcal{M}_\mu^{\Fr}(P,\{\Upsilon_i,\tilde{\Upsilon}_i,A_i\})$ of $[A]$ in the previous proposition can then be chosen such that \[ [g] \cdot V_{[A]} \cap V_{[A]} \neq \emptyset \textup{ for $[g] \in \mathcal{G}_{\mu+1}/\mathcal{G}_{0,\mu+1}$ if and only if $[g] \in (\Stab_{\mathcal{G}}(A)\cdot \mathcal{G}_{0,\mu+1})/\mathcal{G}_{0,\mu+1}$.}\] Moreover, the homeomorphism constructed in the previous proposition is $\Stab_{\mathcal{G}}(A)$-equivariant. This implies that an open neighbourhood of $[A]$ in $\mathcal{M}_\mu^{\Fr}(P,\{\Upsilon_i,\tilde{\Upsilon}_i,A_i\})/(\mathcal{G}_{\mu+1}/\mathcal{G}_{0,\mu+1})$ is homeomorphic to $\mathcal{S}^{k,\alpha}_A/\Stab_{\mathcal{G}}(A)$.
\end{corollary}

For any conically singular $\SU(3)$-instanton $\mathbb{A}\in \mathcal{A}_\mu(\{P_i,A_i\})$ we now define a smooth Fredholm map between Banach spaces whose zero locus parametrises a neighbourhood of $[\mathbb{A}]$ in $\mathcal{M}_\mu(\{P_i,A_i\})$. The existence of a Kuranishi chart follows then from the theory of non-linear Fredholm maps (cf. \cite[Section~4.2.4]{DonaldsonKronheimer-4-manifolds}).

\begin{definition}\label{def: non-linear Fredholm map and its linearisation}
Let $\mathbb{A} \coloneqq (S,\pi \col P \to Z\setminus S,A)\in \mathcal{A}_{\mu}(\{P_i,A_i\})$ be a conically singular $\SU(3)$-instanton and let $\{(\Upsilon_i,\tilde{\Upsilon}_i)\}$ be a set of framings for $A$. Moreover, let $k\geq 1$ and $\alpha \in (0,1)$ be fixed. As in \autoref{def: family of diffeomorphism parametrising open neighbourhood} we choose $\mathfrak{m}_i\subset \mathfrak{su}(3)$ to be a complement of $\textup{image}(\mathfrak{stab}_{\SU(3)}(A_i) \to \mathfrak{su}(3))$ and denote by $B_\varepsilon^{\mathfrak{m}_i}(0) \subset \mathfrak{m}_i$ and $B_{\varepsilon}(0)\subset \mathbb{C}^3$ for $\varepsilon>0$ the respective open $\varepsilon$-balls. Recall also from \autoref{def: family of diffeomorphism parametrising open neighbourhood} the families of vector fields $\mathfrak{vec}_{0/1/2}$ and for $(\vec{v},\vec{u}) \in (B_{\varepsilon}(0))^N \times (\times_i B_{\varepsilon}^{\mathfrak{m}_i}(0))$ the corresponding diffeomorphism $f_{\vec{v},\vec{u}} \col Z \to Z$.
We now define the following maps:
\begin{align*}
\mathfrak{F}_{\mathbb{A}}^{k,\alpha} \col (B_{\varepsilon}(0))^N  \times (\times_i B_{\varepsilon}^{\mathfrak{m}_i}(0)) &\times C^{k,\alpha}_\mu(\mathfrak{g}_P \oplus \mathfrak{g}_P \oplus T^*Z \otimes\mathfrak{g}_P) \to C_{\mu-1}^{k-1,\alpha}(\mathfrak{g}_P \oplus \mathfrak{g}_P \oplus T^*Z \otimes\mathfrak{g}_P) \\
(\vec{v},\vec{u}, \xi_1,\xi_2, a) & \mapsto \big(\diff_A^* a, \Lambda_{f_{\vec{v},\vec{u}}^*\omega} F_{A+a}, \\
& \qquad *_{f_{\vec{v},\vec{u}}^*g}(F_{A+a} \wedge f_{\vec{v},\vec{u}}^*\Im \Omega) + \diff_{A+a} \xi_1 - (f_{\vec{v},\vec{u}}^*J)^*(\diff_{A+a} \xi_2)\big)
\end{align*}
and
\begin{align*}
\mathfrak{L}_{\mathbb{A}}^{k,\alpha} \col (\mathbb{C}^3)^N  \oplus ( \oplus_i \mathfrak{m}_i) &\oplus  C^{k,\alpha}_\mu(\mathfrak{g}_P \oplus \mathfrak{g}_P \oplus T^*Z \otimes\mathfrak{g}_P) \to C_{\mu-1}^{k-1,\alpha}(\mathfrak{g}_P \oplus \mathfrak{g}_P \oplus T^*Z \otimes\mathfrak{g}_P) \\
(\vec{v},\vec{u}, \xi_1,\xi_2, a) & \mapsto \big(\diff_A^* a, \Lambda_{\omega} \diff_{A}a + *(\diff (i_{X(\vec{v},\vec{u})} \omega \wedge \omega) \wedge F_A), \\
& \qquad *(\diff_A a \wedge \Im \Omega + F_A \wedge \diff i_{X(\vec{v},\vec{u})}\Im \Omega) + \diff_{A} \xi_1 - J^*(\diff_{A} \xi_2)\big)
\end{align*}
where the vector field $X(\vec{v},\vec{u}) \in \Gamma(TZ)$ is defined by 
\begin{align*}
X(\vec{v},\vec{u}) &\coloneqq (\partial_{\vec{v}}\mathfrak{vec}_2)(0) + (\partial_{\vec{v}}\mathfrak{vec}_1)(0) + (\partial_{\vec{u}}\mathfrak{vec}_0)(0) \\
&= (\partial_{\vec{v}}\mathfrak{vec}_2)(0) + \mathfrak{vec}_1(\vec{v}) + \mathfrak{vec}_0(\vec{u})
\end{align*}
(and where we regard the respective derivative of $\mathfrak{vec}_{0,1,2}$ at $0$ again as a vector field on $Z$).
\end{definition}

\begin{proposition}\label{prop: non-linear Fredholm map linearisation and index}
Let $\mathbb{A}\in \mathcal{A}_\mu(\{P_i,A_i\})$ be a conically singular $\SU(3)$-instanton. Additionally, let $\mathfrak{F}^{k,\alpha}_{\mathbb{A}}$ and $\mathfrak{L}^{k,\alpha}_{\mathbb{A}}$ be as in the previous definition. Then $\mathfrak{F}^{k,\alpha}_{\mathbb{A}}$ is a well-defined and smooth map with linearisation at zero given by $D_0 \mathfrak{F}^{k,\alpha}_\mathbb{A} = \mathfrak{L}^{k,\alpha}_\mathbb{A}$. Moreover, $\mathfrak{L}^{k,\alpha}_{\mathbb{A}}$ is Fredholm with Fredholm index given by 
\begin{align*}
\textup{index}\big(\mathfrak{L}^{k,\alpha}_\mathbb{A}\big) &=   6N + \sum_{i=1}^N \dim(\mathfrak{m}_i) - \hspace*{-20pt} \sum_{\scriptscriptstyle \nu_i \in \mathcal{D}(L_{A_i}) \cap (-5/2,\mu_i)} \hspace*{-20pt} \dim \mathcal{K}(L_{A_i})_{\nu_i},
\end{align*}
where for $\nu_i \in \mathbb{R}$ 
\begin{align}\label{equ: definition homogeneous kernel}
\mathcal{K}(L_{A_i})_{\nu_i} \coloneqq \big\{\underline{a} \in \ker(L_{\pr_{S^5}^*A_i}) \big\vert \textup{ $\underline{a}$ is homogeneous of rate $\nu_i$} \big\}.
\end{align}
When we equip the spaces $C^{k(-1),\alpha}_{\mu(-1)}(\mathfrak{g}_P \oplus \mathfrak{g}_P \oplus T^*Z \otimes\mathfrak{g}_P)$ with the obvious $\Stab_{\mathcal{G}}(A)$-action (where $\Stab_{\mathcal{G}}(A)$ is as in the previous corollary), then $\mathfrak{F}_\mathbb{A}^{k,\alpha}$ and $\mathfrak{L}_{\mathbb{A}}^{k,\alpha}$ become equivariant maps.
\end{proposition}

\begin{proof}
By construction we have $f_{\vec{v},\vec{u}}^*(\omega, \Omega)_{f_{\vec{v},\vec{u}}(s_i)} = (\omega, \Omega)_{s_i}$ for every $(\vec{v},\vec{u}) \in (B_\varepsilon(0))^N \times (\times_i B_{\varepsilon}^{\mathfrak{m}_i}(0))$ and $s_i \in S$. Since $\mu_i>-1$ it is therefore not difficult to see that $\mathfrak{F}^{k,\alpha}_\mathbb{A}$ indeed maps into $C^{k-1,\alpha}_{\mu-1}$. Similarly, a moment's thought shows that $\mathfrak{F}^{k,\alpha}_\mathbb{A}$ and therefore also $\mathfrak{L}^{k,\alpha}_\mathbb{A}$ are $\Stab_{\mathcal{G}}(A)$-equivariant. The smoothness of $\mathfrak{F}^{k,\alpha}_\mathbb{A}$ follows from the smoothness of $(\vec{v},\vec{u}) \mapsto f_{\vec{v},\vec{u}} \in \textup{Diff}(Z)$ and the fact that exterior and interior products of differential forms are smooth operations. 

In order to linearise $\mathfrak{F}^{k,\alpha}_\mathbb{A}$ we first note the identity $\Lambda_{f^*_{\vec{v},\vec{u}}\omega} F_{A+a} = \tfrac{1}{2}*_{f^*_{\vec{v},\vec{u}}\omega}(F_{A+a} \wedge f^*_{\vec{v},\vec{u}}(\omega \wedge \omega))$ (cf. \cite[Proposition~1.2.30]{Huybrechts-complex-Geometry}). The derivation of $D_0 \mathfrak{F}^{k,\alpha}_\mathbb{A} = \mathfrak{L}^{k,\alpha}_\mathbb{A}$ is then straight forward and makes use of the closedness of $\omega \wedge \omega$ and $\Im \Omega$ and the assumption that $A$ is an $\SU(3)$-instanton.

In order to prove that $\mathfrak{L}^{k,\alpha}_\mathbb{A}$ is Fredholm we first note that its restriction $\mathfrak{L}^{k,\alpha}_\mathbb{A}(0,0,\cdot)$ equals 
\begin{equation}\label{equ: Fredholmnes auxilary equation}
L_A \col C^{k,\alpha}_\mu(Z\setminus S,\mathfrak{g}_P \oplus \mathfrak{g}_P \oplus T^*Z \otimes \mathfrak{g}_P) \mapsto C^{k-1,\alpha}_{\mu-1}(Z \setminus S,\mathfrak{g}_P \oplus \mathfrak{g}_P \oplus T^*Z \otimes \mathfrak{g}_P) 
\end{equation} 
(where $L_A$ was defined prior to \autoref{prop: adjoint of model operator}). Since the rate $\mu \in \mathbb{R}^N$ does not lie in \[\mathcal{D}(L) \coloneqq\big\{(\lambda_1,\dots,\lambda_N) \in \mathbb{R}^N \ \big\vert \ \lambda_i \in \mathcal{D}(L_{s_i}) \textup{ for some $i \in \{1,\dots,N\}$} \big\},\] the operator \eqref{equ: Fredholmnes auxilary equation} is Fredholm (cf. \autoref{app-prop: Fredholm away from critical rates}). Because $(\mathbb{C}^3)^N \oplus (\oplus_i \mathfrak{m}_i)$ is finite dimensional, this implies the Fredholm property of $\mathfrak{L}^{k,\alpha}_\mathbb{A}$. Moreover, a moment's thought reveals \[\textup{index}\big(\mathfrak{L}^{k,\alpha}_\mathbb{A}\big) = \textup{index}\big(L_A\col C^{k,\alpha}_\mu \to C^{k-1,\alpha}_{\mu-1}\big) + 6N+ \textstyle{\sum} \dim(\mathfrak{m}_i) . \] 

It therefore remains to prove the formula 
\begin{align}\label{equ: auxiliary formula for Fredholm index}
\textup{index}\big(L_A\col C^{k,\alpha}_\mu \to C^{k-1,\alpha}_{\mu-1}\big) = -\sum_{i=1}^N \sum_{\scriptscriptstyle \nu_i \in \mathcal{D}(L_{A_i}) \cap (-5/2,\mu_i)} \hspace*{-20pt} \dim \mathcal{K}(L_{A_i})_{\nu_i}.
\end{align} 
The formal self-adjointness of $L_A$ (cf. \autoref{prop: adjoint of model operator}) implies that \[\textup{index}\big(L_A\col C^{k,\alpha}_{-5/2} \to C^{k-1,\alpha}_{-5/2-1}\big) = 0\] (cf. \autoref{app-prop: cokernel kernel pairing}) and \eqref{equ: auxiliary formula for Fredholm index} follows from the theory of elliptic operators on weighted spaces (cf. \autoref{app-prop: ker and coker are locally constrant + index change formula}).
\end{proof}

\begin{theorem}\label{thm: Kuranishi structure}
Let $[\mathbb{A}]\coloneqq [(S,\pi \col P \to Z\setminus S,A)]$ be an element in $\mathcal{M}_\mu(\{P_i,A_i\})$, where the rate $\mu =(\mu_1,\dots,\mu_N)$ satisfies $-1< \mu_i< \bar{\mu}_i$ (with $\bar{\mu}_i$ as in \autoref{prop: rates of invertibility of Laplacian}). Furthermore, let $\mathfrak{L}^{k,\alpha}_{\mathbb{A}}$ be as above. Then there exists a smooth $\Stab_{\mathcal{G}}(A)$-equivariant map 
\begin{align*}
\ob_{\mathbb{A}} \col \ker \mathfrak{L}^{k,\alpha}_{\mathbb{A}} \to \coker \mathfrak{L}^{k,\alpha}_{\mathbb{A}}
\end{align*}
(where $\Stab_{\mathcal{G}}(A)$ is as in \autoref{cor: local equivariant structure of moduli space for fixed bundle}) with $\ob_{\mathbb{A}}(0)=0$ and an $\Stab_{\mathcal{G}}(A)$-invariant open neighbourhood $V_{\mathbb{A}}\subset \ob^{-1}(0)$ of $0$  such that a neighbourhood of $[\mathbb{A}]$ in $\mathcal{M}_\mu(\{P_i,A_i\})$ is homeomorphic to $V_{\mathbb{A}}/\Stab_{\mathcal{G}}(A)$. Moreover, 
\begin{align*}
\textup{virt-dim}(V_{\mathbb{A}}) &\coloneqq \textup{index} \big(\mathfrak{L}^{k,\alpha}_{\mathbb{A}}\big) \\
&= 6N + \sum_{i=1}^N \dim(\mathfrak{m}_i) - \hspace*{-20pt} \sum_{\scriptscriptstyle \nu_i \in \mathcal{D}(L_{A_i}) \cap (-5/2,\mu_i)} \hspace*{-20pt} \dim \mathcal{K}(L_{A_i})_{\nu_i}.
\end{align*}
\end{theorem}
\begin{remark}
Note that the asymptotic limit map $\lim_{\tilde{\Upsilon}_{i_0}}$ (as defined in \autoref{thm: local structure of B Psi is homeo}) embeds $\Stab_{\mathcal{G}}(A)$ for any $i_0 = 1,\dots,N$ into $\Stab_{\mathcal{G}(P_{i_0})}(A_{i_0})$. Since all tangent connections $A_i \in \mathcal{A}(P_i)$ are assumed to be infinitesimally irreducible and the center of $G$ is finite, $\Stab_{\mathcal{G}(P_{i_0})}(A_{i_0})$ and therefore also $\Stab_{\mathcal{G}}(A)$ are discrete and therefore finite.
\end{remark}

\begin{remark}\label{rem: Kuranishi chart for groups with positive center}
Using \autoref{rem: Coulomb gauge and slice for groups with positive center} one can show that the previous theorem can be extended to any compact structure group $G$ (whose center might be positive dimensional). Moreover, the assumption that the tangent connections $A_i$ are infinitesimally irreducible can also be dropped, if one instead assumes $A$ to be infinitesimally irreducible. This appeared in more detail in \cite[Chapter~I.5]{SoleFarre-thesis}.
\end{remark}

\begin{proof}[Proof of \autoref{thm: Kuranishi structure}]
Since all tangent cone connections $A_i \in \mathcal{A}(P_i)$ are infinitesimally irreducible, \autoref{prop: rates of invertibility of Laplacian} implies that the connection $A+a$ is also infinitesimally irreducible for every $a \in C^{k,\alpha}_\mu(Z\setminus S,T^*Z\otimes \mathfrak{g}_P)$. This implies $\ker(\diff_{A+a}) =0$. \autoref{thm: local structure for instanton moduli space 1}, an extension of \autoref{prop: local structure of moduli space for fixed bundle} and \autoref{cor: local equivariant structure of moduli space for fixed bundle} to variable $\SU(3)$-structures, and \autoref{prop: augmented instanton equation} imply therefore that a neighbourhood of $[\mathbb{A}]$ in $\mathcal{M}_\mu(\{P_i,A_i\})$ is homeomorphic to $V^\prime_{\mathbb{A}}/\Stab_{\mathcal{G}}(P)$ where $V^\prime_{\mathbb{A}} \subset (\mathfrak{F}^{k,\alpha}_{\mathbb{A}})^{-1}(0)$ is a $\Stab_{\mathcal{G}}(P)$-invariant open neighbourhood of zero (with $\mathfrak{F}^{k,\alpha}_{\mathbb{A}}$ as in the previous definition). 

Since $\mathfrak{F}^{k,\alpha}_{\mathbb{A}}$ is a smooth non-linear and equivariant Fredholm map with linearisation $\mathfrak{L}^{k,\alpha}_{\mathbb{A}}$, the rest of the proof now follows from standard arguments (cf. \cite[Section~4.2.4]{DonaldsonKronheimer-4-manifolds}): Choose closed $\Stab_{\mathcal{G}}(A)$-invariant complements $\coimage(\mathfrak{L}^{k,\alpha}_\mathbb{A})$ and $\coker(\mathfrak{L}^{k,\alpha}_\mathbb{A})$ of $\ker(\mathfrak{L}^{k,\alpha}_\mathbb{A})$ and $\image(\mathfrak{L}^{k,\alpha}_\mathbb{A})$ in the domain and codomain of $\mathfrak{L}^{k,\alpha}_\mathbb{A}$, respectively. With respect to the corresponding decomposition, we write 
\begin{align*}
\mathfrak{F}^{k,\alpha}_{\mathbb{A}} \col V^\prime_{\mathbb{A}} \subset \ker\big(\mathfrak{L}^{k,\alpha}_{\mathbb{A}}\big) \oplus \coimage\big(\mathfrak{L}^{k,\alpha}_{\mathbb{A}}\big) &\to \coker\big(\mathfrak{L}^{k,\alpha}_{\mathbb{A}}\big) \oplus \image \big(\mathfrak{L}^{k,\alpha}_{\mathbb{A}}\big) \\
(w_1,w_2) &\mapsto \big((\mathfrak{F}^{k,\alpha}_{\mathbb{A}})_1(w_1,w_2),(\mathfrak{F}^{k,\alpha}_{\mathbb{A}})_2(w_1,w_2)\big).
\end{align*}
The Implicit Function Theorem gives rise to $\Stab_{\mathcal{G}}(A)$-invariant open subsets \[V_\mathbb{A}\subset \ker\big(\mathfrak{L}^{k,\alpha}_{\mathbb{A}}\big) \quad \tilde{V}_\mathbb{A} \subset \coimage\big(\mathfrak{L}^{k,\alpha}_{\mathbb{A}}\big) \quad \textup{with} \quad (0,0) \in V_\mathbb{A} \times \tilde{V}_\mathbb{A} \subset V_\mathbb{A}^\prime \] and a smooth $\Stab_{\mathcal{G}}(A)$-equivariant map $\mathfrak{h}_\mathbb{A} \col V_\mathbb{A} \to \tilde{V}_\mathbb{A}$ with $\mathfrak{h}_\mathbb{A}(0)=0$ such that \[ (\mathfrak{F}^{k,\alpha}_{\mathbb{A}})_2^{-1}(0) \cap (V_{\mathbb{A}}\times \tilde{V}_\mathbb{A}) = \{(w_1,\mathfrak{h}_\mathbb{A}(w_1)) \mid w_1 \in V_\mathbb{A}\}. \] It is now not difficult to see that (after replacing $V_\mathbb{A}^\prime$ by $V_{\mathbb{A}}\times \tilde{V}_{\mathbb{A}}$)
\begin{align*}
\ob_\mathbb{A} \col V_\mathbb{A}\subset \ker\big(\mathfrak{L}^{k,\alpha}_{\mathbb{A}}\big) &\to \coker\big(\mathfrak{L}^{k,\alpha}_{\mathbb{A}}\big)\\
w_1 &\mapsto (\mathfrak{F}^{k,\alpha}_{\mathbb{A}})_1(w_1,\mathfrak{h}_\mathbb{A}(w_1))
\end{align*}
satisfies the wanted property. The formula for the (local) virtual dimension of $\mathcal{M}_\mu(\{P_i,A_i\})$ around $[\mathbb{A}]$ follows from the previous proposition.
\end{proof}

\subsection{Relating Moduli spaces for different rates}

As in the previous sections, we fix $N\in \mathbb{N}$ and for every $i=1,\dots,N$ a bundle with connection $(\pi_i\col P_i \to S^5,A_i)$ where each $A_i \in \mathcal{A}(P_i)$ is infinitesimally irreducible and satisfies \eqref{equ: cone reduction of instanton equation}. Moreover, we again assume that $G$ is compact and has a discrete (hence finite) center (but see also \autoref{rem: Coulomb gauge and slice for groups with positive center} and \autoref{rem: Kuranishi chart for groups with positive center}).

If $\mu,\nu \in \mathbb{R}^N$ are rates with $\nu_i\leq\mu_i$ for every $i=1,\dots,N$, then there is an obvious inclusion of framed conically singular connections $\mathcal{A}_\mu^\Fr(P,\{\Upsilon_i,\tilde{\Upsilon}_i,A_i\}) \subset \mathcal{A}_\nu^\Fr(P,\{\Upsilon_i,\tilde{\Upsilon}_i,A_i\})$ on a fixed framed bundle $\pi \col P \to Z \setminus S$. This induces a continuous inclusion $\mathcal{M}_\mu(\{P_i,A_i\}) \subset \mathcal{M}_\nu(\{P_i,A_i\})$. In this section we shall sketch a proof that this inclusion is, in fact, a homeomorphism for rates which lie in the open cube $ (-1,\Bar{\mu}_1)\times \dots \times(-1,\Bar{\mu}_N)$ with $\bar{\mu}_i$ as in \autoref{prop: rates of invertibility of Laplacian}.

For this we start with the following result:

\begin{proposition}
Let $\pi \col P \to Z\setminus S$ be a fixed framed principal $G$-bundle and let \[\mathcal{M}_\mu^\Fr(P,\{\Upsilon_i,\tilde{\Upsilon}_i,A_i\})\coloneqq \big\{ [A] \in \mathcal{A}_\mu^\Fr(P,\{\Upsilon_i,\tilde{\Upsilon}_i,A_i\})/\mathcal{G}_{0,\mu+1} \mid \textup{$A$ satisfies \eqref{equ: SU(3)-instanton}} \big\} \] be as in \autoref{prop: local structure of moduli space for fixed bundle}. For any $\mu,\nu \in (-1,\Bar{\mu}_1)\times \dots \times(-1,\Bar{\mu}_N) $ with $\Bar{\mu}_i$ as in \autoref{prop: rates of invertibility of Laplacian} there exists a homeomorphism $\Psi \col \mathcal{M}_\mu^\Fr(P,\{\Upsilon_i,\tilde{\Upsilon}_i,A_i\}) \to \mathcal{M}_\nu^\Fr(P,\{\Upsilon_i,\tilde{\Upsilon}_i,A_i\})$.
\end{proposition}

\begin{proof}
For simplicity, we will assume that $N=1$ and $\mu<\nu$. The general case is similar. 

Let $A\in \mathcal{A}_\mu^\Fr(P,\{\Upsilon_i,\tilde{\Upsilon}_i,A_i\})$ be an $\SU(3)$-instanton and let $-1<\mu^\prime<\mu$, $k\geq 1$, and $\alpha\in (0,1)$ be fixed. It is not difficult to see that for any $\varepsilon>0$, there exists a connection $B\in \mathcal{A}_{\mu^\prime}^\Fr(P\{\Upsilon_i,\tilde{\Upsilon}_i,A_i\})$ that agrees with $(\tilde{\Upsilon}_s)_*(\pr_{S^5}^*A_s)$ in a small neighbourhood of $s \in S$ and that satisfies \[ \Vert A - B \VertWH{k}{\mu^\prime}{} \leq \varepsilon. \] When $\varepsilon$ is sufficiently small, \autoref{thm: Coulomb Gauge} implies the existence of a $g \in \mathcal{G}^{k+1,\alpha}_{0,\mu^\prime+1}$ such that $\diff_{A}^*((g^{-1})^*B-A)=0$ or, equivalently, $\diff_{B}^*(g^*A-B)=0$. 

Next, we write $g^*A-B = a \in C^{k,\alpha}_{\mu^\prime}(Z\setminus S, T^*Z\otimes \mathfrak{g}_P)$. The $\SU(3)$-instanton equation on $g^*A$ can now be written as 
\begin{equation}\label{equ: auxiliary equation for improving decay}
L_B(0,0,a) = - \Theta_B(0) - Q_B(0,0,a), 
\end{equation} 
where $\Theta_B$, $L_B$, and $Q_B$ are as defined prior to \autoref{prop: adjoint of model operator}. Since $B$ coincides with $(\tilde{\Upsilon}_s)_*(\pr_{S^5}^*A_s)$ on an open neighbourhood of $s$, we have \[ \vert \nabla^{i}(\Theta_B(0)+Q_B(0,0,a)) \vert = \mathcal{O}(r^{-1-i} + r^{2\mu^\prime-i})\quad \textup{ for $i=0,\dots,k$.} \] Since $-1<\mu^\prime<0$, the right-hand side of \eqref{equ: auxiliary equation for improving decay} blows up slower than expected. The rate of $a$ can therefore iteratively be improved until $a \in C^{k,\alpha}_{\Bar{\mu}}(Z\setminus S, T^*Z \otimes \mathfrak{g}_P)$ (cf. \autoref{app-cor: decay improvement CS operators not crossing rate} and \autoref{app-prop: decay improvement CS operators when crossing rate}). Moreover, $a$ is smooth by elliptic regularity. We now define \[\Psi(A) \coloneqq g^*A = B+a \in \mathcal{A}_{\Bar{\mu}}(P,\{\Upsilon_i,\tilde{\Upsilon}_i,A_i\})\subset \mathcal{A}_{\nu}(P,\{\Upsilon_i,\tilde{\Upsilon}_i,A_i\}).\] 

Next, we show that up to the action of $\mathcal{G}_{0,\bar{\mu}+1}\subset \mathcal{G}_{0,\nu+1}$ the element $\Psi(A)$ is independent of the particular choices of $-1<\mu^\prime<\mu$, $k\geq 1$, $\alpha\in(0,1)$, $B$, and $g$. We start with the independence of the choice of $g$. For this, assume that \[A_1= \Psi_1(A) \equiv g_1^*A \in \mathcal{A}_{\bar{\mu}}^{\Fr}(P,\{\Upsilon_i,\tilde{\Upsilon}_i,A_i\}) \quad \textup{and} \quad A_2=\Psi_2(A)\equiv g_2^*A \in \mathcal{A}_{\bar{\mu}}^{\Fr}(P,\{\Upsilon_i,\tilde{\Upsilon}_i,A_i\}) \] are both constructed as above associated to different $g_1,g_2 \in \mathcal{G}^{k+1,\alpha}_{0,\mu^\prime+1}$. Then there exists a further $g \in \mathcal{G}^{k+1,\alpha}_{0,\mu^\prime+1}$ such that $g^*A_1 = A_2$. Once more assuming that $G\subset \GL(W)$ for some vector space $W$, we may regard $g$ and $A_1-A_2$ as a section of (and a 1-form with values in) the vector bundle associated to $\End(W)$. The equation $g^*A_1=A_2$ is then equivalent to \[ \diff_{A_1}g = g(A_2-A_1).\] Since $A_2-A_1 \in \Omega^1_{\bar{\mu}}(Z\setminus S,\mathfrak{g}_P)$, we obtain $g \in \mathcal{G}_{0,\bar{\mu}+1}$. A similar argument also shows the independence (up to the action of $\mathcal{G}_{0,\bar{\mu}+1}\subset \mathcal{G}_{0,\nu+1}$) of $\mu^\prime$, $k$, $\alpha$, and $B$.

Now assume that two $\SU(3)$-instantons $A_1,A_2 \in \mathcal{A}_{\mu}^\Fr(P,\{\Upsilon_i\tilde{\Upsilon}_i,A_i\})$ differ by $g \in \mathcal{G}_{0,\mu+1}$. Let $\Psi(A_i)\coloneqq g_i^*A_i \in \mathcal{A}_{\bar{\mu}}$ be as constructed above. Then \[ (g_2^{-1}gg_1)^*(g_2^*A_2)=g_1^*A_1 \] and the same argument as in the previous paragraph proves $g_2^{-1}gg_1 \in \mathcal{G}_{0,\bar{\mu}+1}$. This implies that $\Psi$ descends to a well-defined map (which we again denote by) \[ \Psi \col \mathcal{M}_\mu^\Fr(P,\{\Upsilon_i,\tilde{\Upsilon}_i,A_i\}) \to \mathcal{M}_\nu^\Fr(P,\{\Upsilon_i,\tilde{\Upsilon}_i,A_i\}). \] A scale-broken elliptic estimate (cf. \autoref{app-prop: scalebroken Schauder estimate}) and \autoref{prop: local structure of moduli space for fixed bundle} (more precisely, \autoref{rem: extension of local stucture of moduli space for fixed bundle to non instanton slice}) imply that this map is continuous\footnote{Note that this argument finally uses that $\nu\in \mathbb{R}^N\setminus \mathcal{D}(L_A)$ is a non-critical rate.} and a moment's thought shows that the inclusion map $\mathcal{M}_\nu^\Fr(P,\{\Upsilon_i,\tilde{\Upsilon}_i,A_i\}) \subset \mathcal{M}_\mu^\Fr(P,\{\Upsilon_i,\tilde{\Upsilon}_i,A_i\})$ is its (continuous) inverse.
\end{proof}

The following theorem can be proven in a similar fashion.

\begin{theorem}\label{thm: relating different rates}
Assume the situation described in the beginning of this section. For any $ \mu,\nu \in (-1,\Bar{\mu}_1)\times \dots \times(-1,\Bar{\mu}_N) $ where $\Bar{\mu}_i$ are as in \autoref{prop: rates of invertibility of Laplacian} there exists a homeomorphism $\Psi \col \mathcal{M}_\mu(\{P_i,A_i\}) \to \mathcal{M}_\nu(\{P_i,A_i\})$.
\end{theorem}

\section{The obstruction space}\label{sec: obstruction space}

Let $(Z,\omega,\Omega)$ be a 6-manifold equipped with an $\SU(3)$-structure satisfying $\diff^*\omega = 0$ and $\diff \Omega = w_1 \omega^2$ for $w_1 \in \mathbb{R}$. Moreover, let $G$ be a compact Lie group with center and let $(\pi_i \col P_i \to S^5,A_i)$ for $i=1,\dots,N\in \mathbb{N}$ be a collection of principal $G$-bundles with infinitesimally irreducible connections $A_i\in \mathcal{A}(P_i)$ satisfying~\eqref{equ: cone reduction of instanton equation} (but recall that by \autoref{rem: Coulomb gauge and slice for groups with positive center} and \autoref{rem: Kuranishi chart for groups with positive center} these assumptions can be removed). In \autoref{thm: Kuranishi structure} we have seen that $\mathcal{M}_\mu(\{P_i,A_i\})$ is for certain rates locally modelled on a quotient of the zero set of a smooth map 
\begin{align*}
\ob_{\mathbb{A}} \col \ker \mathfrak{L}^{k,\alpha}_{\mathbb{A}} \to \coker \mathfrak{L}^{k,\alpha}_{\mathbb{A}}.
\end{align*}
Thus, whenever $\coker \mathfrak{L}^{k,\alpha}_{\mathbb{A}} = 0$ (or, more generally, whenever $0$ is a regular value of $\ob_{\mathbb{A}}$) $\mathcal{M}_\mu(\{P_i,A_i\})$ is locally an orbifold of dimension $\textup{index}(\mathfrak{L}^{k,\alpha}_{\mathbb{A}})$. In the following we will give in \autoref{cor: estimate dimension cokernel} under certain assumptions on $(\pi_i\col P_i\to S^5,A_i)$ and $\mathbb{A}\in\mathcal{A}_\mu(\{P_i,A_i\})$ a formula for $\dim\coker \mathfrak{L}^{k,\alpha}_{\mathbb{A}}$.

\subsection{A pairing for the cokernel}
\label{subsec: pairing for the cokernel}
Let $(Z,\omega,\Omega)$ be a 6-manifold equipped with an $\SU(3)$-structure satisfying $\diff^*\omega = 0$ and $\diff \Omega = w_1 \omega^2$ for $w_1 \in \mathbb{R}$. Assume further that $\pi \col P \to Z\setminus S$ for $S=\{s_1,\dots,s_N\} \subset Z$ is a principal $G$-bundle, where $G$ is compact with discrete center. Moreover, let $A \in \mathcal{A}_\mu(P,\{\Upsilon_i,\tilde{\Upsilon}_i,A_i\})$ be a conically singular $\SU(3)$-instanton with infinitesimally irreducible tangent cones $(\pi_i \col P_i \to S^5,A_i)$ at $s_i \in S$, framings $\tilde{\Upsilon}_i \col \pr^*_{S^5}P_i \to P$ and rates $\mu_i \in (-1,\bar{\mu}_i)$ with $\bar{\mu}_i$ as in \autoref{prop: rates of invertibility of Laplacian}. Throughout this section we will restrict to tangent connections $A_i \in \mathcal{A}(P_i)$ that additionally satisfy the following:
\begin{assumption}\label{ass: critical rates contained in Z}
Assume that for every $i= 1,\dots,N$ we have \[\mathcal{D}(L_{A_i})\cap [-4,-1] \subset \{-4,-3,-2,-1\},\]
where $\mathcal{D}(L_{A_i})$ is as in \autoref{def: critical rates}.
\end{assumption}
\begin{remark}
By \cite[Theorem~1.8]{Wang-spectrum_of_operator_for_instantons} the previous assumptions holds for $A_i \in \mathcal{A}(P_i)$ if both the bundle and the connection $(\pi_i \col P_i \to S^5,A_i)$ are pulled back from $\mathbb{P}^2$. By \autoref{prop: tangent cones pulled back from P2} this holds automatically if $A_i$ is irreducible and $G$ has trivial center.
\end{remark}

Recall from \autoref{def: non-linear Fredholm map and its linearisation} the linear operator
\begin{align*}
\mathfrak{L}_{\mathbb{A}}^{k,\alpha} \col (\mathbb{C}^3)^N  \oplus ( \oplus_i \mathfrak{m}_i) &\oplus  C^{k,\alpha}_\mu(\mathfrak{g}_P \oplus \mathfrak{g}_P \oplus T^*Z \otimes\mathfrak{g}_P) \to C_{\mu-1}^{k-1,\alpha}(\mathfrak{g}_P \oplus \mathfrak{g}_P \oplus T^*Z \otimes\mathfrak{g}_P) \\
(\vec{v},\vec{u}, \xi_1,\xi_2, a) & \mapsto \big(\diff_A^* a, \Lambda_{\omega} \diff_{A}a + *(\diff (i_{X(\vec{v},\vec{u})} \omega \wedge \omega) \wedge F_A), \\
& \qquad *(\diff_A a \wedge \Im \Omega + F_A \wedge \diff i_{X(\vec{v},\vec{u})}\Im \Omega) + \diff_{A} \xi_1 - J^*(\diff_{A} \xi_2)\big)
\end{align*}
where the vector field $X(\vec{v},\vec{u}) \in \Gamma(TZ)$ is defined by 
\begin{align*}
X(\vec{v},\vec{u}) &\coloneqq (\partial_{\vec{v}}\mathfrak{vec}_2)(0) + (\partial_{\vec{v}}\mathfrak{vec}_1)(0) + (\partial_{\vec{u}}\mathfrak{vec}_0)(0) \\
&= (\partial_{\vec{v}}\mathfrak{vec}_2)(0) + \mathfrak{vec}_1(\vec{v}) + \mathfrak{vec}_0(\vec{u})
\end{align*}
for $\mathfrak{vec}_{0/1/2}$ as in \autoref{prop: construction of families of vector fields} (and where we regard the respective derivatives of $\mathfrak{vec}_{0,1,2}$ at $0$ again as a vector field on $Z$). 

Furthermore, recall from the discussion prior to \autoref{prop: adjoint of model operator} the operator $L_A$ associated to the connection $A$. Since \[\mu \notin \mathcal{D}(L) \coloneqq\big\{(\lambda_1,\dots,\lambda_N) \in \mathbb{R}^N \ \big\vert \ \lambda_i \in \mathcal{D}(L_{s_i}) \textup{ for some $i \in \{1,\dots,N\}$} \big\},\] we have that \[L_A \col  C^{k,\alpha}_\mu( Z \setminus S,\mathfrak{g}_P \oplus \mathfrak{g}_P \oplus T^*Z \otimes\mathfrak{g}_P) \to C_{\mu-1}^{k-1,\alpha}(Z \setminus S,\mathfrak{g}_P \oplus \mathfrak{g}_P \oplus T^*Z \otimes\mathfrak{g}_P)\] is Fredholm (cf. \autoref{app-prop: Fredholm away from critical rates}).

\begin{definition}
We define the following map:
\begin{align*}
\Phi_\mathbb{A} \col (\mathbb{C}^3)^N \oplus (\oplus_i \mathfrak{m}_i) &\to \coker(L_A \col C^{k,\alpha}_\mu \to C^{k-1,\alpha}_{\mu-1}) \\
(\vec{v},\vec{u}) & \mapsto (\pr_{\coker(L_A)} \circ \mathfrak{L}^{k,\alpha}_{\mathbb{A}})(\vec{v},\vec{u},0,0,0)
\end{align*}
\end{definition}

\begin{proposition}\label{prop: observations on cokernel}
The following hold:
\begin{enumerate}
\item $\coker(\Phi_\mathbb{A}) \cong \coker(\mathfrak{L}^{k,\alpha}_{\mathbb{A}})$.
\item The $L^2$-inner product gives rise to a perfect pairing between $\coker(L_A \col C^{k,\alpha}_\mu \to C^{k-1,\alpha}_{\mu-1})$ and $\ker(L_A)_{-5-\mu}$, where \[\ker(L_A)_{-5-\mu} \coloneqq \ker(L_A \col C^{k,\alpha}_{-5-\mu} \to C^{k-1,\alpha}_{-6-\mu}).\]
\item If $\underline{a}\coloneqq (\xi_1,\xi_2,a) \in \ker(L_A)_{-5-\mu}$, then $\xi_1,\xi_2=0$.
\item Let $\underline{a}=(0,0,a)\in \ker(L_A)_{-5-\mu}$. For every $i = 1,\dots, N$ there are $\tilde{a}_i,\tilde{b}_i \in \Omega^1(S^5,\mathfrak{g}_{P_i})$ and $\tilde{\xi}_i,\tilde{\zeta}_i \in \Omega^0(S^5,\mathfrak{g}_{P_i})$  with \[ r^{-3}(\tilde{a}_i+\tilde{\xi}_i \tfrac{\diff r}{r}) \in \mathcal{K}(L_{A_i})_{-4} \quad \textup{and} \quad r^{-2}(\tilde{b}_i+\tilde{\zeta}_i \tfrac{\diff r}{r}) \in \mathcal{K}(L_{A_i})_{-3}\] and a linear function \[ \eta_i \col \mathcal{K}(L_{A_i})_{-4} \to \Omega^1_{-3+\mu_i}(\mathbb{C}^3\setminus \{0\},\pr^*_{S^5}\mathfrak{g}_{P_i})\]
such that \[ \big\vert \tilde{\Upsilon}^*_i a - r^{-3}(\tilde{a}_i+\tilde{\xi}_i \tfrac{\diff r}{r}) - \eta_i\big(r^{-3}(\tilde{a}_i+\tilde{\xi}_i \tfrac{\diff r}{r})\big) - r^{-2} (\tilde{b}_i+\tilde{\zeta}_i \tfrac{\diff r}{r}) \big\vert = \mathcal{O}(r^{-2+ \varepsilon})\] for some small $\varepsilon>0$. Note that in the formulation above we implicitly identified both $\mathbb{C}^3\setminus \{0\} \cong S^5 \times \mathbb{R}_{>0}$ and $\pr_{S^5} \mathfrak{g}_{P_i} \cong \mathfrak{g}_{P_i} \times \mathbb{R}_{>0}$.
\end{enumerate}
\end{proposition}

\begin{proof}
The first point simply follows from 
\begin{align*}
\coker (\mathfrak{L}^{k,\alpha}_{\mathbb{A}}) &\cong C^{k-1,\alpha}_{\mu-1}(T^*Z\otimes \mathfrak{g}_{P} \oplus \mathfrak{g}_{P} \oplus \mathfrak{g}_{P})/\textup{image}(\mathfrak{L}^{k,\alpha}_\mu) \\
&\cong \frac{C^{k-1,\alpha}_{\mu-1}(T^*Z\otimes \mathfrak{g}_{P} \oplus \mathfrak{g}_{P} \oplus \mathfrak{g}_{P})/\textup{image}(L_A)}{\textup{image}(\mathfrak{L}^{k,\alpha}_\mathbb{A})/\textup{image}(L_A)} \\
&\cong \coker(L_A) / \textup{image}(\Phi_{\mathbb{A}}) \\
&\cong \coker(\Phi_{\mathbb{A}}),
\end{align*}
where we have abbreviated $\textup{image}(L_A)$ and $\coker(L_A)$ for the respective image and cokernel of the operator $L_A \col C^{k,\alpha}_\mu \to C^{k-1,\alpha}_{\mu-1}$. The second point is \autoref{app-prop: cokernel kernel pairing}. For the third point we may argue as in \autoref{prop: augmented instanton equation} that $\Delta_A \xi_1=0$. Since $\mu_i < \bar{\mu}_i$ for every $i=1,\dots,N$, \autoref{prop: critical rates of model Laplacian} (and the fact that the set of critical rates is symmetric under reflection at -2) implies that there are no critical rates of $\Delta_{A_i}$ between $-5-\mu_i$ and $0$. Thus, the decay of $\xi_1$ can be improved to conclude $\diff_A \xi_1=0$ (cf. \autoref{app-cor: kernel is locally constant}). However, since $A_i$ is infinitesimally irreducible we have $\xi_1 =0$ and similarly $\xi_2=0$. The fourth point is \autoref{app-prop: cs operators asymptotic expansions of kernel elements} together with \autoref{ass: critical rates contained in Z} that the only critical rates of $L_{A_i}$ in $(-5-\bar{\mu}_i,-2]$ are $-4$, $-3$, and $-2$.
\end{proof}

\begin{remark}
The sections $(\tilde{a}_i,\tilde{\xi}_i)$ and $(\tilde{b}_i,\tilde{\zeta}_i)$ over $S^5$ in the last point of the previous proposition are eigensections to a suitable self-adjoint elliptic operator $P$ on $S^5$ (arising from $L_{A_i}$) written down explicitly in~\cite[Lemma~2.14]{Wang-spectrum_of_operator_for_instantons} (see also the proof of \autoref{prop: duality between homogeneous kernels} for a formula of $P$).
\end{remark}

\begin{proposition}\label{prop: obstruction pairing}
Let $\underline{a} = (0,0,a) \in \ker(L_A)_{-5-\mu}$, where $\ker(L_A)_{-5-\mu}$ is as in the previous proposition. Moreover, let $\tilde{a}_i,\tilde{b}_i \in \Omega^1(S^5,\mathfrak{g}_{P_i})$ and $\tilde{\xi}_i,\tilde{\zeta}_i \in \Omega^0(S^5,\mathfrak{g}_{P_i})$ for every $i=1,\dots,N$ be associated to $\underline{a}$ as in the last point of the previous proposition. For any $u \in \mathfrak{su}(3)$ we denote by $\hat{u}_{S^5} \in \Gamma(TS^5)$ the vector field defined at $z\in S^5 \subset \mathbb{C}^3$ by $u\cdot z \in z^\perp = T_zS^5$. Then for every $(\vec{v},\vec{u}) \in (\mathbb{C}^3)^N \oplus (\oplus_i \mathfrak{m}_i)$ we have
\begin{align*}
\begin{split}
\int_Z \big\langle \Phi_{\mathbb{A}}(\vec{v},\vec{u}),\underline{a} \big\rangle \vol_Z &= \sum_{i=1}^N \int_{S^5} \big\langle i_{(\hat{u}_i)_{S^5}} F_{A_i}, J_2^*(\tilde{a}_i)_H \big\rangle \vol_{S^5} + \int_{S^5}  \big\langle i_{v_i} F_{A_i}, J_2^*(\tilde{b}_i)_H \big\rangle \vol_{S^5} \\
& \qquad + \mathfrak{e}_i\big(v_i,(\tilde{a}_i,\tilde{\xi}_i)\big).
\end{split}
\end{align*}
Here, $\alpha_H$ for $\alpha \in \Omega^1(S^5)$ denotes the projection to the (dual of the) contact distribution $H = \ker \theta$ and $J_2$ is the complex structure on $H$ corresponding to $\omega_2$ for the canonical Sasaki--Einstein structure $(\theta,\omega_1,\omega_2,\omega_3)$ described in \autoref{subsec: Sasaki--Einstein structure on S5}. Moreover, the functions $\mathfrak{e}_i$ in the expression above are linear in $v_i$ and $(\tilde{a}_i,\tilde{\xi}_i)$.
\end{proposition}

\begin{remark}\label{rem: interpretation of homogeneous kernel elements as eigensection}
It follows from \autoref{prop: critical rates of model Laplacian} and \autoref{prop: duality between homogeneous kernels} (which will be proven in the next section) that $\tilde{\xi}_i,\tilde{\zeta}_i \in \Omega^0(S^5,\mathfrak{g}_{P_i})$ in the last point of \autoref{prop: observations on cokernel} vanish identically. Thus, the functions $\mathfrak{e}_i$ in the previous proposition only depend on $v_i$ and $\tilde{a}_i$ (and are linear with respect to both arguments).
\end{remark}

\begin{proof}
For notational convenience we only prove this theorem for $N=1$ and drop the subscripts. To any vector field $X \in \Gamma(TZ)$ we will in the following associate a vector $X_0 \in \mathbb{C}^3$ and a vector field $X_{S^5} \in \Gamma(TS^5)$ as follows: The vector is defined as $X_0 \coloneqq (D_0 \Upsilon)^{-1} X_{s} \in \mathbb{C}^3$. In order to construct $X_{S^5} \in \Gamma(TS^5)$ we first consider the pullback vector field $\Upsilon^*X$ over $B_R(0)\subset \mathbb{C}^3$ as a map $B_R(0) \to \mathbb{C}^3$. Let $D_0 (\Upsilon^*X) \col \mathbb{C}^3 \to \mathbb{C}^3$ be its derivative at zero. At $z \in S^5 \subset \mathbb{C}^3$ we then set $ X_{S^5}(z) \coloneqq (D_0 (\Upsilon^*X) \cdot z)^{\perp z} \in z^\perp = T_zS^5$, where $(\cdot)^{\perp z}$ denotes the projection onto the orthogonal complement of $z$. 

For any $X\in \Gamma(TZ)$ we now calculate
\begin{align*}
\int_Z \big\langle *(F_A \wedge \diff i_X \Im \Omega), a \big\rangle \vol_Z &= - \lim_{r \to 0} \int_{Z \setminus \Upsilon(B_r)} \big\langle a \wedge \diff_A (F_A \wedge i_X \Im \Omega) \big\rangle_{\mathfrak{g}_P} \\
&= \lim_{r \to 0} \int_{Z \setminus \Upsilon(B_r)} \diff \big\langle a \wedge F_A \wedge i_X \Im \Omega \big\rangle_{\mathfrak{g}_P} \\
& \qquad \qquad \qquad \qquad - \big\langle (\diff_A a) \wedge F_A \wedge i_X\Im \Omega\big\rangle_{\mathfrak{g}_P} \\
&=  \lim_{r\to 0} \int_{S^5_r} \big\langle (\tilde{\Upsilon}^* a)_{\vert S^5_r}\wedge (\tilde{\Upsilon}^* (F_A \wedge i_{X} \Im \Omega )_{\vert S^5_r}) \big\rangle_{\mathfrak{g}_{P_0}} \\
&=  \lim_{r\to 0} \int_{S^5} \big\langle \tilde{\delta}_r^*(\tilde{\Upsilon}^* a)_{\vert S^5_r}\wedge (\tilde{\delta}_r^*\tilde{\Upsilon}^* (F_A \wedge i_{X} \Im \Omega )_{\vert S^5_r}) \big\rangle_{\mathfrak{g}_{P_0}} \\
&=  \lim_{r\to 0} \int_{S^5} \big\langle \tilde{\delta}_r^*(\tilde{\Upsilon}^* a)_{\vert S^5_r}, *_{S^5} (\tilde{\delta}_r^*\tilde{\Upsilon}^* (F_A \wedge i_{X} \Im \Omega )_{\vert S^5_r}) \big\rangle_{S^5} \vol_{S^5} 
\end{align*}
where the last term in the second line vanishes because $F_A$ and $\diff_A a$ are both of type $(1,1)$ (following from the $\SU(3)$-instanton condition) and $i_X\Im \Omega$ is of type $(2,0)+(0,2)$ and where $\delta_r \col S^5 \to S^5_r$ is the dilation-diffeomorphism and $\tilde{\delta}_r \col \pr_{S^5}^*P_0 \to \pr_{S^5}^*P_0$ its canonical lift. Since $A$ is an $\SU(3)$-instanton we have \[F_A \wedge i_X \Im \Omega = i_X (F_A \wedge \Im \Omega) - (i_X F_A) \wedge \Im \Omega = - (i_X F_A) \wedge \Im \Omega.\] Moreover, from \[ \vert \nabla^k (\tilde{\Upsilon}^* A -A_0) \vert = \mathcal{O}(r^{\mu-k}), \quad \vert \Upsilon^*\Im \Omega - \Im \Omega_0 \vert = \mathcal{O}(r), \quad \Upsilon^*X = X_0 + r \cdot (X_{S^5} + f_{S^5} \cdot \partial_r) + \mathcal{O}(r^2) \] (for a suitable function $f_{S^5}$ on $S^5$) we obtain 
\begin{align*}
\tilde{\delta}_r^* \tilde{\Upsilon}^*(F_A \wedge i_X \Im \Omega)_{\vert S^5_r} &= - r^2 (i_{X_0}F_{A_0}) \wedge (\Im \Omega_0)_{\vert S^5} - r^3 (i_{X_{S^5}}F_{A_0}) \wedge (\Im \Omega_0)_{\vert S^5} \\
& \quad - r^2 (i_{X_0}F_{A_0})\wedge \big(\delta^*_r(\tfrac{\Im \Omega}{r^3})-\Im \Omega_0 \big) - r^2 i_{X_0} (\tilde{\delta}_r^*F_A -F_{A_0}) \wedge (\Im \Omega_0) \\
& \quad + \mathcal{O}(r^{4+\mu}).
\end{align*}

Next, we express $\tilde{\Upsilon}_s^*a$ as in the previous proposition as \[ \tilde{\Upsilon}^*_s a = r^{-3}(\tilde{a}_s+\tilde{\xi}_s \tfrac{\diff r}{r}) + \eta_s(\tilde{a}_s,\tilde{\xi}_s) + r^{-2} (\tilde{b}_s+\tilde{\zeta}_s \tfrac{\diff r}{r}) + \mathcal{O}(r^{-2+ \varepsilon}).\]  Noting that the terms in the second line in the expression of $\tilde{\delta}_r^* \tilde{\Upsilon}^*(F_A \wedge i_X \Im \Omega)_{\vert S^5_r}$ are $\mathcal{O}(r^3 + r^{3+\mu})$ and that $(\tilde{\delta}_r^*\eta_s(\tilde{a}_s,\tilde{\xi}_s))_{\vert S^5} = \mathcal{O}(r^{-2+\mu_s})$ gives
\begin{align}\label{equ: auxilliary calculation for obstruction pairing}
\begin{split}
\int_Z \big\langle *(F_A \wedge \diff i_X \Im \Omega), a \big\rangle_{S^5} \vol_Z = - \lim_{r\to 0} &\int_{S^5} \big\langle \tilde{b}_s,*(i_{X_0}F_{A_0}) \wedge \Im \Omega_0  \big\rangle_{S^5} \\
& \quad + \big\langle \tilde{a}_s,*(i_{X_{S^5}}F_{A_0}) \wedge \Im \Omega_0 \big\rangle_{S^5} \\
& \quad + r^{-1}\big\langle \tilde{a}_s,*(i_{X_0}F_{A_0}) \wedge \Im \Omega_0  \big\rangle_{S^5} \\
& \quad + r^{-1} \big\langle \tilde{a}_s, *(i_{X_0}F_{A_0})\wedge \big(\delta^*_r(\tfrac{\Im \Omega}{r^3})-\Im \Omega_0 \big)\big\rangle_{S^5} \\
& \quad + r^{-1} \big\langle \tilde{a}_s, *(i_{X_0} (\tilde{\delta}_r^*F_A -F_{A_0})) \wedge \Im \Omega_0 \big\rangle_{S^5} \\
& \quad  + r^2 \big\langle \tilde{\delta}_r^*\eta_s,*_{S^5}(i_{X_0}F_{A_s} \wedge\Im \Omega_0) \big\rangle \\
& \quad  + r^2 \big\langle \tilde{\delta}_r^*\eta_s,*_{S^5}(i_{X_0}(\tilde{\delta}_r^*F_A-F_{A_s}) \wedge\Im \Omega_0) \big\rangle \vol_{S^5}.
\end{split}
\end{align}
Since the integral on the left-hand side is finite and the last four terms are $\mathcal{O}(r^0+r^{\mu}+r^{\mu}+r^{1+2\mu})$, we must have $\int_{S^5} \langle \tilde{a}_s,(i_{X_0}F_{A_0}) \wedge \Im \Omega_0 \rangle \vol_{S^5} = 0.$\footnote{Alternatively, this also follows from \autoref{prop: infinitesimal translations and Atiyah classes}, \autoref{prop: duality between homogeneous kernels}, and the interpretation of $(\tilde{a}_s,\tilde{\xi}_s)$ as eigensections of an operator $P$ on $S^5$ (cf. \autoref{rem: interpretation of homogeneous kernel elements as eigensection}). Indeed, since $P$ is formally self-adjoint (cf. \cite[Lemma~2.14]{Wang-spectrum_of_operator_for_instantons} or the proof of \autoref{prop: duality between homogeneous kernels}), the eigensections to different eigenvalues are $L^2$-orthogonal.} 

Recall the canonical $\SU(2)$-structure $(\theta,\omega_1, \omega_2, \omega_3)$ on $S^5$ as reviewed in \autoref{subsec: Sasaki--Einstein structure on S5}. We have $(\Im \Omega_0)_{\vert S^5} = \theta \wedge \omega_2$ and therefore $ *_{S^5} (\alpha \wedge \Im \Omega_0)_{\vert S^5} = J_2^* \alpha_H$ for any $\alpha \in \Omega^1(S^5)$ where $\alpha_H$ denotes the projection of $\alpha$ to $H^*$ (the dual of the contact distribution $H = \ker \theta$) and where $J_2$ denotes the almost complex structure on $H$ corresponding to $\omega_2$ (cf. \cite[Proposition~1.2.31]{Huybrechts-complex-Geometry}).

Ultimately, we note that for $(v,u) \in \mathbb{C}^3 \oplus \mathfrak{m}_s$ the vector field $X(v,u)\in \Gamma(TZ)$ (as defined at the beginning of this section) satisfies \[ (X(v,u))_0 = v \qquad \textup{and} \qquad (X(v,u))_{S^5} = \hat{u}_{S^5} + \mathfrak{r}(v) \] where $\hat{u}_{S^5} \in \Gamma(TS^5)$ is defined at $z \in S^5$ by $u \cdot z \in z^\perp = T_zS^5$ and $\mathfrak{r}(v) \in \Gamma(TS^5)$ depends linearly on $v$. (This follows from the second items in \autoref{bul: properties of vec0}, \autoref{bul: properties of vec1}, and \autoref{bul: properties of vec2} of \autoref{prop: construction of families of vector fields}, respectively.) Inserting this into \eqref{equ: auxilliary calculation for obstruction pairing} and combining $\int \langle \tilde{a}_s,*(i_{\mathfrak{r}(v)}F_{A_0}) \wedge \Im \Omega_0 \rangle_{S^5} \vol_{S^5}$ and the last four terms into $\mathfrak{e}_s(v,(\tilde{a}_s,\tilde{\xi}_s))$ finishes the proof.
\end{proof}

\subsection{On the non-degenerateness of the pairing}

In the following we will show that under a suitable condition on $A$ the pairing between $(\mathbb{C}^3)^N \oplus (\oplus_i\mathfrak{m}_i)$ and $\ker(L_A)_{-5-\mu}$ introduced in \autoref{prop: obstruction pairing} is left-non-degenerate. For this we will first show that the projections of $i_{v_i}F_{A_i}$ and $i_{(\hat{u}_i)_{S^5}} F_{A_i}$ onto $\mathcal{K}(L_{A_i})_{-2}$ and $\mathcal{K}(L_{A_i})_{-1}$ are non-trivial, respectively. In \autoref{sec: proof of non-degenerateness} we will then use this together with a duality-map between homogeneous kernels to prove the non-degenerateness of the pairing.

\subsubsection{Infinitesimal rotations and deformations of the tangent connection}

Let $\pi_0 \col P_0 \to S^5$ be a principal $G$-bundle. Recall from \autoref{prop: dilation invariant instantons} that $\pr^*_{S^5}A_0$ for $A_0 \in \mathcal{A}(P_0)$ is an $\SU(3)$-instanton over $\mathbb{C}^3\setminus \{0\}$ if and only if $F_{A_0} \wedge \omega_i = 0$ for every $i=1,2,3$, where $(\theta,\omega_1,\omega_2,\omega_3)$ denotes the canonical $\SU(2)$-structure on $S^5$ (cf. \autoref{subsec: Sasaki--Einstein structure on S5}). The following is the infinitesimal version of \autoref{prop: dilation invariant instantons} and its proof follows from the same arguments as in \autoref{prop: augmented instanton equation} and \autoref{prop: dilation invariant instantons}.

\begin{proposition}\label{prop: homogeneous kernel of degree -1}
Let $\pr_{S^5}^*A_0$ for $A_0 \in \mathcal{A}(P_0)$ be a dilation-invariant and infinitesimally irreducible $\SU(3)$-instanton. Then $L_{\pr_{S^5}^*A_0} (\pr_{S^5}^*(\xi_1,\xi_2,a)) = 0$ for some $(\xi_1,\xi_2,a) \in \Omega^0(S^5,\mathfrak{g}_{P_0}\oplus \mathfrak{g}_{P_0} \oplus TS^5 \otimes \mathfrak{g}_{P_0})$ if and only if 
\begin{align}
\xi_1,\xi_2 &= 0\\
(\diff_{A_0} a) \wedge \omega_i &= 0 \quad \text{for every $i=1,2,3$} \label{equ: SU(2) instantons deformations} \\
\diff_{A_0}^*a &= 0. \label{equ: SU(2) instantons Coulomb gauge}
\end{align}
Consequently, $ \mathcal{K}(L_{A_0})_{-1} = \big\{ \pr_{S^5}^*a \mid a \in \Omega^1(S^5,\mathfrak{g}_{P_0}) \text{ satisfies \eqref{equ: SU(2) instantons deformations} and \eqref{equ: SU(2) instantons Coulomb gauge}} \big\}$.
\end{proposition}

\begin{proposition}
For $u \in \mathfrak{su}(3)$ let $\hat{u}_{S^5} \in \Gamma(TS^5)$ be the infinitesimal vector field induced by $u$ (i.e. $\hat{u}_{S^5}(z) = u \cdot z \in T_zS^5$ for every $z\in S^5$). Moreover, let $P_0$ and $A_0$ be as in the previous proposition. Then $i_{\hat{u}_{S^5}}F_{A_0} \in \Omega^1(S^5,\mathfrak{g}_{P_0})$ satisfies~\eqref{equ: SU(2) instantons deformations}.
\end{proposition}

\begin{remark}
Let $\textup{Flow}_t^{\hat{u}_{S^5}} \col S^5 \to S^5$ be the time $t$-flow of the vector field $\hat{u}_{S^5}$ and let $\textup{tra}_t^{A_0} \col P_0 \to P_0$ be its lift via parallel transport. Then $i_{\hat{u}_{S^5}} F_{A_0} = \tfrac{\diff}{\diff t} ((\textup{tra}_t^{A_0})^*A_0)_{\vert t=0}$. Thus, the previous proposition states that (to first order) $(\textup{tra}_t^{A_0})^*A_0$ is a family of instantons on $P_0$. 
\end{remark}

\begin{proof}
Denote by $\textup{Flow}_t^{\hat{u}_{S^5}} \col S^5 \to S^5$ the time $t$-flow of the vector field $\hat{u}_{S^5}$ and by $\textup{tra}_t^{A_0} \col P_0 \to P_0$ its lift via parallel transport. Since $\text{Flow}_t^{\hat{u}_{S^5}} = \exp(tu) \in \SU(3)$ we have 
\begin{align*}
\diff_{A_0} (i_{\hat{u}_{S^5}}F_{A_0}) \wedge \omega_i &= \tfrac{\diff}{\diff t} \big(F_{(\text{tra}_t^{A_0})^*A_0} \wedge \omega_i\big)_{\vert t=0} \\
&=\tfrac{\diff}{\diff t} \big((\text{tra}_t^{A_0})^*(F_{A_0} \wedge \omega_i)\big)_{\vert t=0} = 0
\end{align*}
since $\pr_{S^5}^*A_0$ is a dilation-invariant instanton over $\mathbb{C}^3\setminus \{0\}$.
\end{proof}

Recall the group $\Stab_{\SU(3)}(A_0)$ introduced in~\eqref{equ: definition Stab_SU(3)(A_s)} and let $\mathfrak{m}\subset \mathfrak{su}(3)$ be a linear subspace which is a complement to $\textup{image}(\mathfrak{stab}_{\SU(3)}(A_0) \to \mathfrak{su}(3))$, the image of $\mathfrak{stab}_{\SU(3)}(A_0)$ under the canonical projection to $\mathfrak{su}(3)$.

\begin{proposition}\label{prop: infinitesimal rotations and deformations of tangent cone}
If $u \in \mathfrak{m}$ is non-zero, then there exists an $a \in \Omega^1(S^5,\mathfrak{g}_{P_0})$, which satisfies \eqref{equ: SU(2) instantons deformations} and \eqref{equ: SU(2) instantons Coulomb gauge}, such that \[ \int_{S^5} \big\langle i_{\hat{u}_{S^5}}F_{A_0}, a \big\rangle \vol_{S^5} \neq 0. \]
\end{proposition}

\begin{proof}
Elliptic theory for the operator $\diff_{A_0} \col \Omega^0(S^5,\mathfrak{g}_{P_0}) \to \Omega^1(S^5,\mathfrak{g}_{P_0})$ (together with the observation that $\image(\diff_{A_0}) \subset \ker(\diff_{A_0}\cdot \wedge \omega_i)$ because $A_0$ is an instanton; cf. \autoref{prop: dilation invariant instantons}) implies that \[ \cap_i \ker(\diff_{A_0}\cdot \wedge \omega_i) = \big\{ a \in \Omega^1(S^5,\mathfrak{g}_{P_0}) \mid \text{$a$ satisfies \eqref{equ: SU(2) instantons deformations} and \eqref{equ: SU(2) instantons Coulomb gauge}}\big\} \oplus \textup{image}(\diff_{A_0}),\] where the splitting is $L^2$-orthogonal. Thus, by the previous proposition, if $i_{\hat{u}_{S^5}}F_{A_0}$ lies $L^2$-orthogonal to every $a \in \Omega^1(S^5,\mathfrak{g}_{P_0})$ satisfying \eqref{equ: SU(2) instantons deformations} and \eqref{equ: SU(2) instantons Coulomb gauge}, then $i_{\hat{u}_{S^5}}F_{A_0} = \diff_{A_0} \xi$ for some $\xi \in \Omega^0(S^5,\mathfrak{g}_{P_0})$.

We now define a 1-parameter family of connections \[ A_t \coloneqq (\textup{tra}_t^{A_0})^*A_0 \in \mathcal{A}(P_0) \] where $\text{tra}_t^{A_0} \col P_0 \to P_0$ is the lift of $\textup{Flow}_t^{\hat{u}_{S^5}}$ via parallel transport. We then have \[ \tfrac{\diff}{\diff t} (A_t)_{\vert t=t_0} = (\text{tra}_{t_0}^{A_0})^* (i_{\hat{u}_{S^5}} F_{A_0}) = (\text{tra}_{t_0}^{A_0})^* (\diff_{A_0} \xi) = \diff_{A_{t_0}}\big( (\text{tra}_{t_0}^{A_0})^* \xi \big).  \] In particular, $A_t$ is at every $t$ tangent to the gauge orbit. This implies that for every $t$, there exists a gauge transformation mapping $A_t$ to $A_0$. This however is in contradiction to the fact that $\mathfrak{m}$ is (by definition) transverse to $\textup{image}(\mathfrak{stab}_{\SU(3)}(A_0) \to \mathfrak{su}(3))$.
\end{proof}

\subsubsection{Infinitesimal translations}

As in the previous section, let $\pi \col P_0 \to S^5$ be a principal $G$-bundle together with a connection $A_0$ that pulls back to a dilation invariant $\SU(3)$-instanton over $\mathbb{C}^3\setminus \{0\}$.

\begin{proposition}[{\cite[Proposition~4.1]{Wang-AtiyahClasses}}]\label{prop: infinitesimal translations and Atiyah classes}
Let $v \in \mathbb{C}^3$ be a constant vector. Then $i_{v} (\pr_{S^5}^*F_{A_0}) \in \mathcal{K}(L_{A_0})_{-2}$. Moreover, if $F_{A_0} \neq 0$, then the map 
\begin{align*}
\mathbb{C}^3 \to \mathcal{K}(L_{A_0})_{-2}, v  \mapsto i_{v} (\pr_{S^5}^*F_{A_0})
\end{align*}
is injective.
\end{proposition}
A proof of the previous proposition appeared in \cite[Section~4.3]{Wang-AtiyahClasses}. We have included a short sketch for the convenience of the reader.
\begin{proof}
Since $\pr_{S^5}^*A_0$ is an instanton over $\mathbb{C}^3\setminus \{0\}$, it is in particular Yang--Mills. Thus,
\begin{align*}
\diff_{\pr_{S^5}^*A_0}^* \big(i_{v} (\pr_{S^5}^*F_{A_0})\big) &= -\textstyle{\sum} i_{e_i} \nabla_{e_i}\big(i_{v} (\pr_{S^5}^*F_{A_0})\big) = -\textstyle{\sum} i_{e_i}i_{v} \nabla_{e_i}(\pr_{S^5}^*F_{A_0}) \\
&=  - i_v \diff_{\pr_{S^5}^*A_0}^*(\pr_{S^5}^*F_{A_0}) = 0
\end{align*}
where $e_1,\dots,e_6$ is the standard (constant) frame of $\mathbb{C}^3$ and where we have used that $v$ is constant. Similarly, we obtain 
\begin{align*}
\diff_{\pr_{S^5}^*A_0} \big(i_{v} (\pr_{S^5}^*F_{A_0})\big) &= \nabla_v(\pr_{S^5}^*F_{A_0}) - i_v\big(\diff_{\pr_{S^5}^*A_0} (\pr_{S^5}^*F_{A_0})\big) = \nabla_v(\pr_{S^5}^*F_{A_0})
\end{align*}
by the Bianchi-identity, and therefore
\begin{align*}
\diff_{\pr_{S^5}^*A_0} \big(i_{v} (\pr_{S^5}^*F_{A_0})\big) \wedge \Im \Omega_0 &= \nabla_v(\pr_{S^5}^*F_{A_0}) \wedge \Im \Omega_0 = \nabla_v \big( \pr_{S^5}^*F_{A_0} \wedge \Im \Omega_0 \big) = 0
\end{align*}
because $\Im \Omega_0$ is constant and $A_0$ is an instanton. The equation $\Lambda_{\omega_0}\diff_{\pr_{S^5}^*A_0} \big(i_{v} (\pr_{S^5}^*F_{A_0})\big) = 0$ is proven analogously, which shows the first assertion.

In order to prove that $i_{v}(\pr_{S^5}^*F_{A_0}) \neq 0$ for $v \neq 0$, we first choose a point $z_0 \in S^5$ at which $\pr_{H}(v) \neq 0$ and $(F_{A_0})_{z_0} \neq 0$ (where $\pr_{H}$ denotes the projection to the contact distribution $H \subset TS^5$; cf. \autoref{subsec: Sasaki--Einstein structure on S5}). Moreover, let $\varepsilon^1,\dots,\varepsilon^4$ be a positively oriented orthonormal coframe of $H^*$ which satisfies $\pr_{S^5}^*\varepsilon^1(v) \neq 0$ and $\pr_{S^5}^*\varepsilon^i(v) = 0$ $i=2,3,4$. By \autoref{prop: dilation invariant instantons} we locally have \[ F_{A_0} = \xi_1 (\varepsilon^{12}-\varepsilon^{34}) +\xi_2 (\varepsilon^{13}-\varepsilon^{42}) + \xi_3 (\varepsilon^{14}-\varepsilon^{23})  \] for some local sections $\xi_1,\xi_2,\xi_3$ of $\mathfrak{g}_{P_0}$. This explicit form shows $i_{v}(\pr_{S^5}^*F_{A_0}) \neq 0$ and concludes the proof.
\end{proof}

\subsubsection{Proof of the non-degenerateness}\label{sec: proof of non-degenerateness}

Assume that we are in the situation of \autoref{subsec: pairing for the cokernel} with a framed conically singular $\SU(3)$-instanton $A$ on $\pi \col P \to Z\setminus S$ whose (infinitesimally irreducible) tangent cones $(\pi_i \col P_i \to S^5,A_i)$ all satisfy \autoref{ass: critical rates contained in Z}.
\begin{theorem}\label{thm: non-degenerateness of pairing}
Assume additionally that $\ker(L_A)_{-5/2} = 0$. Then the pairing introduced in \autoref{prop: obstruction pairing} is non-degenerate on the left. That is, $\int \langle \Phi_{\mathbb{A}}(\vec{v},\vec{u}),\underline{a}\rangle \vol_Z = 0$ for all $\underline{a} \in \ker(L_A)_{-5-\mu}$ implies $(\vec{v},\vec{u})=0 \in (\mathbb{C}^3)^N \oplus (\oplus_i \mathfrak{m}_i)$. Equivalently, the map $\Phi_\mathbb{A}$ is injective.
\end{theorem}
\begin{remark}
Because $L_A$ is formally self-adjoint, we have $\text{index}(L_A\col C^{k,\alpha}_{-5/2} \to C^{k,\alpha}_{-5/2-1} ) = 0$. We therefore hope that the assumption in the previous theorem holds after a suitable generic perturbation.
\end{remark}
\begin{corollary}\label{cor: estimate dimension cokernel}
The previous theorem and \autoref{prop: observations on cokernel} imply that if $\ker(L_A)_{-5/2}=0$, then \[ \dim \textup{coker}(\mathfrak{L}^{k,\alpha}_{\mathbb{A}}) = - 6N + \big(\sum_i \dim \mathcal{K}(L_{A_i})_{-3} + \dim \mathcal{K}(L_{A_i})_{-4}- \dim \mathfrak{m}_i\big)   \]
and $\dim \ker(\mathfrak{L}^{k,\alpha}_{\mathbb{A}}) = 0$.
\end{corollary}

In order to proof \autoref{thm: non-degenerateness of pairing} we need the following result:
\begin{proposition}[{\cite[Equation~(41)]{Wang-spectrum_of_operator_for_instantons}}]\label{prop: duality between homogeneous kernels}
Fix $i\in \{1,\dots,N\}$ and assume that \[r^{\lambda}(\xi_1,\xi_2, a_r \diff r + a_\theta (r\theta) + r a_H) \in \mathcal{K}(L_{A_i})_{\lambda}\] where $\xi_1,\xi_2,a_r,a_\theta \in \Omega^0(S^5,\mathfrak{g}_{P_i})$ and $a_H \in \Omega^0(S^5,H^*\otimes \mathfrak{g}_{P_i})$ is horizontal (with respect to the canonical Sasaki--Einstein structure on $S^5$ as revised in \autoref{subsec: Sasaki--Einstein structure on S5}) and $\theta \in \Omega^1(S^5)$ is the canonical contact 1-form (also reviewed in \autoref{subsec: Sasaki--Einstein structure on S5}). Then \[r^{-5-\lambda}(-a_r,a_\theta,\xi_1 \diff r - \xi_2 (r\theta) + r J_2^*a_H) \in \mathcal{K}(L_{A_i})_{-5-\lambda},\] where $J_2$ is the horizontal complex structure $J_2$ on $H \subset TS^5$ induced by $\omega_2$ (cf. \autoref{subsec: Sasaki--Einstein structure on S5}).
\end{proposition}
The proof of the previous proposition is given in~\cite[Section~2.2]{Wang-spectrum_of_operator_for_instantons}. For the convenience of the reader we indicate a proof.
\begin{proof}
Under the isomorphism 
\begin{align*}
\mathfrak{g}_{P_i} \oplus \mathfrak{g}_{P_i} \oplus T^*(\mathbb{C}^3\setminus\{0\}) \otimes \mathfrak{g}_{P_i} &\cong \mathfrak{g}_{P_i} \oplus \mathfrak{g}_{P_i} \oplus \mathfrak{g}_{P_i} \oplus \mathfrak{g}_{P_i} \oplus H^* \otimes \mathfrak{g}_{P_i} \\
\big( \xi_1,\xi_2, a_r \diff r + a_\theta \theta + a_H \big) &\mapsto \big( \xi_1,\xi_2,a_r, \tfrac{1}{r} a_\theta , \tfrac{1}{r}a_H \big)
\end{align*}
(where we have suppressed the pullback from $S^5$ from our notation) the model operator\footnote{Note that our model operator differs from the one used in \cite{Wang-spectrum_of_operator_for_instantons} in the following way: The isomorphism between $\mathfrak{g}_{P_i} \oplus \mathfrak{g}_{P_i} \oplus T^*(\mathbb{C}^3\setminus \{0\})\otimes \mathfrak{g}_{P_i}$ and $\mathfrak{g}_{P_i}\oplus \mathfrak{g}_{P_i}\oplus \mathfrak{g}_{P_i}\oplus \mathfrak{g}_{P_i} \oplus H^* \otimes \mathfrak{g}_{P_i}$ considered here differs to the one in \cite{Wang-spectrum_of_operator_for_instantons} by an overall factor of $r^{-1}$. This implies that $P_{\textup{here}} = P_{\textup{Wang}}- \textup{Id}$ (modulo the fact that~\cite{Wang-spectrum_of_operator_for_instantons} uses a different (but equivalent) $\SU(3)$-instanton equation).} becomes
\begin{align*}
L_{A_i} &\equiv \underline{J} \big( \partial_r - \tfrac{P}{r} \big)
\end{align*}
for 
\begin{align*}
\underline{J} = \begin{pmatrix}
0 & 0 & -1 & 0 & 0 \\
0 & 0 & 0 & 1 & 0 \\
1 & 0 & 0 & 0 & 0 \\
0 & -1 & 0 & 0 & 0 \\
0 & 0 & 0 & 0 & J_2^* 
\end{pmatrix}
\end{align*}
and 
\begin{align*}
P = \begin{pmatrix}
0 & -\nabla_{X_\theta} & 0 & 0 &  \Lambda_{\omega_2}^H \diff_H \\ \nabla_{X_\theta} & 0 & 0 & 0 &  \Lambda_{\omega_3}^H \diff_H \\
0 & 0 & -5 & - \nabla_{X_\theta} &  \diff_H^* \\
0 & 0 &  \nabla_{X_\theta} & -5 & -\Lambda_{\omega_1}^H \diff_H \\
J_2^* \diff_H & J_3^* \diff_H &  \diff_H & -J_1^* \diff_H & -1 + J_1^* \diff_{X_\theta} 
\end{pmatrix}.
\end{align*}
In the previous formula for $P$, $\diff_H$ denotes the projection of $\diff_{A_i} \col \Omega^k_{\text{hor}}(S^5,\mathfrak{g}_{P_i}) \to \Omega^{k+1}(S^5,\mathfrak{g}_{P_i})$ to its horizontal part lying in $\Lambda^{k+1}H^* \otimes \mathfrak{g}_{P_i}$ and $\diff_{X_\theta}$ denotes the linear differential operator acting on any $\alpha_H \otimes \xi \in \Omega^0(S^5,H^*\otimes \mathfrak{g}_{P_i})$ via \[ \diff_{X_\theta} (\alpha_H \otimes \xi) \coloneqq (L_{X_\theta} \alpha_H) \otimes \xi + \alpha_H \otimes (\nabla_{X_\theta} \xi) \] (where ${X_\theta}$ is the Reeb vector field to $\theta$, i.e. the infinitesimal generator of the $\U(1)$-action on $S^5$). 

The formula for $L_{A_i}$ implies that the homogeneous kernel elements of $L_{A_i}$ of degree $\tilde{\lambda}\in \mathbb{R}$ correspond to eigensections of $P$ with eigenvalue $\tilde{\lambda}$. A calculation now shows $\underline{J}P + P \underline{J} = -5 \underline{J}$. Thus, $\underline{J}$ maps eigensections of $P$ with eigenvalue $\lambda$ to eigensections with eigenvalue $- \lambda - 5$. This implies the statement.
\end{proof}
\begin{remark}
Note that if $r^{\lambda}(\xi_1,\xi_2, a_r \diff r + a_\theta (r\theta) + r a_H) \in \mathcal{K}(L_{A_i})_{\lambda}$ is a homogeneous kernel element, then by the same proof as in \autoref{prop: augmented instanton equation} we may conclude that $\Delta_{A_i}(r^{\lambda} \xi_1) = 0$. Thus, if $\lambda \in [-4,0]$, \autoref{prop: critical rates of model Laplacian} implies $\xi_1 = 0$. Similarly, one can then conclude $\xi_2 = 0$. By the previous proposition, all homogeneos kernel elements $\underline{a}\in \mathcal{K}(L_{A_i})_\lambda$ for $\lambda \in [-4,-1]$ are therefore of the form $\underline{a}=(0,0,r^{\lambda+1}a_H) \in \Omega^0(\mathbb{C}^3\setminus \{0\},\mathfrak{g}_{P_i} \oplus\mathfrak{g}_{P_i} \oplus T^*\mathbb{C}^3 \otimes \mathfrak{g}_{P_i})$, where $a_H\in \Omega^1(S^5,\mathfrak{g}_{P_i})$ is $H$-horizontal.
\end{remark}
\begin{proof}[Proof of \autoref{thm: non-degenerateness of pairing}]
We again assume that $N =1$ and drop the subscripts to ease notation. Let now $(v,u) \in \mathbb{C}^3 \oplus \mathfrak{m}_s$ be such that $\int \langle \Phi_{\mathbb{A}}(v,u),\underline{a}\rangle \vol_Z = 0$ for all $\underline{a} \in \ker(L_A)_{-5-\mu}$. Since $\ker(L_A)_{-5/2} = 0$, we have for every kernel element of the form \[(0,0,r^{-3}\tilde{a}_s + r^{-2} \tilde{b}_s) \in \ker (L_{A_s}), \] with $\tilde{a}_s,\tilde{b}_s\in \Omega
^1(S^5,\mathfrak{g}_{P_i})$ being horizontal, an element $\underline{a} = (0,0,a) \in \ker(L_{A})_{-5-\mu}$ with \[ \big\vert \tilde{\Upsilon}^*_s a - r^{-3}\tilde{a}_s- \eta_s(\tilde{a}_s) - r^{-2} \tilde{b}_s \big\vert = \mathcal{O}(r^{-2+ \varepsilon})\] (cf. \autoref{app-prop: for every expansion there is a kernel element}). Then by \autoref{prop: obstruction pairing} 
\begin{align*}
0=\int_Z \langle \Phi_\mathbb{A}(v,u),\underline{a}\rangle \vol_Z = \int_{S^5} \langle i_{\hat{u}_{S^5}} F_{A_s},J_2^*\tilde{a}_s \rangle + \langle i_v F_{A_s},J_2^* \tilde{b}_s \rangle \vol_{S^5} + \mathfrak{e}_s(v,\tilde{a}_s),
\end{align*}
where $\hat{u}_{S^5}\in \Gamma(TS^5)$ is the vector field defined at $z\in S^5$ by $u \cdot S^5 \in \{z\}^\perp = T_zS^5$, and where $\mathfrak{e}_s$ is bilinear. By \autoref{prop: infinitesimal translations and Atiyah classes} and \autoref{prop: duality between homogeneous kernels} we may choose $\tilde{a}_s=0$ and $\tilde{b}_s = J_2^*(i_{v_{\vert S^5}} F_{A_s}) \neq 0$. \autoref{prop: infinitesimal translations and Atiyah classes} implies then $v=0$ (this is because we assume $A_s$ to be infinitesimally irreducible and therefore $F_{A_s} \neq 0$). A similar argument using \autoref{prop: infinitesimal rotations and deformations of tangent cone} and the bilinearity of $\mathfrak{e}_s$ implies $u=0$.
\end{proof}

\begin{remark}\label{rem: overcoming obstructions via deformations of tangent cone}
The previous proof shows that the infinitesimal translations of a singular point $s \in S$ overcome (some) of the obstructions arising from $\mathcal{K}(L_{A_s})_{-3}$ whereas the infinitesimal rotations of the bundle around $s$ overcome (some) of the obstructions coming from $\mathcal{K}(L_{A_s})_{-4}$. Recall from \autoref{prop: duality between homogeneous kernels} that $J_2^*$ (which also appears on the right-hand side of the pairing in \autoref{prop: obstruction pairing}) induces an isomorphism between $\mathcal{K}(L_{A_s})_{-4}$ and $\mathcal{K}(L_{A_s})_{-1}$ and that $\mathcal{K}(L_{A_s})_{-1}$ consists precisely of the deformations of the tangent connection $A_s \in \mathcal{A}(P_s)$ (cf. \autoref{prop: homogeneous kernel of degree -1}). The phenomenon that the infinitesimal rotations of $\pi \col P \to Z\setminus S$ around $s$ overcome the obstructions arising from $\mathcal{K}(L_{A_s})_{-4}\cong \mathcal{K}(L_{A_s})_{-1}$ therefore again supports the interpretation of these rotations as deformations of $A \in \mathcal{A}(P)$ that also deform the tangent connection stated in \autoref{rem: alternative interpretation of rotations}. 

More generally, we believe that in a 'full' moduli theory, which allows for variable tangent cones (cf. \autoref{rem: totally ordered singular set framed case} and \autoref{rem: full moduli theory for unframed connections}), the deformations of $A$ that also deform the tangent connection $A_s \in \mathcal{A}(P_s)$ will overcome all obstructions in $\mathcal{K}(L_{A_s})_{-4} = J_2^*\mathcal{K}(L_{A_s})_{-1}$ that arise from integrable infinitesimal deformations.
\end{remark}

\section{Instantons with structure group $G= \mathbb{P}\U(n)$}\label{sec: structure group PU(n)}

In this section we apply the previous results to $\SU(3)$-instantons with structure group $G= \mathbb{P}\U(n)$. For this we first begin with the following (well-known) properties of $\mathbb{P}\U(n)$-connections over $S^5$ and $\mathbb{P}^2$:

\begin{observation}
Let $\hat{\pi}_0^{\circ} \col \hat{P}_0^{\circ} \to S^5$ be a $\mathbb{P}\U(n)$-bundle together with an irreducible connection $\hat{A}^{\circ}_0 \in \mathcal{A}(\hat{P}_0^{\circ})$ satisfying~\eqref{equ: cone reduction of instanton equation}. Since $\mathbb{P}\U(n)$ has a trivial center, \autoref{prop: tangent cones pulled back from P2} implies that there is a $\mathbb{P}\U(n)$-bundle $\pi_0^{\circ} \col P_0^{\circ} \to \mathbb{P}^2$ and an ASD-instanton $A^{\circ}_0 \in \mathcal{A}(P_0^{\circ})$ such that $\hat{P}_0^{\circ} = \pr_{\mathbb{P}^2}^*P_0^{\circ}$ and $\hat{A}_0^{\circ} = \pr_{\mathbb{P}^2}^*A_0^{\circ}$.
\end{observation}

Let $(\pi_0^\circ \col P_0^\circ \to \mathbb{P}^2,A_0^\circ)$ be a $\mathbb{P}\U(n)$-bundle together with an ASD-instanton $A_0^\circ \in \mathcal{A}(P_0^\circ)$ and denote by $(\hat{\pi}_0^\circ \col \hat{P}_0^\circ \to S^5,\hat{A}_0^\circ) \coloneqq \pr_{\mathbb{P}^2}^*(\pi_0^\circ \col P_0^\circ \to \mathbb{P}^2,A_0^\circ)$ their respective pullbacks to $S^5$. Recall the group $\Stab_{\SU(3)}(\hat{A}_0^{\circ})$ from \eqref{equ: definition Stab_SU(3)(A_s)}. Similarly, we define \begin{equation}\label{equ: definition Stab_SU(3)(A_s) over P2}
\Stab_{\SU(3)}(A_0^\circ) \coloneqq \{ \tilde{U} \col P_0^\circ \xrightarrow{\sim} P_0^\circ \mid \textup{$\tilde{U}$ covers an element in $\SU(3)$ and $\tilde{U}^*A_0^\circ = A_0^\circ$}\} 
\end{equation}
where $\SU(3)$ acts on $\mathbb{P}^2=S^5/\U(1)$ in the obvious way.

\begin{proposition}\label{prop: isomorphism of stabiliser groups}
Pulling back a gauge transformation from $P_0^\circ$ to $\hat{P}_0^\circ$ induces an isomorphism between $\Stab_{\SU(3)}(A_0^\circ)$ and $\Stab_{\SU(3)}(\hat{A}_0^{\circ})$.
\end{proposition}

\begin{proof}
Let $\tilde{U} \col \hat{P}_0^\circ \to \hat{P}_0^\circ$ be a bundle isomorphism that preserves $\hat{A}_0^\circ$. As in the previous sections, we may regard $\tilde{U}$ as a section of a vector bundle associated to $\hat{P}_0^{\circ}$ (namely, $\tilde{U} \in \Gamma(\Hom(\hat{P}_0^\circ\times_{\mathbb{P}\U(n)}W, U^*\hat{P}_0^\circ\times_{\mathbb{P}\U(n)}W))$ where $W$ is a vector space such that $\mathbb{P}\U(n) \subset \GL(W)$). Since $\hat{P}_{0}^\circ$ is pulled back from a bundle over $\mathbb{P}^2$, there exists a canonical $\U(1)$-action on $\hat{P}_{0}^\circ$ covering the action on $S^5$. Moreover, we have $\nabla^{\hat{A}_0^{\circ}} \tilde{U} = 0$ because $\tilde{U}$ preserves $\hat{A}_0^{\circ}$. Since $\hat{A}_0^{\circ}$ is also pulled back from a connection over $\mathbb{P}^2$, this implies $L_{X_{\theta}}\tilde{U} = 0$, where $X_{\theta}\in \Gamma(TS^5)$ denotes the infinitesimal generator of the $\U(1)$-action and $L_{X_\theta}$ its Lie derivative. That is, $\tilde{U}$ is $\U(1)$-invariant and therefore pulled back from an isomorphism $P_0^{\circ} \to P_0^{\circ}$ over $\mathbb{P}^2$. This implies the result.
\end{proof}

Finally, we have the following result about lifting $\mathbb{P}\U(n)$-connections to $\U(n)$-connections:

\begin{proposition}\label{prop: PU(n)-bundles and instantons are liftible to U(n)-bundles and HYM}
Let $\pi_0^\circ \col P_0^\circ \to \mathbb{P}^2$ be a $\mathbb{P}\U(n)$-bundle. Then there exists a $\U(n)$-bundle $\pi_0 \col P_0 \to \mathbb{P}^2$ such that $P_0^\circ = P_0/\U(1)$. Moreover, $\pi_0 \col P_0 \to \mathbb{P}^2$ is unique up to isomorphism and twisting by a $\U(1)$-bundle. If $A^\circ_0 \in \mathcal{A}(P_0^\circ)$ is an ASD-instanton, then there exists a Hermitian Yang--Mills connection $A_0 \in \mathcal{A}(P_0)$ (i.e. a connection that satisfies \[F_{A_0}^{0,2} = 0 \quad \textup{and} \quad i\Lambda_{\omega_{\textup{FS}}} F_{A_0} = \lambda \cdot \textup{id}\] for some constant $\lambda \in \mathbb{R}$) such that $A_0^\circ$ is induced by $A_0$. The connection $A_0 \in \mathcal{A}(P_0)$ is unique up to $\U(1)$-gauge transformations $C^\infty(\mathbb{P}^2,\U(1)) \subset \mathcal{G}(P_0)$.
\end{proposition}

This proposition is well-known and we therefore only sketch a proof for the convenience of the reader.
\begin{proof}
The short exact sequence
\begin{equation*}
\begin{tikzcd}
1 \arrow[r] & \U(1) \arrow[r] & \U(n) \arrow[r] & \mathbb{P}\U(n) \arrow[r] & 1
\end{tikzcd}
\end{equation*}
induces the following exact sequence in \v{C}ech cohomology:
\begin{equation*}
\begin{tikzcd}
\check{\mathrm{H}}^1(\mathbb{P}^2,\U(1)) \arrow[r] & \check{\mathrm{H}}^1(\mathbb{P}^2,\U(n)) \arrow[r] & \check{\mathrm{H}}^1(\mathbb{P}^2,\mathbb{P}\U(n))  \arrow[r,"\beta"] & \check{\mathrm{H}}^2(\mathbb{P}^2,\U(1)).
\end{tikzcd}
\end{equation*}
Thus, the (isomorphism class of the) $\mathbb{P}\U(n)$-bundle $[\pi_0^\circ \col P_0^\circ \to \mathbb{P}^2] \in \check{\mathrm{H}}^1(\mathbb{P}^2,\mathbb{P}\U(n))$ lifts to a $\U(n)$-bundle if and only if $\beta([\pi_0^\circ \col P_0^\circ \to \mathbb{P}^2]) = 0 \in \check{\mathrm{H}}^2(\mathbb{P}^2,\U(1))$. Moreover, since $\U(1) \subset \U(n)$ lies central, any two lifts differ by a twist with a $\U(1)$-bundle and an isomorphism. The first part of the proposition now follows from the observation $\check{\mathrm{H}}^2(\mathbb{P}^2,\U(1)) \cong \mathrm{H}^3(\mathbb{P}^2,\mathbb{Z}) = 0$, which follows from the same long exact sequence in \v{C}ech cohomolgy associated to 
\begin{equation*}
\begin{tikzcd}
0 \arrow[r] & \mathbb{Z} \arrow[r] & \mathbb{R} \arrow[r] & \U(1) \arrow[r] & 0
\end{tikzcd}
\end{equation*}
and the observation that $\check{\mathrm{H}}^i(\mathbb{P}^2,C^\infty(\ \cdot \ ,\mathbb{R})) = 0$ for $i>0$ (cf. \cite[Proposition~8.5]{BottTu-differentialForms}).

For the second part, note that a lift $A_0 \in \mathcal{A}(P_0)$ of a connection $A_0^\circ \in \mathcal{A}(P_0^\circ)$ is equivalent to a choice of a connection $A_{\det}$ on the determinant $\U(1)$-bundle $P_0 \times_{\det} \U(1)$. Moreover, $A_0$ is Hermitian Yang--Mills if and only if $A_0^\circ$ is an ASD-instanton and $A_{\det}$ is Hermitian Yang--Mills. Since $\mathrm{H}^{0,1}(\mathbb{P}^2, \mathbb{C}) = 0$, there exists an up to gauge unique connection $A_{\det} \in \mathcal{A}(P_0 \times_{\det} \U(1))$ with this property (cf. \cite[Lemma~4.B.4]{Huybrechts-complex-Geometry}).
\end{proof}

Let now $Z^6$ again be a compact 6-manifold with an $\SU(3)$-structure $(\omega,\Omega)$ that satisfies $\diff^*\omega = 0$ and $\diff \Omega = w_1 \omega^2$ for some $w_1 \in \mathbb{R}$. Moreover, let $N\in \mathbb{N}$ and for every $i=1,\dots,N$ let $(\hat{\pi}_i^{\circ} \col \hat{P}_i^{\circ} \to S^5,\hat{A}_i^\circ)$ be a $\mathbb{P}\U(n)$-bundle (for $n\geq 2$) together with an irreducible connection $\hat{A}_i^\circ \in \mathcal{A}(\hat{P}_i^{\circ})$ satisfying~\eqref{equ: cone reduction of instanton equation}. By the previous results, there exists for every $i=1,\dots,N$ a $\U(n)$-bundle $\pi_i \col P_i \to \mathbb{P}^2$ together with a Hermitian Yang--Mills connection $A_i$, such that $\hat{P}_i^\circ = (\pr_{\mathbb{P}^2}^*P_i)/\U(1)$ and $\hat{A}_i^\circ$ is induced by $\pr_{\mathbb{P}^2}^*A_i$. Denote by $E_i\to \mathbb{P}^2$ the complex vector bundle associated to $P_i$. Note that the Hermitian Yang--Mills connection $A_i$ equips $E_i \to \mathbb{P}^2$ with a holomorphic structure. The following result due to the second named author determines the (relevant) critical rates of the operators $L_{\hat{A}_i^\circ}$ and relates their corresponding homogeneous kernels to the cohomology of the holomorphic vector bundle $(\End E_i)(\ell) \coloneqq (\End E_i) \otimes \mathcal{O}(\ell)$.

\begin{proposition}[{\cite[Theorem~1.8]{Wang-spectrum_of_operator_for_instantons}}]
The critical rates of $L_{\hat{A}_i^\circ}$ satisfy \[ \{-3,-2\} \subset \mathcal{D}(L_{\hat{A}_i^\circ}) \cap [-4,-1] \subset \{-4,-3,-2,-1\} \] and for $\nu_i \in \{-4,-3,-2,-1\}$ we have \[ \dim \mathcal{K}(L_{\hat{A}_i^\circ})_{\nu_i} = 2 \mathrm{h}^1(\mathbb{P}^2,(\End E_i)(\nu_i+1)), \] where $\mathrm{h}^1(\mathbb{P}^2,(\End E_i)(\nu_i+1)) = \dim_\mathbb{C}(\mathrm{H}^1(\mathbb{P}^2,(\End E_i)(\nu_i+1))$.
\end{proposition}

Inserting this into the virtual-dimension formula of \autoref{thm: Kuranishi structure} and using \autoref{prop: isomorphism of stabiliser groups} immediately proves the first part of \autoref{introthm: moduli space for PU(n)-connections}:

\begin{corollary}\label{cor: virtual dimension for structure group PU(n)}
For $\mu \in (-1,\Bar{\mu}_1)\times \dots \times (-1,\Bar{\mu}_N)$, where $\bar{\mu}_i \coloneqq \min \{((-1,0)\cap \mathcal{D}(L_{\hat{A}_i^\circ})) \cup \{0\}\}$ we have \begin{align*}
\textup{virt-dim}\big(\mathcal{M}_\mu(\{\hat{P}_i^\circ,\hat{A}_i^\circ\})\big) = \textstyle{\sum}_{i=1}^{N} 6 &+ (8-\dim \Stab_{\SU(3)}(A_i^\circ)) \\
&- 2\mathrm{h}^{1}(\mathbb{P}^{2}, \End E_i) - 2\mathrm{h}^{1}(\mathbb{P}^{2}, (\End E_i)(-1)). 
\end{align*}
\end{corollary}

The following example based on \cite[Corollary~1.11]{Wang-spectrum_of_operator_for_instantons} calculates the virtual dimension in the case of conically singular instantons whose tangent cones are modelled on the Fubini--Study connection on $T\mathbb{P}^2$.

\begin{example}\label{example: moduli space for tangent cone = Fubini-Study}
Let $(T\mathbb{P}^2,h_{\textup{FS}})$ be the holomorphic tangent bundle over $\mathbb{P}^2$ equipped with the Fubini--Study metric. The corresponding Chern--connection $A_{\textup{FS}}$ is Hermitian Yang--Mills (cf. \cite[Example~4.B.16]{Huybrechts-complex-Geometry}) and the induced connection $A_{\textup{FS}}^\circ$ on the associated $\mathbb{P}\U(2)$-bundle $\mathbb{P}\U(T\mathbb{P}^2,h_{\textup{FS}})$ over $\mathbb{P}^2$ is therefore an ASD-instanton. Thus, by \autoref{cor: ASD-instantons = pulled back SU(3)-instantons} the pullback of $(\mathbb{P}\U(T\mathbb{P}^2,h_{\textup{FS}}),A_{\textup{FS}}^\circ)$ to $S^5$ satisfies~\eqref{equ: cone reduction of instanton equation}. Moreover, \cite[Corollary~1.11]{Wang-spectrum_of_operator_for_instantons} shows:
\begin{align*}
\mathcal{D}(L_{\hat{A}_{\textup{FS}}^\circ}) \cap [-4,-1] &= \{-3,-2\} \\
2 \textrm{h}^1(\mathbb{P}^2,(\End T\mathbb{P}^2)(-1)) &= 6 \\
\min \{((-1,0)\cap \mathcal{D}(L_{\hat{A}_{\textup{FS}}^\circ})) \cup \{0\}\} &= 2 \sqrt{2}-3
\end{align*}
and \autoref{prop: infinitesimal rotations and deformations of tangent cone} or the observation that $(\mathbb{P}^2,g_{\textup{FS}})$ is the symmetric space $\SU(3)/\textup{S}(\U(1)\times \U(2))$ (see also \autoref{bul: relationship SE structure S5 and K structure on P2} of \autoref{subsec: Sasaki--Einstein structure on S5}) implies \[\dim \Stab_{\SU(3)}(A_{\textup{FS}}^\circ)) = \dim \SU(3) = 8.\]

Thus, if $N \in \mathbb{N}$ and all prescribed tangent connections $(\hat{\pi}_i^{\circ} \col \hat{P}_i^{\circ} \to S^5,\hat{A}_i^\circ)$ for $i=1,\dots,N$ are isomorphic to the pullback of $(\mathbb{P}\U(T\mathbb{P}^2,h_{\textup{FS}}),A_{\textup{FS}}^\circ)$ to $S^5$, then for $\mu \in (-1,2\sqrt{2}-3)^N$, we have \[\textup{virt-dim}\big(\mathcal{M}_\mu(\{\hat{P}_i^{\circ},\hat{A}_i^{\circ}\})\big) = 0. \] Moreover, the homeomorphism type of $\mathcal{M}_\mu(\{P_i,A_i\})$ is independent of the choice of $\mu \in (-1,2\sqrt{2}-3)^N$ (cf. \autoref{thm: relating different rates}).
\end{example}

If, on the other hand, $(\pi_0^{\circ} \col P_0^{\circ} \to \mathbb{P}^2,A_0^{\circ})$ is a non-flat ASD-instanton with structure group $\mathbb{P}\U(n)$-that is \textit{not} isomorphic to the Fubini--Study connection on $\mathbb{P}\U(T\mathbb{P}^2,h_{\textup{FS}})$, then \cite[Proposition~4.1]{Wang-AtiyahClasses} shows that \[2h^1(\mathbb{P}^2,(\End E_0)(-1)) > 6, \] where $E_0 \to \mathbb{P}^2$ is the holomorphic vector bundle associated to $(P_0^{\circ},A_0^{\circ})$ via \autoref{prop: PU(n)-bundles and instantons are liftible to U(n)-bundles and HYM}. Moreover, in \autoref{prop: infinitesimal rotations and deformations of tangent cone} we have seen that \[ 8-\dim \Stab_{\SU(3)}(A_0^{\circ})  \leq 2 \mathrm{h}^1(\mathbb{P}^2,\End(E_0)). \] Together with \autoref{cor: virtual dimension for structure group PU(n)} and the previous example this implies:

\begin{theorem}\label{thm: non-positive dimension for PU(n) connections}
Let $N \in \mathbb{N}$ and for each $i=1,\dots,N$ let $(\hat{\pi}_i^{\circ} \col \hat{P}_i^{\circ} \to S^5,\hat{A}_i^\circ)$ be a $\mathbb{P}\U(n)$-bundle (for $n>1$) together with an irreducible (hence non-flat) connection $\hat{A}_i^\circ \in \mathcal{A}(\hat{P}_i^{\circ})$ satisfying~\eqref{equ: cone reduction of instanton equation}. Then \[\textup{virt-dim}\big(\mathcal{M}_\mu(\{\hat{P}_i^\circ,\hat{A}_i^\circ\})\big) \leq 0\] with equality if and only if all $(\hat{\pi}_i^{\circ} \col \hat{P}_i^{\circ} \to S^5,\hat{A}_i^\circ)$ are isomorphic to the pullback of the Fubini--Study connection $(\pi\col \mathbb{P}\U(T\mathbb{P}^2,h_{\textup{FS}}) \to \mathbb{P}^2,A_{\textup{FS}}^\circ)$.
\end{theorem}

\begin{remark}\label{rem: only Fubini--Study connections appear generically}
Recall from \autoref{thm: Kuranishi structure} and (the construction prior to) \autoref{thm: local structure of unframed B Phi is homeo} that the contribution of $8-\dim \Stab_{\SU(3)}(A_i^\circ)$ in the virtual dimension formula of \autoref{cor: virtual dimension for structure group PU(n)} for every $i=1,\dots,N$ came from rotating the bundle around the singular point $s_i \in S$. Moreover, recall from \autoref{rem: alternative interpretation of rotations} that these rotations may also be interpreted as deformations of the conically singular $\SU(3)$-instanton that change the tangent connection at $s_i$ (inside a certain class of connections on $\hat{P}_i^{\circ}$). In a universal moduli space $\mathcal{M}$ of conically singular instantons with structure group $\mathbb{P}\U(n)$ (which allows for variable tangent connections; cf. \autoref{rem: full moduli theory for unframed connections}) the rotations parametrised by $\mathfrak{m}_i$ in \autoref{thm: Kuranishi structure} would therefore be replaced by the deformations of the tangent connections (cf. \cite[Section~5.1]{Bera-cs_associatives}). Since the deformations of the (irreducible) tangent connection $(\hat{\pi}_i^\circ \col \hat{P}_i \to S^5,\hat{A}_i^{\circ})$ form a moduli space of (real) virtual dimension $2h^1(\mathbb{P}^2,\End E_i)$, we expect the virtual dimension of the universal moduli space at an element whose tangents are precisely $(\hat{\pi}_i^\circ \col \hat{P}_i \to S^5,\hat{A}_i^{\circ})$ (for $i=1,\dots,N$) to be 
\begin{align*}
\textup{virt-dim}(\mathcal{M}) &= \textstyle{\sum}\ 6 + 2\mathrm{h}^1(\mathbb{P}^2,\End E_i)-2\mathrm{h}^1(\mathbb{P}^2,\End E_i)-2\mathrm{h}^1(\mathbb{P}^2,\End E_i(-1)) \\
&= \textstyle{\sum}\ 6 -2\mathrm{h}^1(\mathbb{P}^2,\End E_i(-1)).
\end{align*}
Therefore, \cite[Proposition~4.1]{Wang-AtiyahClasses} still implies that $\textup{virt-dim}(\mathcal{M}) \leq 0$. Moreover, equality holds precisely on the connected components of $\mathcal{M}$ consisting of those conically singular instantons whose tangent connections are all isomorphic to the pullback of the Fubini--Study connection. This suggests that after a generic perturbation of the instanton equations (cf. \cite[Chapter~5.5]{Donaldson-FloerHomology} and \cite[Section~4]{Ma-counting_flat_connections}), one only encounters singular instantons with Fubini--Study tangent connections.
\end{remark}

\appendix

\section{Analytic preliminaries}\label{app-sec: analytic preliminaries}

This section contains a summary of well-known analytic results for elliptic operators on bundles with isolated singularities that are used frequently throughout the text. The original references which develop the Fredholm theory of such operators mapping between weighted Sobolev spaces are \cite{LockhardMcOwen-ellipticOperators_on_noncompact_mfds} and \cite{MelroseMendoza--bCalculus} (see also \cite{Melrose-AtyiahPatodiSinger}). Good expositions and summaries can for example be found in \cite[Section~1]{Bartnik-mass_of_ALF}, \cite[Chapter~3]{Donaldson-FloerHomology}, \cite[Chapter~4]{Marshal-deformations_special_Lagrangians}, \cite[Section~4]{KarigiannisLotay-conifolds}, and \cite[Sections~3 and~4]{Langlais-analysis_of-neck-stretching_problems}. A treatment of the mapping properties of elliptic operators between weighted Hölder spaces can be found in \cite[Section~12]{Pacard-lecture_notes_connected_sums} and \cite[Section~2.1]{HaskinsHeinNordstroem--ACylCalabiYaus}.

\subsection{Conical operators}

We begin with a special class of bundles and operators over $\mathbb{R}^6\setminus \{0\}$ that interact with the dilation action $\delta_r \col \mathbb{R}^6\setminus \{0\} \to \mathbb{R}^6\setminus \{0\}$. We restrict ourselves to $\mathbb{R}^6 \setminus \{0\}$ for concreteness and because it is the relevant case for the rest of this article. Note, however, that the results discussed in this section also hold for general cones.

\begin{definition}\label{def: conical bundle}
A triple $(\pi \col E \to \mathbb{R}^6\setminus \{0\},h,\nabla)$ consisting of a vector bundle $\pi \col E \to \mathbb{R}^6\setminus \{0\}$, a bundle metric $h$ on $E$, and a metric connection $\nabla \col \Gamma(E) \to \Omega^1(\mathbb{R}^6\setminus \{0\},E)$ is called conically admissible, if one of the following two equivalent conditions is met:
\begin{enumerate}
\item There exists a bundle $\pi_0 \col E_0 \to S^5$ together with a bundle metric $h_0$ and a metric connection $\nabla_0 \col \Gamma(E_0) \to \Omega^1(S^5,E_0)$ such that $(E,h,\nabla) = \pr_{S^5}^*(E_0,h_0,\nabla_0)$.
\item The curvature of $\nabla$ satisfies $i_{\partial_r}F_{\nabla} = 0$. 
\end{enumerate}
\end{definition}
\begin{remark}
Note that parallel transport via $\nabla$ lifts the canonical dilation action $\delta \col \mathbb{R}_{>0} \times \mathbb{R}^6\setminus \{0\} \to \mathbb{R}^6 \setminus \{0\}$ to $\tilde{\delta} \col \mathbb{R}_{>0} \times E \to E$. By construction, this lift satisfies the following properties:
\begin{itemize}
\item for each $r\in \mathbb{R}_{>0}$, the corresponding map $\tilde{\delta}_r \col E \to E$  is a linear isometry
\item for each $e \in E$, $r \mapsto \tilde{\delta}_r(e)$ is parallel with respect to $\nabla$.
\end{itemize}
Using this lift, the equivalence between the two conditions in the previous definition follows by noting that $\tilde{\delta}_r^*\nabla = \nabla$ for every $r\in \mathbb{R}_{>0}$ if and only if $i_{\partial_r}F_\nabla=0$ (cf. \cite[Section~2.5.1]{Donaldson-FloerHomology}). In the following we will call a vector bundle $\pi \col E \to \mathbb{R}^6\setminus \{0\}$ conical, if we have fixed a lift $\tilde{\delta} \col \mathbb{R}_{>0} \times E \to E$ of the dilation action. 
\end{remark}

Let $\pi \col E \to \mathbb{R}^6\setminus \{0\}$ be a conical vector bundle together with its corresponding lift $\tilde{\delta}$ of the dilation action. Recall that for any $r\in \mathbb{R}_{>0}$ the isomorphism $\tilde{\delta}_r$ acts on sections $\Gamma(E)$ via pullback $\tilde{\delta}_r^*u \coloneqq \tilde{\delta}_{1/r} \circ u \circ \delta_{r} \in \Gamma(E)$ for $u\in \Gamma(E)$.

\begin{definition}
For $i=1,2$ let $\pi_i \col E_i \to \mathbb{R}^6\setminus \{0\}$ be conical vector bundles with corresponding lifts $\tilde{\delta}_i$ of the dilation action. A differential operator $L \col \Gamma(E_1) \to \Gamma(E_2)$ of order $\ell$ is called conical if \[ L \circ (\tilde{\delta}_1)_r^* = r^\ell \cdot  (\tilde{\delta}_2)_{r}^* \circ L \quad \textup{for every $r\in \mathbb{R}_{>0}$.} \]
\end{definition}
\begin{remark}
If $L \col \Gamma(E_1) \to \Gamma(E_2)$ is a conical differential operator of order $\ell$, then one can check that the product $r^\ell \cdot L$ (where $r$ now denotes the radius function) satisfies \[ (\tilde{\delta}_2)_{r_0}^* \circ (r^\ell \cdot L) = (r^\ell \cdot L) \circ (\tilde{\delta}_1)_{r_0}^* \quad \textup{for every $r_0\in \mathbb{R}_{>0}$.} \] Thus, when identifying $\mathbb{R}^6\setminus \{0\}$ with the cylinder $\mathbb{R} \times S^5$ by mapping $(t,\sigma) \in \mathbb{R} \times S^5 \mapsto \e^t \cdot \sigma$ the operator $\e^{\ell t} \cdot L$ is translation invariant (in the sense of~\cite{LockhardMcOwen-ellipticOperators_on_noncompact_mfds}).
\end{remark}
\begin{remark}\label{rem: concrete form of conical operators}
If one identifies the bundles $E_i$ with $\mathbb{R}_{>0} \times (E_i)_{\vert S^5}$ via $\tilde{\delta}_i$, then $L$ takes the form \[ L = r^{-\ell}\sum_{i=0}^\ell D_{\ell-i} (r \partial_r)^{i} = \sum_{i=0}^\ell r^{i-\ell} \tilde{D}_{\ell-i} \partial_r^{i} \] where $D_{\ell-i}, \tilde{D}_{\ell-i} \col \Gamma((E_1)_{\vert S^5}) \to \Gamma(E_2)_{\vert S^5})$ are ($r$-independent) differential operators of order at most $\ell-i$ over $S^5$.
\end{remark}

The following are the two examples of conical differential operators appearing in this article.

\begin{example}
Let $\pi_0 \col P_0 \to S^5$ be a principal $G$-bundle and let $A_0 \in \mathcal{A}(P_0)$ be a connection satisfying~\eqref{equ: cone reduction of instanton equation}. Then $\mathfrak{g}_{\pr_{S^5}^*P_0}$ and $T^*(\mathbb{C}^3\setminus \{0\})$ are conical bundles and the (rough) Laplacian $\Delta_{\pr_{S^5}^*A_0} \coloneqq \diff_{\pr_{S^5}^*A_0}^* \diff_{\pr_{S^5}^*A_0}$ on $\mathfrak{g}_{\pr_{S^5}^*A_0}$ and the $\SU(3)$-instanton deformation operator $L_{\pr_{S^5}^*A_0}$ (as defined prior to \autoref{prop: adjoint of model operator}) are conical differential operators.
\end{example}

In order to ease notation we will in the following assume that $L$ is a differential operator which maps between sections of the same bundle. Since we are primarily interested in elliptic operators, this is only a minor restriction.

\begin{definition} 
A section $u \in \Gamma(E)$ of a real conical vector bundle $E$ is called homogeneous of degree $\lambda \in \mathbb{R}$ if it satisfies $\tilde{\delta}_r^* u = r^{\lambda} u$ for every $r \in \mathbb{R}_{>0}$. Similarly, a section $\tilde{u} \in \Gamma(F)$ of a complex conical vector bundle $F$ is called homogeneous of degree $\lambda_c \in \mathbb{C}$ if it satisfies $\tilde{\delta}_r^* u = r^{\lambda_c} u$ for every $r \in \mathbb{R}_{>0}$.
\end{definition}
\begin{example}
Consider $T^*(\mathbb{R}^6\setminus \{0\})$ as a (real) conical bundle, where $\tilde{\delta}_r$ is given by radial parallel transport with respect to the Levi--Civita connection. A 1-form $\alpha \in \Omega^1(\mathbb{R}^6\setminus \{0\})$ is then homogeneous of degree $\lambda \in \mathbb{R}$ if and only if it is of the form \[ \alpha = \textstyle{\sum} f_i \diff x^{i},\] where $f_{i} \in C^\infty(\mathbb{R}^3\setminus \{0\})$ are homogeneous of degree $\lambda$ and $\diff x^1, \dots, \diff x^6$ are the canonical parallel 1-forms over $\mathbb{R}^6$. Note that this behaves under (ordinary) pullback (as a differential form) as $\delta_r^* \alpha = r^{\lambda+1} \alpha$. Conversely, if a 1-form $\alpha \in \Omega^1(\mathbb{R}^6\setminus \{0\})$ satisfies $\delta_r^* \alpha = r^{\lambda+1} \alpha$ for every $r \in \mathbb{R}_{>0}$, then it is homogeneous of degree $\lambda$.

More generally, if $E$ is a conical vector bundle with corresponding isomorphisms $\tilde{\delta}_r$, then $a \in \Omega^k(\mathbb{R}^6\setminus \{0\},E)$ is homogeneous of degree $\lambda$ if and only if $\tilde{\delta}_r^*a = r^{\lambda + k} a$ for every $r \in \mathbb{R}_{>0}$. Here, \[(\tilde{\delta}_r^* a)(v_1,\dots,v_k) \coloneqq \tilde{\delta}_r^{-1}\big(a(D\delta_r v_1,\dots, D\delta_r v_k)\big). \]
\end{example}

\begin{remark}
With the notion of homogeneous sections, one can give the following equivalent definition of a conical differential operator: Express $L = \sum_{i=0}^{\ell} p_i(\nabla^i_{\mathbb{R}^6})$ where $p_i \in \Gamma(\Hom(\underbrace{T^*\mathbb{R}^6\otimes \dots \otimes T^*\mathbb{R}^6}_{\text{$i$ times}}\otimes E,E))$ and the connection $\nabla_{\mathbb{R}^6}$ is the combination of the Levi--Civita connection acting on $T^*\mathbb{R}^6$ and the connection $\nabla$ on $E$ belonging to the conically admissible triple $(E,h,\nabla)$. Then $L$ is conical if and only if each $p_i$ is homogeneous of degree $i-\ell$. That is, if and only if the full symbol $p \coloneqq \sum_{i=0}^\ell p_i$ satisfies $\tilde{\delta}_r^*p = r^{-\ell} p$ where the pullback is to be understood similarly to the previous example via \[(\tilde{\delta}_r^* p_i)(\alpha_1,\dots,\alpha_i) \coloneqq \tilde{\delta}_r^{-1} \circ p_j(\delta_{1/r}^*\alpha_1,\dots,\delta_{1/r}^*\alpha_i) \circ \tilde{\delta}_r \in \End(E) \quad \text{for $\alpha_1,\dots,\alpha_i \in T^*\mathbb{R}^6$.} \]
\end{remark}

\begin{definition}\label{app-def: critical rates and homogeneous kernel}
For a conical differential operator $L \col \Gamma(E) \to \Gamma(E)$ (acting on sections of a real bundle $E$) we define \[ \mathcal{D}(L) \coloneqq \{ \Re (\lambda_c) \in \mathbb{R} \mid \exists \textup{ non-trivial homogeneous $u \in \Gamma(E_{\mathbb{C}})$ of degree $\lambda_c \in \mathbb{C}$ with $Lu=0$} \}, \] where $E_\mathbb{C}$ denotes the complexification $E \otimes \mathbb{C}$. Moreover, for every $\lambda_c \in \mathbb{C}$ we set 
\begin{align*}
V_{\lambda_c} \coloneqq \big\{ u= \textstyle{\sum}_{j=0}^m \log(r)^j u_j \in \Gamma(E_\mathbb{C}) \ \big\vert \  &\textup{$Lu=0$ and each $u_j \in \Gamma(E_{\mathbb{C}})$ is homogeneous} \\
& \qquad \qquad \qquad \qquad \qquad \qquad \qquad \quad \textup{of degree $\lambda_c$}\big\}
\end{align*}
and finally, for $\lambda \in \mathcal{D}(L)$ we define
\begin{align*}
\mathcal{K}(L)_\lambda \coloneqq \big\{ u= \textstyle{\sum}_{j=1}^m (u_{(\lambda_c)_j}+\bar{u}_{(\lambda_c)_j}) \in \Gamma(E) \ \big\vert \  &\textup{where all $(\lambda_c)_j \in \mathbb{C}$ satisfy} \\
&\qquad \textup{$\Re((\lambda_c)_j)=\lambda$ and $u_{(\lambda_c)_j} \in V_{(\lambda_c)_j}$}\big\}.
\end{align*}
\end{definition}

For the operators considered in this article, the sets $\mathcal{D}(L)$ and $\mathcal{K}(L)_\lambda$ are of the following simplified form:

\begin{proposition}\label{prop: condition for operator to have real roots and no log-terms}
Let $(E,h,\nabla) = \pr_{S^5}^*(E_0,h_0,\nabla_0)$ be conically admissible (where $E$ is a real vector bundle) and let $L \col \Gamma(E) \to \Gamma(E)$ be a conical differential operator. Following \autoref{rem: concrete form of conical operators} we express $L$ as \[ L = \tfrac{1}{r^\ell} \sum_{i=0}^\ell  D_{\ell-i} (r \partial_r)^{i} \] where $D_{\ell-i} \col \Gamma(E_0) \to \Gamma(E_0)$ are ($r$-independent) differential operators of order at most $\ell-i$ over $S^5$. Assume there exists an $L^2$-orthogonal basis $\{u_1,u_2,\dots \}$ of $\Gamma((E_0)_\mathbb{C})$ such that for every $i=0,\dots,\ell$ we have \[ D_{\ell-i} u_j = \nu_{\ell-i,j} u_j \quad \textup{for some $\nu_{\ell-i,j} \in \mathbb{C}$}.\] Assume further that for every $j \in \mathbb{N}$ the polynomial $p_j(z) \coloneqq \sum_{i=0}^\ell \nu_{\ell-i,j} z^{i} $ has only real roots of order one. Then \[ \mathcal{D}(L) = \{ \lambda \in \mathbb{R} \mid \exists \textup{ non-trivial homogeneous $u \in \Gamma(E)$ of degree $\lambda$ with $Lu=0$} \}, \] and for every $\lambda \in \mathcal{D}(L)$
\begin{align*}
\mathcal{K}(L)_\lambda &= \big\{ u \in \Gamma(E) \ \big\vert \  \textup{$Lu=0$ and $u$ is homogeneous of degree $\lambda$}\big\}.
\end{align*}
\end{proposition}
\begin{remark}
The previous proposition applies in particular to operators of the form $L = \partial_r + \tfrac{P}{r}$, where $P$ is a formally self-adjoint elliptic operator over $S^5$.
\end{remark}
This result is well-known and its proof is a simple abstraction of the ideas in \cite[Chapter~3]{Donaldson-FloerHomology}. For the convenience of the reader we give a short sketch:
\begin{proof}
We can express any $u \in \Gamma(\pr_{S^5}^*(E_0)_{\mathbb{C}})$ as $u=\sum_j a_j(r) u_j$. The equation $Lu=0$ is then equivalent to the following decoupled system of ordinary differential equations: \[ r^{-\ell}\big(\textstyle{\sum_{i=0}^\ell} \nu_{\ell-i,j} (r \partial_r)^{i} \big) a_j = 0 \quad \textup{for every $j\in \mathbb{N}$}. \] If the polynomial $p_j(z) = \sum_{i=0}^\ell \nu_{\ell-i,j} z^{i} $ has only real roots of order one, then $a_j(r)$ is a linear combination of $\{r^{\lambda_{j,1}},\dots, r^{\lambda_{j,\ell}}\}$ where $\lambda_{j,1},\dots,\lambda_{j,\ell} \in \mathbb{R}$ are the roots of $p_j$. This implies the proposition.
\end{proof}

Let $\pi_0 \col P_0 \to S^5$ be a principal $G$-bundle with connection $A_0 \in \mathcal{A}(P_0)$ satisfying~\eqref{equ: cone reduction of instanton equation}. \cite[Section~2.2.1]{Wang-spectrum_of_operator_for_instantons} (see also the proof of \autoref{prop: duality between homogeneous kernels}) shows that the instanton deformation operator $L_{\pr_{S^5}^*A_0}$ can be written as $\tfrac{1}{r}\underline{J}(r\partial_r - P)$, where $P$ is a formally self-adjoint elliptic operator over $S^5$. Similarly, it is well known that the rough Laplacian $\Delta_{\pr_{S^5}^*A_0} \coloneqq \diff_{\pr_{S^5}^*A_0}^* \diff_{\pr_{S^5}^*A_0}$ on $\mathfrak{g}_{\pr_{S^5}^*A_0} \cong \mathbb{R}_{>0}\times\mathfrak{g}_{P_0}$ can be written as \[\Delta_{\pr_{S^5}^*A_0} = \tfrac{1}{r^2} \big( -(r\partial_r)^2 - 4 (r\partial_r) +  \Delta_{A_0} \big),\] where $\Delta_{A_0}$ denotes the Laplacian over $S^5$. Since $\Delta_{A_0}$ is formally self-adjoint and positive, we immediately obtain the following: 
\begin{corollary}
The results of \autoref{prop: condition for operator to have real roots and no log-terms} apply to the operators $L \equiv L_{pr_{S^5}^*A_0}$ and $L \equiv \Delta_{\pr_{S^5}^*A_0}$.
\end{corollary}

\subsection{Conically singular operators}

Let now $(Z,g)$ be a Riemannian 6-manifold, $S \coloneqq \{s_1,\dots,s_N\} \subset Z$ be a finite set, and $\pi \col E \to Z\setminus S$ be a vector bundle together with an inner product $h$ and a metric connection $\nabla \col \Gamma(E) \to \Omega^1(Z\setminus S,E)$. 

\begin{remark}
As already noted in the previous section, the results discussed in this section also hold for general conically singular $n$-manifolds.
\end{remark}

\begin{definition}\label{def: framed cs bundle}
We call $(E,h,\nabla)$ framed conically singular, if for every $s \in S$ we have fixed the following data:
\begin{enumerate}
\item a coordinate system $\Upsilon_s \col B_R(0) \to Z$, where $B_R(0) \subset \mathbb{R}^6$ denotes the open ball, which satisfies $\Upsilon_s(0) = s$ and which pulls back the metric $\Upsilon_s^*g_s$ over $s$ to the standard flat metric $g_0$ on $\mathbb{R}^6$,
\item A conically admissible triple $(\pi_s \col E_s \to \mathbb{R}^6\setminus \{0\}, h_s, \nabla_s)$ in the sense of \autoref{def: conical bundle},
\item an isomorphism $\tilde{\Upsilon}_s \col E_s \to E$ covering $\Upsilon_s$ that satisfies \[ \big\vert \nabla_s^k (\tilde{\Upsilon}_s^* h-h_s) \big\vert_{h_s} = \mathcal{O}(r^{1-k}) \quad \text{and} \quad \big\vert \nabla_s^k (\tilde{\Upsilon}_s^* \nabla-\nabla_s ) \big\vert_{h_s} = \mathcal{O}(r^{-1+\varepsilon-k}) \] for every $k\in \mathbb{N}_0$, where $\varepsilon>0$ is fixed.
\end{enumerate}
Moreover, we call an isomorphism $\tilde{\Upsilon}_s$ as above a framing of $(E,h,\nabla)$ at $s \in S$.
\end{definition}

As in \autoref{def: weighted Hölder norms} we now define the following weighted Hölder spaces:

\begin{definition}
Let $k\in \mathbb{N}_0$, $\alpha \in (0,1)$ be a Hölder coefficient, and $\lambda = (\lambda_1,\dots,\lambda_N) \in \mathbb{R}^N$ be a fixed set of rates. Furthermore, let $\rho \col Z \setminus S \to (0,\infty)$ and $w_\lambda \col Z \to \mathbb{R}$ be the distance and rate functions of \autoref{def: C^infty_mu topology}. For $x,y\in Z\setminus S$ we set $\rho(x,y)\coloneqq \min\{\rho(x),\rho(y)\}$. For any $u \in C^{k,\alpha}_{\textup{loc}}(Z \setminus S,E)$ we define the following weighted Hölder (semi-) norms:
\begin{align*}
[u\WsH{0}{\lambda}{} &\coloneqq \sup_{2\diff(x,y) < \rho(x,y)} \rho(x,y)^{w_{\lambda-\alpha}(x)} \frac{\vert u (x)- u(y)\vert}{\textup{dist}(x,y)^\alpha} \\
\Vert u \VertWH{0}{\lambda}{} & \coloneqq \Vert \rho^{-w_{\lambda}} u \VertC{0}{} + [u\WsH{0}{\lambda}{}  \\
\Vert u \VertWH{k}{\lambda}{} &\coloneqq \sum_{i=0}^k  \Vert \nabla^i u \VertWH{0}{\lambda-i}{},
\end{align*} 
where $\lambda-i \coloneqq (\lambda_1-i,\dots,\lambda_N-i)$ and where all covariant derivatives are taken with respect to $\nabla$ and the Levi--Civita connection on $T^*Z$. To compare $u(x)$ and $u(y)$ which lie over different fibers, we use parallel transport over the shortest geodesic connecting $x$ and $y$. 

With these norms at hand, we now define $C^{k,\alpha}_\lambda(Z\setminus S, E)$ as the normed vector space consisting of all sections $u\in C^{k,\alpha}_{\textup{loc}}(Z \setminus S,E)$, for which $\Vert u \VertWH{k}{\lambda}{}$ is finite, equipped with the norm $\Vert \cdot \VertWH{k}{\lambda}{}$. Similarly, we define the weighted  $C^k_\lambda(Z\setminus S,E)$-space.
\end{definition}

\begin{remark}
Weighted $C^{k,\alpha}_\lambda$ and $C^k_\lambda$-norms for section of the conical bundle $E_s$ over $\mathbb{R}^6\setminus \{0\} \cong (0,\infty) \times S^5$ can be defined analogously. Our assumptions in \autoref{def: framed cs bundle} implies that all weighted $C^{k,\alpha}_\lambda$ and $C^k_\lambda$-norms over the truncated cone $(0,R)\times S^5$ taken with respect to $(h_s,\nabla_s)$ and $\tilde{\Upsilon}_s^*(h,\nabla)$ are equivalent.
\end{remark}

The following results are straight forward extensions of their respective counterparts for unweighted Hölder-spaces:

\begin{proposition}\label{app-prop: weighted Hölder spaces are Banach and compact embedding}
The normed vectorspaces $C^{k,\alpha}_\lambda (Z\setminus S,E)$ and $C^k_\lambda (Z\setminus S,E)$ are complete and therefore Banach. Moreover, for any $k,\ell \in \mathbb{N}$, $\alpha,\beta \in (0,1)$, and $\lambda,\nu \in \mathbb{R}^N$, with $k+\alpha > \ell + \beta$ and $\lambda_i> \nu_i$ for every $i=1,\dots,N$ the natural embedding $C^{k,\alpha}_\lambda \subset C^{\ell,\beta}_\nu$ is compact.
\end{proposition}

\begin{definition}\label{def: cs differential operators}
Let $(\pi \col E \to Z\setminus S, h, \nabla)$ together with $\{(\pi_s \col E_s \to \mathbb{R}^6\setminus \{0\}, h_s, \nabla_s, \tilde{\Upsilon}_s)_{s\in S}\}$ be a framed conically singular bundle. A differential operator $L \col \Gamma(E) \to \Gamma(E)$ of order $\ell$ is called conically singular with asymptotic limits $\{(L_s \col \Gamma(E_s) \to \Gamma(E_s))_{s\in S}\}$ if for every $s\in S$ one of the following two equivalent conditions is satisfied:
\begin{enumerate}
\item Using the connection $\nabla_s$ and the Levi--Civita connection on $\mathbb{R}^6\setminus \{0\}$ we express the differential operators as $\tilde{\Upsilon}_s^*L= \sum_{i=0}^\ell p_i(\nabla_{\mathbb{R}^6}^i)$ and $L_s= \sum_{i=0}^\ell p_i^\infty(\nabla_{\mathbb{R}^6}^i)$, where $p_i,p_i^\infty \in \Gamma(\Hom(\underbrace{T^*\mathbb{R}^6\otimes \dots \times T^*\mathbb{R}^6}_{\text{$i$ times}}\otimes E_s,E_s))$. Then \[ \big\vert \nabla^k_{\mathbb{R}^6}(p_i - p_i^\infty) \big\vert = o(r^{i-\ell-k}) \quad \text{for every $k \in \mathbb{N}_0$.}\]
\item For any $k \in \mathbb{N}$ and $\alpha \in (0,1)$ we have \[ \Vert \tilde{\Upsilon}_s^*L - L_s \Vert_{k,\alpha,r} = o(1) \quad \text{as $r \to 0$} \] where \[ \Vert \tilde{\Upsilon}_s^*L - L_s \Vert_{k,\alpha,r} \coloneqq \sup_{\substack{u \in C^{k+\ell,\alpha}_0 \\ \Vert u \VertWH{k+\ell}{0}{}=1 \\ \textup{supp}(u) \subset (0,r)\times S^5 }} \Vert (\tilde{\Upsilon}_s^*L -L_s)u \VertWH{k}{-\ell}{}. \]
\end{enumerate}
Moreover, we call $L$ conically singular of rate $1+\mu\in \mathbb{R}^N$, where  $\mu_i>-1$ for all $i=1,\dots,N$, if any (and therefore all) of the previous conditions is satisfied with $\mathcal{O}(r^{\dots+(1+\mu_{s_i})})$ instead of $o(r^{\dots})$.\footnote{The appearance of the rate $1+\mu$ with $\mu_i>-1$ (instead of simply $\nu \in \mathbb{R}^N$ with $\nu_i>0$) in this definition is due to the following example.}
\end{definition}

\begin{example}
Let $\pi \col P \to Z \setminus S$ and $A \in \mathcal{A}(P)$ be framed conically singular of rate $\mu \in (-1,0)^{\vert S\vert}$ in the sense of \autoref{def: framed CS connections}. The instanton deformation operator $L_A$ of $A$ as defined prior to \autoref{prop: adjoint of model operator} is conically singular of rate $1+\mu$. 
\end{example}

From now on we assume that $L$ is an elliptic conically singular differential operator of order $\ell$ acting on a fixed framed conically singular bundle $E$, which is asymptotic to the conical differential operators $L_{s}$ for $s \in S$. Note that this implies that all the asymptotic limits $L_s$ are also elliptic. The following proposition follows from the (ordinary) interior Schauder estimate and scaling (see for example \cite[Lemma~12.1]{Pacard-lecture_notes_connected_sums} or \cite[Proposition~1.6]{Bartnik-mass_of_ALF} for proofs of similar statements):

\begin{proposition}\label{app-prop: weighted Schauder estimate}
Let $\lambda \in \mathbb{R}^N$ and $u \in C^0_\lambda(Z\setminus S,E)$ satisfy $Lu \in C^{k,\alpha}_{\lambda-\ell}(Z\setminus S,E)$. Then $u \in C^{k+\ell,\alpha}_\lambda(Z\setminus S,E)$ and there exists a $c>0$ (independent of $u$) such that \[ \Vert u \VertWH{k+\ell}{\lambda}{} \leq c \big( \Vert Lu \VertWH{k}{\lambda-\ell}{} + \Vert u \VertWC{0}{\lambda}{} \big).\]
\end{proposition}

\begin{corollary}\label{cor:kernel of cs-operator increases with weight}
The kernel $\ker(L \col C^{k+\ell,\alpha}_\lambda \to C^{k,\alpha}_{\lambda-\ell})$ is independent of $k\in \mathbb{N}_0$ and $\alpha\in(0,1)$. In the following we will therefore simply denote it by $\ker(L)_\lambda$. Moreover, $\ker(L)_\lambda \subset \ker(L)_\nu$ for any $\lambda,\nu \in \mathbb{R}^N$ with $\lambda_i \geq \nu_i$ for every $i\in \mathbb{N}$.
\end{corollary}

We now define the set of critical rates of $L$ by \[\mathcal{D}(L) \coloneqq\big\{(\lambda_1,\dots,\lambda_N) \in \mathbb{R}^N \ \big\vert \ \lambda_i \in \mathcal{D}(L_{s_i}) \textup{ for some $i \in \{1,\dots,N\}$} \big\},\] where $\mathcal{D}(L_{s_i})$ is as in \autoref{app-def: critical rates and homogeneous kernel}. If the weight $\lambda$ does not lie in $\mathcal{D}(L)$, then the previous proposition can be strengthened as in \cite[Theorem~1.10]{Bartnik-mass_of_ALF} using \cite[Theorem~5.1]{Mazya-weighted_Lp_and_Hölder_estimates} (see also \cite[Proposition~12.2.1]{Pacard-lecture_notes_connected_sums}):

\begin{proposition}\label{app-prop: scalebroken Schauder estimate}
Let $\lambda \in \mathbb{R}^N\setminus \mathcal{D}(L)$ and $u \in C^{k+\ell,\alpha}_\lambda(Z\setminus S,E)$. Then there exists an $\varepsilon>0$ and an open $\varepsilon$-neighbourhood $B_\varepsilon(S)$ of $S$ such that \[ \Vert u \VertWH{k+\ell}{\lambda}{} \leq c \big( \Vert Lu \VertWH{k}{\lambda-\ell}{} + \Vert u \VertWC{0}{}{(Z\setminus B_{\varepsilon}(S))} \big).\]
\end{proposition}
 
This can now be used as in \cite[Section~2 and Section~6]{LockhardMcOwen-ellipticOperators_on_noncompact_mfds} and \cite[Theorem~1.10]{Bartnik-mass_of_ALF} to prove\footnote{In order to show that the cokernel is finite-dimensional, one can, for example, embed the weighted Hölder spaces into weighted Sobolev spaces (of slightly decreased weight) and then use the Fredholm property of $L$ as a map between these Sobolev spaces (as proven in \cite[Theorem~6.1]{LockhardMcOwen-ellipticOperators_on_noncompact_mfds}).}

\begin{proposition}[{\cite[Proposition~2.4]{HaskinsHeinNordstroem--ACylCalabiYaus}}]\label{app-prop: Fredholm away from critical rates}
If $\lambda \in \mathbb{R}^N \setminus \mathcal{D}(L)$, then $L \col C^{k+\ell,\alpha}_\lambda(Z\setminus S,E) \to C^{k,\alpha}_{\lambda-\ell}(Z\setminus S,E)$ is Fredholm for every $k\in \mathbb{N}_0$ and $\alpha \in (0,1)$.
\end{proposition}

The following proposition relates the kernels, cokernels, and the Fredholm indices of the operator $L \col C^{k+\ell,\alpha}_\lambda(Z\setminus S,E) \to C^{k,\alpha}_{\lambda-\ell}(Z\setminus S,E)$ for different uncritical rates $\lambda$. It is the analogue of \cite[Theorem~6.5 and Lemma~7.1]{LockhardMcOwen-ellipticOperators_on_noncompact_mfds}  (or \cite[Proposition~1.14]{Bartnik-mass_of_ALF}) for weighted Hölder spaces and can be proven as for weighted Sobolev spaces.

\begin{proposition}\label{app-prop: ker and coker are locally constrant + index change formula}
The functions 
\begin{align*}
\dim \ker(L) \col \mathbb{R}^N\setminus \mathcal{D}(L) &\to \mathbb{N}_0\\
\dim \coker(L) \col \mathbb{R}^N\setminus \mathcal{D}(L) &\to \mathbb{N}_0 \\
\textup{index}(L)\col \mathbb{R}^N\setminus \mathcal{D}(L) &\to \mathbb{Z}
\end{align*}
that assign to each weight $\lambda$ the respective dimensions of the kernel, the cokernel, and the index $\textup{index}(L)_\lambda \coloneqq \dim \ker(L)_\lambda - \dim \coker(L)_\lambda$ of $L\col C^{k+\ell,\alpha}(Z\setminus S,E)_\lambda \to C^{k,\alpha}_{\lambda-\ell}(Z\setminus S,E)$ are locally constant (and independent of $k\in \mathbb{N}_0$ and $\alpha\in (0,1)$). Moreover, if $\lambda,\nu\in \mathbb{R}^N\setminus \mathcal{D}(L)$ are such that $\lambda_i \geq \nu_i$, then \[\textup{index}(L)_\nu - \textup{index}(L)_\lambda = \sum_{i=1}^N \sum_{\tilde{\nu}_i \in \mathcal{D}(L_{s_i})\cap(\nu_i,\lambda_i)} \hspace{-20pt} \dim \mathcal{K}(L_{s_i})_{\tilde{\nu}_i}\] with $\mathcal{K}(L_{s_i})_{\tilde{\nu}_i}$ as in \autoref{app-def: critical rates and homogeneous kernel}.
\end{proposition}
\begin{remark}
From the formula of the index-change given in the previous proposition for $\lambda,\nu \in \mathbb{R}^N \setminus \mathcal{D}(L)$ with $\lambda_i \geq \nu_i$ for every $i=1,\dots,N$, one can easily deduce the formula for two general $\lambda,\nu \in \mathbb{R}^N\setminus \mathcal{D}(L)$
\end{remark}

Together with \autoref{cor:kernel of cs-operator increases with weight} this implies:

\begin{corollary}\label{app-cor: kernel is locally constant}
Let $\lambda,\nu \in \mathbb{R}^N \setminus \mathcal{D}(L)$ be such that for every $i=1,\dots,N$ we have $\lambda_i \geq \nu_i$ and $[\nu_i,\lambda_i]\cap \mathcal{D}(L_{s_i}) = \emptyset$. Then \[ \ker(L)_\lambda = \ker(L)_\nu.\]
\end{corollary}

The following can be deduced by embedding $C^{k,\alpha}_\lambda$ into a weighted Sobolev space (of slightly decreased weight) and then using \cite[Theorem~4.25]{Marshal-deformations_special_Lagrangians} together with (a slightly strengthened version of) \autoref{app-prop: scalebroken Schauder estimate} (cf. \cite[Proposition~12.2.1]{Pacard-lecture_notes_connected_sums}).

\begin{proposition}\label{app-prop: cokernel kernel pairing}
Let $\lambda \in \mathbb{R}^N\setminus \mathcal{D}(L)$. Since, $\ker(L^*)_{-6-(\lambda-\ell)}=\ker(L^*)_{-6-(\lambda-\ell)+\varepsilon}$ for any sufficiently small $\varepsilon>0$, the natural $L^2$-pairing induces a well-defined map \[ \int \col C^{k,\alpha}_{\lambda-\ell} \otimes \ker(L^*)_{-6-(\lambda-\ell)} \to \mathbb{R}. \] This pairing is non-degenerate on the right and induces therefore a surjective map $C^{k,\alpha}_{\lambda-\ell} \to (\ker(L^*)_{-6-(\lambda-\ell)})^*$. The kernel of the latter map is precisely $\textup{image}(L \col C^{k+\ell,\alpha}_\lambda \to C^{k,\alpha}_{\lambda-\ell})$ and therefore $\coker(L)_\lambda \cong (\ker(L^*)_{-6-(\lambda-\ell)})^*$.
\end{proposition}

The characterisation of $\textup{image}(L\col C^{k+\ell,\alpha}_\lambda \to C^{k,\alpha}_{\lambda-\ell})$ as the annihilator of $\ker(L^*)_{-6-(\lambda-\ell)}$ under the $L^2$-pairing together with the fact that this kernel is locally constant (when varying $\lambda$) implies:

\begin{corollary}\label{app-cor: decay improvement CS operators not crossing rate}
Let $\lambda,\nu \in \mathbb{R}^N\setminus \mathcal{D}(L)$ be such that $\lambda_i \geq \nu_i$ for every $i=1,\dots,N$ and $[\nu_i,\lambda_i]\cap \mathcal{D}(L_{s_i}) = \emptyset$. Assume that $u \in C^{k+\ell,\alpha}_\nu(Z\setminus S,E)$ and $Lu \in C^{k,\alpha}_{\lambda-\ell}(Z\setminus S,E)$. Then $u \in C^{k+\ell,\alpha}_\lambda(Z\setminus S,E)$. 
\end{corollary} 

If $L$ is an elliptic differential operator which is conically singular of rate $1+\mu \in \mathbb{R}^N$ (i.e. it satisfies any of the conditions in \autoref{def: cs differential operators} with right-hand side of $\mathcal{O}(r^{\dots+(1+\mu_i)})$ for $\mu_i>-1$ instead of $o(r^{\dots})$), then one can extend the previous corollary to the situation when one crosses critical rates:

\begin{proposition}[{\cite[Proposition~4.21]{KarigiannisLotay-conifolds}}]\label{app-prop: decay improvement CS operators when crossing rate}
Let $L$ be an elliptic differential operator which is conically singular of rate $1+\mu\in \mathbb{R}^N$ (with $\mu_i>-1$). Let $\lambda,\nu \in \mathbb{R}^N\setminus \mathcal{D}(L)$ be such that $\lambda_i \geq \nu_i$ for every $i=1,\dots,N$ and such that there exists at most one critical rate of $L_{s_i}$ in $[\nu_i,\lambda_i]$. For a critical rate  $\tilde{\nu} \in \mathcal{D}(L)$ with $\nu_i\leq \tilde{\nu}_i \leq \lambda_i$ for every $i$ we then define \[\mathcal{F}_{\nu} \coloneqq \big\{ u \in C^{k+\ell,\alpha}_{\nu}(Z\setminus S,E) \ \big\vert \ Lu \in C^{k,\alpha}_{\lambda-\ell}(Z\setminus S,E) \big\}. \] (In particular, if $u \in \mathcal{F}_{\nu}$, then $Lu$ decays faster / blows up slower than expected.) For every $s_i\in S$ there are linear functions \[ \gamma_{s_i} \col \mathcal{F}_{\nu} \to \mathcal{K}(L_{s_i})_{\tilde{\nu}_i} \quad \text{and} \quad \vartheta_{s_i} \col \mathcal{K}(L_{s_i})_{\tilde{\nu}_i} \to C^\infty_{\tilde{\nu}_i+(1+\mu_i)}(\mathbb{R}^6\setminus \{0\},E_s) \] such that for every $u\in \mathcal{F}_{\nu}$ \[ u - \textstyle{\sum}_i (\tilde{\Upsilon}_{s_i})_* \big( \chi \cdot \big(\gamma_{s_i}(u) - \vartheta_{s_i}(\gamma_{s_i}(u)) \big)\big) \in C^{k,\alpha}_{\lambda}(Z\setminus S,E), \] where $\chi$ is a cut-off function which is 1 for $r<\tfrac{R}{2}$ and 0 for $r>\tfrac{3R}{2}$.
\end{proposition}

We also need the following consequence of the previous proposition:

\begin{proposition}[{\cite[Corollary~4.22]{KarigiannisLotay-conifolds}}]\label{app-prop: cs operators asymptotic expansions of kernel elements}
Let $L$ be an elliptic differential operator which is conically singular of rate $1+\mu\in \mathbb{R}^N$ (with $\mu_i>-1$). For any fixed $\nu \in \mathbb{R}^N\setminus \mathcal{D}(L)$ and $s_i \in S$ let $\tilde{\nu}_{i,1},\tilde{\nu}_{i,2} \in \mathcal{D}(L_{s_i})$ be the first two elements in $\mathcal{D}(L_{s_i})$ satisfying $\nu_i < \tilde{\nu}_{i,1}< \tilde{\nu}_{i,2}$. Then there exist linear functions 
\begin{align*}
\gamma_{s_i,1} \col \ker(L)_\nu &\to \mathcal{K}(L_{s_i})_{\tilde{\nu}_{i,1}}, \qquad \qquad
\gamma_{s_i,2} \col \ker(L)_\nu \to \mathcal{K}(L_{s_i})_{\tilde{\nu}_{i,2}}
\end{align*}
and \[\eta_{s_i} \col \mathcal{K}(L_{s_i})_{\tilde{\nu}_{i,1}} \to C^\infty_{\tilde{\nu}_{i,1}+(1+\mu_i)}(\mathbb{R}^6\setminus \{0\}, E_{s_i})\] such that \[ \big\vert \tilde{\Upsilon}_{s_i}^*u - \gamma_{s_i,1}(u) - \eta_{s_i}(\gamma_{s_i,1}(u)) - \gamma_{s_i,2}(u) \big\vert = \mathcal{O}(r^{\tilde{\nu}_{i,2}+\varepsilon}) \quad \text{as $r \to 0$} \] for every $u \in \ker(L)_\nu$ and some $\varepsilon>0$.
\end{proposition}
\begin{remark}
Note that in the previous proposition $\gamma_{s_i,1}(u) \in \mathcal{K}(L_{s_i})_{\tilde{\nu}_{i,1}}$ gives the leading order contribution of $u \in \ker(L)_\nu$ close to $s_i \in S$. Similarly, $\gamma_{s_i,2}(u)\in \mathcal{K}(L_{s_i})_{\tilde{\nu}_{i,2}}$ gives the leading order contribution of $\tilde{\Upsilon}_{s_i}^*u -\gamma_{s_i,1}(u) - \eta_{s_i}(\gamma_{s_i,1}(u))$ and therefore depends on the construction of $\eta_{s_i}$ (which is neither unique nor canonical). Note however that if $\tilde{\nu}_{i,1}+(1+\mu_i) > \tilde{\nu}_{i,2}$, then the leading order contribution of $\tilde{\Upsilon}_{s_i}^*u -\gamma_{s_i,1}(u) - \eta_{s_i}(\gamma_{s_i,1}(u))$ depends only on $u$. That is, in this situation $\gamma_{s_i,2}$ is defined independently of $\eta_{s_i}$.
\end{remark}
\begin{remark}
This proposition can be generalised to obtain contributions $\gamma_{s_i,1}(u),\dots,\gamma_{s_i,K}(u)$ of the first $K$ indicial roots $\nu_i < \tilde{\nu}_{i,1} < \dots < \tilde{\nu}_{i,K}$ for any $K\in \mathbb{N}$. As in the previous remark, if $\tilde{\nu}_{i,K}-\tilde{\nu}_{i,1} < 1+\mu_i$ holds, then these depend only on $u \in \ker(L)_\nu$ and not on the choices leading to the construction of the corresponding functions $\eta_{s_i,j}$ for $j=1,\dots,K$.
\end{remark}

We end this section with the following proposition which is stated in the exact form needed for the proof of \autoref{thm: non-degenerateness of pairing}.

\begin{proposition}\label{app-prop: for every expansion there is a kernel element}
Assume that $L$ is an elliptic differential operator which is conically singular of rate $1+\mu\in \mathbb{R}^N$ (with $\mu_i>-1$) and that $\lambda \in \mathbb{R}^N \setminus \mathcal{D}(L)$ is such that $\dim \ker(L)_{\lambda} = 0 = \dim \coker(L)_{\lambda}$ (and hence $\textup{index}(L)_{\lambda} = 0$). Let $\nu \in \mathbb{R}^N\setminus \mathcal{D}(L)$ be such that for every $i=1,\dots,N$ there are precisely two indicial roots $\tilde{\nu}_{i,1}, \tilde{\nu}_{i,2}\in \mathcal{D}(L_{s_i})$ contained in $[\nu_i,\lambda_i]$. We assume that these are ordered as $\nu_i < \tilde{\nu}_{i,1}< \tilde{\nu}_{i,2} < \lambda_i$. For any collection of $u_{s_i,1} \in \mathcal{K}(L_{s_i})_{\tilde{\nu}_{i,1}}$ and $u_{s_i,2} \in \mathcal{K}(L_{s_i})_{\tilde{\nu}_{i,2}}$ for $i=1,\dots,N$ there exists a $u \in \ker(L)_\lambda$ with $\gamma_{s_i,j}(u)=u_{s_i,j}$ for $j=1,2$ and $i=1,\dots,N$, where $\gamma_{s_i,j}$ are the functions appearing in the previous proposition.
\end{proposition}
\begin{proof}
The formula for the index change given in \autoref{app-prop: ker and coker are locally constrant + index change formula} together with the observation that $\dim \coker(L)$ is non-increasing when going from $\lambda$ to $\nu< \lambda$ implies that \[\dim \ker(L)_\nu = \sum_{i=1}^N \dim K(L_{s_i})_{\tilde{\nu}_{i,1}} + \dim K(L_{s_i})_{\tilde{\nu}_{i,2}}. \] Thus \[\oplus_i (\gamma_{s_i,1},\gamma_{s_i,2}) \col \ker(L)_\nu \to \oplus_i \big(\mathcal{K}(L_{s_i})_{\tilde{\nu}_{i,1}} \oplus \mathcal{K}(L_{s_i})_{\tilde{\nu}_{i,2}}\big)\] is a linear map between vector spaces of the same dimension. A moment's thought reveals that it is injective, hence an isomorphism.
\end{proof}

\bibliography{references}{}
\bibliographystyle{alpha}

\end{document}